\documentclass[11pt]{amsart}
\usepackage[utf8]{inputenc}
\usepackage[english]{babel}

\usepackage{setspace}
\usepackage{hyperref}
\usepackage{amssymb}
\usepackage{amsmath}
\usepackage{amsthm}
\usepackage{stmaryrd}
\usepackage{graphicx}
\usepackage{mathrsfs}
\usepackage[noabbrev]{cleveref}

\usepackage{halloweenmath}

\usepackage{mathtools}

\usepackage{relsize}

\usepackage{imakeidx}
\makeindex

\usepackage{quiver}

\usepackage[textwidth=3cm, textsize=small, colorinlistoftodos]{todonotes}

\usepackage{xcolor}
\hypersetup{
    colorlinks=true,
    linkcolor=blue,
    filecolor=magenta,      
    urlcolor=cyan,
    citecolor=blue,
}
\usepackage{enumerate}

\input xy
\xyoption{all}
\usepackage[color,matrix,arrow]{xy}
\usepackage{tikz}
\usepackage{tikz-cd}
\usetikzlibrary{bending}
\usetikzlibrary{arrows.meta}
\usetikzlibrary{matrix,arrows,decorations.markings, calc, positioning}
\usetikzlibrary {shapes.geometric}
\usetikzlibrary {intersections}

\usepackage{geometry}
\makeatletter
\providecommand{\leftsquigarrow}{%
  \mathrel{\mathpalette\reflect@squig\relax}%
}
\newcommand{\reflect@squig}[2]{%
  \reflectbox{$\m@th#1\rightsquigarrow$}%
}
\makeatother

\let\P=\relax

\let\epsilon=\relax

\let\th=\relax
\let\SS=\relax

\let\mod=\relax

\DeclareMathOperator{\THH}{THH}

\DeclareMathOperator{\HH}{HH}

\DeclareMathOperator{\Tor}{Tor}
\DeclareMathOperator{\Hom}{Hom}

\DeclareMathOperator{\mod}{\, \, (\text{mod} \,}

\newcommand{\CC}{\mathbb{C}}
\newcommand{\FF}{\mathbb{F}}

\newcommand{\RR}{\mathbb{R}}
\newcommand{\ZZ}{\mathbb{Z}}

\newcommand{\PP}{\mathbb{P}}

\newcommand{\SS}{\mathbb{S}}

\newcommand{\wedgey}{\vee}
\newcommand{\smashy}{\wedge}

\newcommand{\id}{\text{id}}

\newcommand{\th}{\text{th}}

\newcommand{\epsilon}{\varepsilon}

\newcommand{\Ab}{\mathcal{A}\emph{b}}

\newcommand{\AC}{\mathcal{A}^{C_p}}
\newcommand{\ACop}{(\mathcal{A}^{C_p})^{\text{op}}}

\newcommand{\P}{\mathcal{P}}

\newcommand{\I}{\mathcal{I}}

\newcommand{\op}{^\text{op}}

\newcommand{\uR}{\underline{R}}
\newcommand{\uF}{\underline{F}}
\newcommand{\uM}{\underline{M}}
\newcommand{\uN}{\underline{N}}
\newcommand{\uL}{\underline{L}}

\newcommand{\uI}{\underline{I}}
\newcommand{\uA}{\underline{A}}

\makeatletter
\newcommand*{\boxwedge}{\mathbin{\mathpalette\@boxwedge{}}}
\newcommand*{\@boxwedge}[2]{%
  \sbox0{$#1\boxtimes\m@th$}%
  \dimen2=.5\dimexpr\wd0-\ht0-\dp0\relax % side bearing
  \dimen@=\dimexpr\ht0+\dp0\relax
  \def\lw{.07} % line width as factor for height of \boxtimes
  \kern\dimen2 % side bearing
  \tikz[
    line width=\lw\dimen@,
    line join=round,
    x=\dimen@,
    y=\dimen@,
    baseline=0
  ]
  \draw
    (\lw/2,0) rectangle (1-\lw,1-\lw)
    (\lw,0) -- (.5,1-\lw-\lw/2) -- (1-\lw-\lw/2,0)
  ;%
  \kern\dimen2 % side bearing
}
\makeatother

\theoremstyle{plain}
\newtheorem{theorem}{Theorem}[section]
\newtheorem{lemma}[theorem]{Lemma}
\newtheorem{proposition}[theorem]{Proposition}
\newtheorem{corollary}[theorem]{Corollary}

\theoremstyle{definition}
\newtheorem{definition}[theorem]{Definition}

\newtheorem{example}[theorem]{Example}
\newtheorem{remark}[theorem]{Remark}

\title{The algebraic structure of twisted topological Hochschild homology}
\author{Danika Van Niel}

\begin{document}

\begin{abstract}
Topological Hochschild homology (THH) is an invariant of ring spectra developed by B\"okstedt. In recent years many equivariant analogues to THH have emerged. One example is twisted THH which is an invariant of $C_n$-equivariant ring spectra developed by Angeltveit, Blumberg, Gerhardt, Hill, Lawson, and Mandell. In this paper, we study the algebraic structure of twisted THH, and perform some computations. Specifically, we compute $C_2$-twisted THH of the Real bordism spectrum and show that the $C_p$-twisted THH of geometric ring $C_p$-spectra reduces to a computation of classical THH. We extend the algebraic structure of twisted THH to the twisted B\"okstedt spectral sequence of Adamyk, Gerhardt, Hess, Klang, and Kong. We show that, under appropriate flatness conditions and for $R$ a commutative ring $C_p$-spectrum, the $C_p$-twisted B\"okstedt spectral sequence is a spectral sequence of commutative $\underline{E}_\star(R)$-algebras.
\end{abstract}

\maketitle

\section{Introduction}\label{Sec:Intro}

Algebraic K-theory is an invariant of rings that is famously difficult to compute. While algebraic K-theory is defined using topology, it has connections to many fields of mathematics, including number theory and algebraic geometry. One very successful approach to the study of algebraic K-theory is trace methods. Trace methods are tools that allow us to approximate algebraic K-theory by mapping from algebraic K-theory to more computable invariants. One example of such an approximation, the Dennis trace, relates algebraic K-theory to a classical invariant of rings, Hochschild homology (HH). For an associative ring $A$, the Dennis trace is a map:
\[
K_*(A) \to \HH_*(A).
\]
\noindent For a closer approximation to algebraic K-theory, B\"okstedt defined a topological analogue of HH called topological Hochschild homology (THH), which is an invariant of ring spectra \cite{BokstedtTHHComputations}. There is a trace map, the topological Dennis trace, from algebraic K-theory to THH. Topological Hochschild homology has an $S^1$-action. Using this $S^1$-action one can define topological cyclic homology (TC) and the cyclotomic trace, which gives an even more accurate approximation to algebraic K-theory \cite{Bokstedt-Hsiang-Madsen(Spaces)}. Further, the topological Dennis trace factors through the cyclotomic trace:
\[
K(A) \to \text{TC}(A) \to \THH(A)
\]
\noindent for an associative ring spectrum $A$.

One of the main tools we use to compute THH is the B\"okstedt spectral sequence which relates HH to THH \cite{BokstedtTHHComputations}. For $k$ a field and $A$ an associative ring spectrum this spectral sequence takes the form:
\[
E^2_{*,*} = \HH_*(H_*(A;k)) \Rightarrow H_*(\THH(A);k).
\]
\noindent One way to facilitate spectral sequence calculations is to understand algebraic structures in the spectral sequence. Angeltveit and Rognes study the algebraic structure of the B\"okstedt spectral sequence in \cite{Angeltveit-Rognes}, building off of results of \cite{EKMM} and \cite{MSV}. Let us recall that a Hopf algebra can be thought of as both an algebra and a coalgebra with an antipode such that these structures are compatible. For $A$ a commutative ring spectrum, the authors of \cite{EKMM} use maps on the circle to induce the necessary maps on THH to prove $\THH(A)$ is an $A$-Hopf algebra in the stable homotopy category. This is possible because for $A$ a commutative ring spectrum, $\THH(A) \cong A \otimes S^1$ \cite{MSV}. For example, the fold map of spaces $S^1 \wedgey S^1 \to S^1$ induces the product map $\THH(A) \smashy_A \THH(A) \to \THH(A)$. Angeltveit and Rognes extend this result by using simplicial maps on the circle, allowing the algebraic structure to extend to the B\"okstedt spectral sequence under certain flatness conditions \cite{Angeltveit-Rognes}. These authors then use this algebraic structure to facilitate many computations of THH.

As stated above, THH has an $S^1$-action and is an invariant of ring spectra. Equivariant norms $N_H^G$ for $G$ a finite group and $H \leq G$ were used in the proof of the Kervaire invariant one problem in \cite{HHR}. Building off of these equivariant norms, the authors of \cite{twTHH} show that for an associative ring spectrum $A$, $\THH(A)$ can be written as the equivariant norm $N_e^{S^1}A$. When $A$ is a commutative ring spectrum this recovers the tensor homeomorphism $N_e^{S^1}A \cong A \otimes S^1$ \cite{EKMM,MSV}.

In recent years, equivariant analogues of algebraic K-theory and THH have emerged, meaning analogues to THH which are invariants of equivariant spectra. For example, Angeltveit, Blumberg, Gerhardt, Hill, Lawson, and Mandell construct a generalization of THH called $C_n$-twisted THH, for $C_n$ a finite cyclic subgroup of $S^1$, denoted $\THH_{C_n}$ \cite{twTHH}. This generalized theory is an invariant of $C_n$-equivariant ring spectra. Building off the equivariant norms of \cite{HHR}, the authors of \cite{twTHH} show that for an associative ring $C_n$-spectrum $R$, $\THH_{C_n}(R)$ can be written as the equivariant norm $N_{C_n}^{S^1}R$. Twisted THH is related to the equivariant algebraic K-theory of Merling \cite{EquivariantKTheory}, and Malkiewich--Merling \cite{ATheory}, as seen in \cite{AGHKKShadows}. 

To compute the equivariant homology of $C_n$-twisted THH, Adamyk, Gerhardt, Hess, Klang, and Kong construct an equivariant analogue to the B\"okstedt spectral sequence \cite{AGHKK}. The twisted B\"okstedt spectral sequence has coefficients in an equivariant analogue to rings.

To study equivariant homotopy theory, one needs equivariant analogues of familiar algebraic objects. Mackey functors arise naturally in equivariant homotopy theory as the equivariant analogue to abelian groups. Let $G$ be a finite abelian group. For a $G$-equivariant spectrum $E$, the equivariant homotopy groups of $E$ form a $G$-Mackey functor. The category of $G$-Mackey functors has a symmetric monoidal product called the box product, denoted $\square$, allowing one to define an equivariant analogue to rings, called $G$-Green functors \cite{Lewis}. Lewis defines $G$-Mackey fields to be commutative $G$-Green functors with no nontrivial ideals \cite{Lewis}. In this paper, we develop the notion of graded $G$-Mackey fields. We will show that if we use $C_p$-Mackey fields as the coefficients in the twisted B\"okstedt spectral sequence, the spectral sequence is easier to compute. We indicate that an object is a Mackey functor with an underline, the symbol $*$ indicates a $\ZZ$-grading, and the symbol $\star$ indicates an $RO(G)$-grading for the given $G$. In this paper, we compute $H\uF_\star$ for all $C_p$-Mackey fields $\uF$ where $H\uF$ is the Eilenberg-Mac Lane spectrum of $\uF$. 

An important $C_2$-spectrum is $MU_\RR$, the Real bordism spectrum. Hill, Hopkins, and Ravenel use $MU_\RR$ in their solution of the Kervaire invariant one problem in \cite{HHR}. In this paper, we compute the equivariant homology of $\THH_{C_2}(MU_\RR)$ with coefficients in the $C_2$-Mackey field $\uF$, such that $\uF(C_2/C_2) = \FF_2$ and $\uF(C_2/e) = 0$. 

\begin{theorem} For $\uF$ as above the $RO(C_2)$-graded equivariant homology of $\THH_{C_2}(MU_\RR)$ with coefficients in $\uF$ is

\begin{center}
    $\underline{H}_\star(\THH_{C_2}(MU_\RR) ; \uF) \cong H\uF_\star [\beta_1, \beta_2, \ldots] \square_{H\uF_{\star}} \mathlarger{\mathlarger{\Lambda}}_{H\uF_{\star}}(z_1, z_2, \ldots) $
\end{center}
\noindent as an $H\uF_{\star}$-module. Here $|\beta_i| = i\rho$ and $|z_i| = 1 + i\rho$.
\end{theorem}

Classically, the algebraic structure in the B\"okstedt spectral sequence has lead to computations of THH. In the current work, we study the algebraic structures of $C_p$-twisted THH and the twisted B\"okstedt spectral sequence. Angeltveit, Blumberg, Gerhardt, Hill, Lawson, and Mandell show in \cite{twTHH} that for a commutative $C_n$-ring spectrum $R$, $\THH_{C_n}(R)$ is $R \otimes_{C_n} S^1$. Using equivariant simplicial models of the circle, we demonstrate that $C_p$-twisted THH has the structure of an $R$-algebra.

\begin{theorem}
     For $p$ prime and $R$ a commutative $C_p$-ring spectrum, $\THH_{C_p}(R)$ is a commutative $R$-algebra in the category of $C_p$-spectra.
\end{theorem}

Using the equivariant simplicial maps that provide the algebraic structure on $C_p$-twisted THH, we extend this algebraic structure to the twisted B\"okstedt spectral sequence. Before we discuss this induced structure, recall that the authors of \cite{AGHKK} show that the twisted B\"okstedt spectral sequence is a spectral sequence of $\underline{E}_\star$-algebras. In the current work we show that for $R$ a commutative ring $C_p$-spectrum and under certain flatness conditions, this spectral sequence is a spectral sequence of $\underline{E}_\star(R)$-algebras.

\begin{theorem}
For a prime $p$, let $R$ and $E$ be commutative ring $C_p$-spectra, such that the chosen generator of $C_p$ acts trivially on $E$ and $\underline{E}_\star(R)$ is flat over $\underline{E}_\star$. The twisted B\"okstedt spectral sequence is a spectral sequence of commutative $\underline{E}_\star(R)$-algebras.
\end{theorem}

We then leverage known computations of THH from \cite{BCSTHH(Thom)} to make new computations of twisted THH.

\begin{proposition}
    Let $p$ and $q$ be distinct primes, and $\uF$ the $C_p$-Mackey field such that $\uF(C_p/C_p) = \FF_q$ and $\uF(C_p/e) = 0$. The following are equivalences of spectra:
    \begin{align*}
        \Phi^{C_p}(\THH_{C_p}&(H\uF))  \simeq H\FF_q \smashy \Omega(S^3)_+, \quad \text{and} \\
        \Phi^{C_2}(\THH_{C_2}&(MU_\RR))  \simeq MO \smashy BBO_+.
    \end{align*}
\end{proposition}

We then find a formula for $\underline{E}_\star(\THH_{C_p}(A))$ for $A$ any cofibrant, associative ring $C_p$-spectrum, and $E$ any geometric $C_p$-spectrum. One can use the computations above with this formula to prove the following.

\begin{proposition}
    Let $A$ be a cofibrant, associative ring $C_p$-spectrum, $p$ prime. If $E$ is a geometric $C_p$-spectrum, then
   \begin{align*}
    \underline{E}_\star(\THH_{C_p}(A))&(C_p/C_p) \cong \Phi^{C_p}(E)_{\dim(\star^{C_p})}(\Phi^{C_p}\THH_{C_p}(A)), \quad \text{and} \\
    \underline{E}_\star(\THH_{C_p}(A))&(C_p/e) \cong 0.
    \end{align*}
\end{proposition}

\subsection{Organization}

In \Cref{Sec:Mackey Functors} we recall the definitions and properties of Mackey functors, Green functors, Mackey fields, and develop the theory of $\ZZ$-graded and $RO(C_p)$-graded Mackey fields. We end the section by computing the $RO(C_p)$-graded homotopy groups of the Eilenberg-Mac Lane spectra of $C_p$-Mackey fields. In \Cref{Sec:Hochschild Homology} we recall the constructions of Hochschild homology for Green functors, and twisted topological Hochschild homology (THH). \Cref{A Computation} has a computation of the equivariant homology of $C_2$-twisted THH of the Real bordism spectrum.

In Sections \ref{Sec:Algebraic structure} and \ref{Sec: Equivariant Bokstedt SS Structure} we study the algebraic structure of $C_p$-twisted THH and the $C_p$-twisted B\"okstedt spectral sequence respectively. \Cref{More Computations} has more computations of twisted THH, specifically $\THH_{C_p}(H\uF)$, $\THH_{C_2}(MU_\RR)$, and $\underline{E}_\star(\THH_{C_p}(A))$ for $\uF$ a $C_p$-Mackey field such that $\uF(C_p/C_p) = \FF_q$ and $\uF(C_p/e) = 0$ for $p$ and $q$ different primes, $E$ a geometric $C_p$-spectrum, and $A$ an associative ring $C_p$-spectrum.

\subsection{Notation and conventions} 

Throughout this paper let $G$ be a finite abelian group. We are working with genuine orthogonal $G$-spectra indexed on a complete universe. We use $*$ to denote $\ZZ$-gradings, $\star$ to denote $RO(G)$-gradings, and $\bullet$ to denote simplicial gradings. Whenever discussing rotations, we mean counter clockwise rotations.

\subsection{Acknowledgments}

The author thanks Teena Gerhardt for all of her support and guidance on this project. The author also thanks Anna Marie Bohmann, David Chan, Michael Hill, Chloe Lewis, Cary Malkiewich, and Maximilien P\'eroux for many useful conversations. The author is thankful for suggested edits from Teena Gerhardt, Michael Hill, and Cary Malkiewich.

The author has been partially supported by NSF Grants DMS-1810575, DMS-2052042, DMS-2104233, DMS-RTG 2135960, and DMS-RTG 2135884, as well as a Dissertation Continuation Fellowship and Dissertation Completion Fellowship from Michigan State University. This material is in part based on work supported by the National Science Foundation under DMS-1928930, during a visit to the Simons Laufer Mathematical Sciences Institute (previously known as MSRI) in Berkeley, California, during the Fall 2022 semester.

\section{\texorpdfstring{$C_p$}{Cp}-Mackey functors}\label{Sec:Mackey Functors}

Let $p$ be a prime. In this section we will introduce $C_p$-Mackey functors. There is a more general concept of a $G$-Mackey functor for any finite group $G$, but for the purposes of this paper we will focus on $G = C_p$.

Let $S,T,$ and $U$ be finite $C_p$-sets. A \emph{span} from $S$ to $T$ is a diagram
% https://q.uiver.app/#q=WzAsMyxbMCwwLCJTIl0sWzEsMCwiVSJdLFsyLDAsIlQiXSxbMSwwXSxbMSwyXV0=
\[\begin{tikzcd}
	S & U & T
	\arrow[from=1-2, to=1-1]
	\arrow[from=1-2, to=1-3]
\end{tikzcd}\]
\noindent where the maps are $C_p$-equivariant. An \emph{isomorphism of spans} is a commutative diagram of finite $G$-sets
% https://q.uiver.app/#q=WzAsNCxbMSwwLCJVIl0sWzIsMSwiVCJdLFswLDEsIlMiXSxbMSwyLCJWIl0sWzAsMywiXFxjb25nIl0sWzAsMl0sWzMsMl0sWzMsMV0sWzAsMV1d
\[\begin{tikzcd}[column sep=small,row sep=small]
	& U \\
	S && T \\
	& V
	\arrow["\cong", from=1-2, to=3-2]
	\arrow[from=1-2, to=2-1]
	\arrow[from=3-2, to=2-1]
	\arrow[from=3-2, to=2-3]
	\arrow[from=1-2, to=2-3]
\end{tikzcd}\]
\noindent and the composition of spans is given by the pullback. Given two spans $S \leftarrow U_1 \to T$ and $S \leftarrow U_2 \to T$, there is a monoidal product via the disjoint union, $S \leftarrow U_1 \bigsqcup \, U_2  \to T$.

\begin{definition}
    The \emph{Burnside category} of $C_p$, denoted $\AC$, has as objects finite $C_p$-sets. The morphism set $\AC(S,T)$ is the group completion of the monoid of isomorphism classes of spans $S \leftarrow U \to T$.  
\end{definition}

This category allows us to define Mackey functors.

\begin{definition}
    A $C_p$\emph{-Mackey functor} is an additive functor $\underline{M}\colon \ACop \to \Ab$ that sends disjoint unions to direct sums.
\end{definition}

This definition, while simple, is not very explicit. We will now build Mackey functors an equivalent way which the reader may find to be more intuitive.

Let $Fin^{C_p}$ be the category of finite $C_p$-sets. A $C_p$-Mackey functor $\underline{M}$ is equivalent to a pair of additive functors 
\[
M_*, M^*\colon Fin^{C_p} \to \Ab
\]
\noindent which send disjoint unions to direct sums, and are covariant and contravariant, respectively. Further, for any $S \in Fin^{C_p}$, $M_*(S) = M^*(S)$, denoted $\underline{M}(S)$. Recall that any finite $C_p$-set is isomorphic to $\bigsqcup_{\,i} C_p/C_p \sqcup \bigsqcup_{\,j} C_p/e$. Since Mackey functors are additive, one only needs to know $\underline{M}(C_p/C_p)$ and $\uM(C_p/e)$ to know $\underline{M}(S)$ for any finite $C_p$-set $S$. Additionally it is required that if the following diagram is a pullback in $Fin^{C_p}$, then the second diagram commutes:
% https://q.uiver.app/#q=WzAsOCxbMCw0LCJcXHVuZGVybGluZXtNfShXKSJdLFsyLDQsIlxcdW5kZXJsaW5le019KFgpIl0sWzAsNiwiXFx1bmRlcmxpbmV7TX0oWikiXSxbMiw2LCJcXHVuZGVybGluZXtNfShZKSJdLFswLDAsIlciXSxbMiwwLCJYIl0sWzAsMiwiWiJdLFsyLDIsIlkiXSxbMCwxLCJNXyooZycpIl0sWzIsMCwiTV4qKGYnKSJdLFsyLDMsIk1fKihnKSIsMl0sWzMsMSwiTV4qKGYpIiwyXSxbNiw3LCJnIiwyXSxbNCw2LCJmJyIsMl0sWzQsNSwiZyciXSxbNSw3LCJmIl1d
\[\begin{tikzcd}[column sep=small,row sep=small]
	W && X \\
	\\
	Z && Y \\
	{\underline{M}(W)} && {\underline{M}(X)} \\
	\\
	{\underline{M}(Z)} && {\underline{M}(Y).}
	\arrow["{M_*(g')}", from=4-1, to=4-3]
	\arrow["{M^*(f')}", from=6-1, to=4-1]
	\arrow["{M_*(g)}"', from=6-1, to=6-3]
	\arrow["{M^*(f)}"', from=6-3, to=4-3]
	\arrow["g"', from=3-1, to=3-3]
	\arrow["{f'}"', from=1-1, to=3-1]
	\arrow["{g'}", from=1-1, to=1-3]
	\arrow["f", from=1-3, to=3-3]
\end{tikzcd}\]

There is a natural surjection $q\colon C_p/e \to C_p/C_p$. The homomorphism $M_*(q)\colon \underline{M}(C_p/e) \to \underline{M}(C_p/C_p)$ is called the \emph{transfer map}, denoted $tr$. The homomorphism $M^*(q)\colon \underline{M}(C_p/C_p) \to \underline{M}(C_p/e)$ is called the \emph{restriction map}, denoted $res$. 

Each level $\uM(C_p/H)$ naturally has a group action given by the set of $C_p$-maps of $C_p/H$ into itself. This set of maps is isomorphic to the \emph{Weyl group} with respect to $H$, which is defined as $WH \coloneqq N_{C_p}(H)/H$ where $N_{C_p}(H)$ denotes the normalizer of $H$ in $C_p$. 

\begin{remark} \label{MackeyGroupAction}

Let $\uM$ be a $C_p$-Mackey functor. Since $C_p$ is abelian $WH = C_p/H$ for $H \leq C_p$, meaning the abelian group $\underline{M}(C_p/e)$ has a $C_p$-action and $\uM(C_p/C_p)$ has an action of the trivial group. Thus $\uM(C_p/C_p)$ and $\uM(C_p/e)$ are $\ZZ[C_p]$-modules and the transfer and restriction maps are maps of $\ZZ[C_p]$-modules. Therefore $\uM$ admits a $C_p$-action.
\end{remark}

A \emph{Lewis diagram} is a succinct way to describe a Mackey functor. In this paper we will write the Lewis diagram of a $C_p$-Mackey functor, say $\uM$, as the following:
% https://q.uiver.app/#q=WzAsMixbMCwwLCJcXHVNKENfMi9DXzIpIl0sWzAsMiwiXFx1TShDXzIvZSkiXSxbMCwxLCJyZXMiLDIseyJjdXJ2ZSI6Mn1dLFsxLDAsInRyIiwyLHsiY3VydmUiOjJ9XV0=
\[\begin{tikzcd}[column sep=small,row sep=small]
	{\uM(C_p/C_p)} \\
	\\
	{\uM(C_p/e)}
	\arrow["res"', curve={height=12pt}, from=1-1, to=3-1]
	\arrow["tr"', curve={height=12pt}, from=3-1, to=1-1]
\end{tikzcd}\]
note that we are not indicating the Weyl action in the diagram.

\begin{definition}
 A \emph{map between $C_p$-Mackey functors} $\phi\colon \underline{M} \to \underline{N}$ is a natural transformation. This map is defined by $WH$-equivariant group homomorphisms $\phi_H\colon \underline{M}(C_p/H) \to \underline{N}(C_p/H)$ for all $H \leq C_p$. Further, these homomorphisms must respect the transfer and restriction maps. 
\end{definition}

 Let us consider some important examples of $C_p$-Mackey functors. Note that these definitions can all be extended to $G$-Mackey functors for $G$ a finite group, but we will focus on $C_p$-Mackey functors.

\begin{example}
    Let $L$ be an abelian group. The \emph{constant $C_p$-Mackey functor} over $L$, denoted $\underline{L}$, is a Mackey functor where $\underline{L}(C_p/H) = L$ with a trivial $WH$-action, the restriction is the identity, and the transfer is multiplication by $p$.
\end{example}

For the next example, we first need the following definition.

\begin{definition}
    The \emph{Burnside ring of $G$}, denoted $A(G)$, is the group completion of the monoid of isomorphism classes of finite $G$-sets under disjoint union. Multiplication in this ring is given by the Cartesian product of finite $G$-sets.
\end{definition}

\begin{example} \label{BurnsideMackeyFunctor}
    The \emph{$C_p$-Burnside Mackey functor} is denoted $\underline{A}_{C_p}$, or $\underline{A}$ when $C_p$ is clear from context. This functor is defined by $\underline{A}(C_p/e) = A(e) = \ZZ$, $\underline{A}(C_p/C_p) = A(C_p) = \ZZ \oplus \ZZ$, the transfer and restriction maps are given by induction and restriction maps on finite sets. More explicitly, $tr(x) = (0,x)$ and $res(y,z) = y + pz$ for $x,y,z \in \ZZ$.
    
\end{example}

Lewis defines an important example of a $C_p$-Mackey functor in \cite[Definition 5.5]{Lewis} called the J-Mackey functor. We will use these functors in our computations of equivariant homology and cohomology in \Cref{Sec: Homotopy Groups}.

\begin{definition}[{\cite[Definition 5.5]{Lewis}}] \label{J-Mackey Functor Example} 
    Let $H \leq C_p$, $p$ prime. The functor $J_{C_p/H}(V)$ for $V$ a $\ZZ[WH]$-module is the following depending on $H$:
% https://q.uiver.app/#q=WzAsNixbMCwwLCJKX3tDX3AvZX0oVilcXGNvbG9uIl0sWzIsMCwiVl57Q19wfSJdLFsyLDIsIlYiXSxbNCwwLCJKX3tDX3AvQ19wfShWKVxcY29sb24iXSxbNiwwLCJWIl0sWzYsMiwiMCJdLFsxLDIsImluYyIsMix7ImN1cnZlIjoyfV0sWzIsMSwidHIiLDIseyJjdXJ2ZSI6Mn1dLFs0LDUsIjAiLDIseyJjdXJ2ZSI6Mn1dLFs1LDQsIjAiLDIseyJjdXJ2ZSI6Mn1dXQ==
\[\begin{tikzcd}[column sep=small,row sep=small]
	{J_{C_p/e}(V)\colon} && {V^{C_p}} && {J_{C_p/C_p}(V)\colon} && V \\
	\\
	&& V &&&& 0
	\arrow["inc"', curve={height=12pt}, from=1-3, to=3-3]
	\arrow["0"', curve={height=12pt}, from=1-7, to=3-7]
	\arrow["tr"', curve={height=12pt}, from=3-3, to=1-3]
	\arrow["0"', curve={height=12pt}, from=3-7, to=1-7]
\end{tikzcd}\]
\noindent where $tr(x) = \sum_{i = 0}^{p-1} \gamma^i x$ for $x \in V$ and $\gamma$ a generator of $C_p$.
\end{definition}

Mackey functors are the equivariant analogue of abelian groups. In classical homotopy theory many invariants take values in abelian groups. The equivariant analogues of those invariants take values in Mackey functors. The following is an important example of this phenomenon.

\begin{example}
    Let $E$ be a $C_p$-spectrum. For each $n \in \ZZ$, the equivariant homotopy groups of $E$ assemble to a $G$-Mackey functor, denoted $\underline{\pi}_n(E)$, defined by

\begin{center}
$\underline{\pi}_n(E)(C_p/H) \coloneqq \pi_n(E^H)$.
\end{center}    
\end{example}

For a $C_p$-Mackey functor $\uM$, there is an associated Eilenberg-Mac Lane $C_p$-spectrum, denoted $H\uM$ (see, for example, \cite{dosSantos, dosSantosZhaohu}). As is true classically, these Eilenberg-Mac Lane $C_p$-spectra are characterized by their $\ZZ$-graded homotopy groups:
\[
\underline{\pi}_n(H\uM) = \begin{cases}
    \uM & n = 0 \\
    \underline{0} & \text{else}.
\end{cases}
\]

There are two important gradings of $C_p$-Mackey functors we will use in this paper, $\ZZ$ and $RO(C_p)$. Given a group $G$, $RO(G)$ is the ring of formal differences of real representations of $G$. For example, $RO(C_2)$ can be written as $\{k + \ell \sigma \, | \, k,\ell \in \ZZ\}$ where $\sigma$ is the sign representation.

\begin{definition} A $\ZZ$\emph{-graded $C_p$-Mackey functor} is defined as a collection of $C_p$-Mackey functors $\{ \underline{M}_n \}_{n \in \ZZ}$, written as $\underline{M}_*$. A map of $\ZZ$-graded $C_p$-Mackey functors $\underline{M}_* \to \underline{N}_*$, is a collection of maps of $C_p$-Mackey functors $\{\underline{M}_n \to \underline{N}_n \}_{n \in \ZZ}$. An $RO(C_p)$\emph{-graded $C_p$-Mackey functor} is similarly defined.
\end{definition}

The homotopy groups of $G$-spectra are naturally $RO(G)$-graded.

\begin{definition}
Let $E$ be a $C_p$-spectrum. For all $\alpha = [\gamma] - [\beta] \in RO(C_p)$, the equivariant homotopy groups of $E$, denoted $\underline{\pi}_\alpha(E),$ are defined by
\[
    \underline{\pi}_\alpha(E)(C_p/H) = \pi_\alpha(E^H) = [S^\gamma \smashy C_p/H_+, S^\beta \smashy E]^{C_p} = [S^\gamma, S^\beta \smashy E]^H.
\]
\noindent Where $[-,-]^H$ denotes the homotopy classes of $H$-equivariant maps.
\end{definition}

For each subgroup $H \leq C_p$ and each $\alpha \in RO(C_p)$, there are level wise $C_p$-actions on $\underline{\pi}_\alpha(E)$ defined by
\[
    C_p \times \underline{\pi}_\alpha(E)(C_p/H) \to C_p/H  \times \underline{\pi}_\alpha(E)(C_p/H) \to \underline{\pi}_\alpha(E)(C_p/H),
\]
\noindent where the second map is the Weyl group action. These level wise $C_p$-actions assemble into a $C_p$-action on $\underline{\pi}_\alpha(E)$.

The category of $C_p$-Mackey functors is an abelian category that has a symmetric monoidal product called the box product, denoted $\square$. The box product was first defined by Lewis in Section 1 of \cite{Lewis}. The unit for the box product is the Burnside Mackey functor $\underline{A}$, see \Cref{BurnsideMackeyFunctor}. Lewis gives a formula for the box product of two $C_p$-Mackey functors in \cite{LewisROG}.

\begin{definition}[{\cite{LewisROG}}]  \label{BoxProduct} 
Let $p$ be prime and choose a generator $\gamma$ for $C_p$. Let $\uM$ and $\uN$ be $C_p$-Mackey functors, where their transfer and restriction maps are decorated with an $\uM$ and $\uN$ respectively.
We can define $\underline{M} \, \square \, \underline{N}$:
% https://q.uiver.app/#q=WzAsMixbMCwwLCJcXEJpZyhcXHVuZGVybGluZXtNfShDX3AvQ19wKSBcXG90aW1lcyBcXHVuZGVybGluZXtOfShDX3AvQ19wKSBcXG9wbHVzIFxcYmlnKFxcdW5kZXJsaW5le019KENfcC9lKSBcXG90aW1lcyBcXHVuZGVybGluZXtOfShDX3AvZSlcXGJpZykvQ19wXFxCaWcpL197RlJ9Il0sWzAsMiwiXFx1bmRlcmxpbmV7TX0oQ19wL2UpIFxcb3RpbWVzIFxcdW5kZXJsaW5le059KENfcC9lKSJdLFswLDEsInJlcyIsMix7ImN1cnZlIjoyfV0sWzEsMCwidHIiLDIseyJjdXJ2ZSI6Mn1dXQ==
\[\begin{tikzcd}[row sep=1.1em]
	{\Big(\underline{M}(C_p/C_p) \otimes \underline{N}(C_p/C_p) \oplus Im(tr) \Big)/_{FR}} \\
	\\
	{\underline{M}(C_p/e) \otimes \underline{N}(C_p/e).}
	\arrow["res"', curve={height=12pt}, from=1-1, to=3-1]
	\arrow["tr"', curve={height=12pt}, from=3-1, to=1-1]
\end{tikzcd}\]
\noindent The $C_p$-action on $\uM(C_p/e) \otimes \uN(C_p/e)$ is given by $\gamma(x \otimes y) = \gamma(x) \otimes \gamma(y)$. The restriction map is defined by $res(a \otimes b) = res_{\uM}(a) \otimes res_{\uN}(b)$ and for any element $tr(z) \in Im(tr)$, $res(tr(z)) = z + \gamma z + \ldots + \gamma^{p-1} z$. The notation $FR$ denotes the Frobenius reciprocity submodule which is generated by elements of the form:
\begin{align*}
    (a \otimes tr_{\uN}(y), 0) & - (0, tr(res_{\uM}(a) \otimes y)), \, \text{and} \\
    (tr_{\uM}(x) \otimes b, 0) & - (0, tr(x \otimes res_{\uN}(b)))
\end{align*}
\noindent for $x \in \uM(C_p/e), y \in \uN(C_p/e), a \in \uM(C_p/C_p)$, and $b \in \uN(C_p/C_p).$ 
\end{definition}

\begin{proposition}[{\cite{Lewis, LewisROG, ShulmanThesis}}] \label{box product map} 
    Let $p$ be prime. For $C_{p}$-Mackey functors $\uM, \uN$ and $\uL$, maps $\uM \, \square \, \uN \to \uL$ are in natural bijective correspondence with collections of Weyl equivariant maps which satisfy certain conditions. Namely,
    \begin{align*}
        f_{C_p}\colon & \uM(C_{p}/C_{p}) \otimes \uN(C_{p}/C_{p}) \to \uL(C_{p}/C_{p}) \\
        f_e\colon & \uM(C_{p}/e) \otimes \uN(C_{p}/e) \to \uL(C_{p}/e)
    \end{align*}
    \noindent such that the following compatibility conditions are satisfied:
    \begin{enumerate}
        \item $res \circ f_{C_p} = f_e \circ (res \otimes res)$ \\
        \item $f_{C_p} \circ (tr \otimes \id) = tr \circ f_e \circ (\id \otimes res)$ \\
        \item $f_{C_p} \circ (\id \otimes tr) = tr \circ f_e \circ (res \otimes \id).$
    \end{enumerate}
\end{proposition}

There is also a box product on graded $C_p$-Mackey functors. 

\begin{definition}[{\cite[Definition 2.4]{LM}}]
Let $\underline{N}_*$ and $\underline{M}_*$ be $\ZZ$-graded $C_p$-Mackey functors. We define $\underline{N}_* \, \square \, \underline{M}_*$ as a $\ZZ$-graded $C_p$-Mackey functor such that for $n \in \ZZ$

\begin{center}
    $(\underline{N}_* \, \square \, \underline{M}_*)_n = \underset{n = k + \ell}{\bigoplus} (\underline{N}_k \, \square \, \underline{M}_\ell)$.
\end{center}

\noindent The definition is similar for $RO(C_p)$-graded Mackey functors.
\end{definition}

The unit for the product on $\ZZ$-graded $C_p$-Mackey functors is $\underline{A}_*$, which is $\underline{A}$ in degree $0$ and $\underline{0}$ in all other degrees. Similarly, the unit for the product on $RO(C_p)$-graded $C_p$-Mackey functors is $\uA_\star$, which is $\uA$ in degree $0$ and $\underline{0}$ in all other degrees.

\subsection{Green functors} \label{Subsec: Green functors}

An equivariant analogue to rings can be built using the box product on Mackey functors.

\begin{definition}[{\cite[Definition 2.1(a)]{Lewis}}] A \emph{$C_p$-Green functor} $\underline{R}$ is a Mackey functor with a unit map $\eta \colon \underline{A} \to \underline{R}$ and a multiplication map $\phi \colon \underline{R} \, \square \, \underline{R} \to \underline{R}$ such that the usual diagrams commute. A $C_p$-Green functor is said to be commutative if $\phi \cong \phi \circ \tau$ where $\tau$ swaps the two copies of $\underline{R}$.
\end{definition}

\begin{definition}[{\cite[Definition 3.1]{LM}}] A \emph{$\ZZ$-graded $C_p$-Green functor} $\underline{R}_*$ is a collection of Mackey functors, $\{\uR_n\}_{n \in \ZZ}$ with a unit map $\eta \colon \underline{A}_* \to \underline{R}_*$ and a multiplication map $\phi \colon \underline{R}_* \, \square \, \underline{R}_* \to \underline{R}_*$ such that the usual diagrams commute. A $\ZZ$-graded $C_p$-Green functor is said to be commutative if $\phi \cong \phi \circ \tau$ where $\tau$ rotates the two copies of $\uR_*$.
\end{definition}

In order to define an $RO(C_p)$-graded commutative $C_p$-Green functor we will first define the rotating isomorphism. For $\alpha, \beta \in RO(C_p)$, the switch map $S^\alpha \smashy S^\beta \to S^\beta \smashy S^\alpha$ gives an element in the Burnside ring $A(C_p)$. Let us refer to this element as $s(\alpha,\beta)$. Further, for $\uM_\star$ and $\uN_\star$ $RO(C_p)$-graded Mackey functors, $s(\alpha, \beta)$ induces an automorphism $\uM_\alpha \, \square \, \uN_\beta \to \uM_\alpha \, \square \, \uN_\beta$. This automorphism, along with the symmetry isomorphism of the tensor product of abelian groups, gives an isomorphism $r_{\alpha,\beta}\colon \uM_\alpha \, \square \, \uN_\beta \to \uN_\beta \, \square \, \uM_\alpha$.

\begin{definition} \label{rotating isomorphism}
    Let $\uM_\star$ and $\uN_\star$ be $RO(C_p)$-graded Mackey functors. The \emph{rotating isomorphism}, noted $\tau\colon \uM_\star \, \square \, \uN_\star \to \uN_\star \, \square \, \uM_\star$ is defined at the level $\gamma \in RO(C_p)$, by $r_{\alpha,\beta}\colon \uM_\alpha \, \square \, \uN_\beta \to \uN_\beta \, \square \, \uM_\alpha$, as defined above, for all $\alpha + \beta = \gamma$.
\end{definition}

We are now ready to define $RO(C_p)$-graded $C_p$-Green functors.

\begin{definition} An \emph{$RO(C_p)$-graded $C_p$-Green functor} $\underline{R}_\star$ is a collection of Mackey functors, $\{\uR_\alpha \}_{\alpha \in RO(C_p)}$ with a unit map $\eta \colon \underline{A}_\star \to \underline{R}_\star$ and a multiplication map $\phi \colon \underline{R}_\star \, \square \, \underline{R}_\star \to \underline{R}_\star$ such that the usual diagrams commute. An $RO(C_p)$-graded $C_p$-Green functor is said to be commutative if $\phi \cong \phi \circ \tau$ where $\tau$ is the rotating isomorphism from \Cref{rotating isomorphism}. 

\end{definition}
This allows us to define Mackey modules, submodules, and ideals. Let us first consider the following definition.

\begin{definition}
    A \emph{subfunctor}, say $\underline{S}$, of a $C_p$-Mackey functor $\underline{M}$ is a functor from finite $C_p$-sets to abelian groups where $\underline{S}(C_p/e) \leq \underline{M}(C_p/e)$, $\underline{S}(C_p/C_p) \leq \underline{M}(C_p/C_p)$, and the transfer, restriction, and Weyl group action of $\underline{S}$ are induced from $\underline{M}.$
\end{definition}

Lewis defines modules over a Green functor in \cite{Lewis}, and Lewis and Mandell define $RO(C_p)$-graded modules over $RO(C_p)$-graded Green functors in \cite{LM}.

\begin{definition}[{\cite[Definition 2.1(b)]{Lewis},\cite[Definition 3.2]{LM}}] For $\underline{R}$ a $C_p$-Green functor, a \emph{left $\underline{R}$-module} is a $C_p$-Mackey functor $\underline{M}$ with a module structure map $\xi\colon \underline{R} \, \square \, \underline{M} \to \underline{M}$ such that the usual diagrams commute. \emph{Right $\underline{R}$-modules} and \emph{$\uR$-bimodules} are defined analogously. If $\underline{R}$ is commutative, then every left (right) $\uR$-module is an $\uR$-bimodule. 

Graded modules over $\ZZ$-graded Green functors and $RO(C_p)$-graded Green functors are defined similarly. It is notable that a $\ZZ$-graded \emph{$\underline{R}_*$-submodule} $\underline{N}_*$ of $\underline{M}_*$ is a subfunctor on every level, meaning that $\uN_n$ is a subfunctor of $\uM_n$ for all $n \in \ZZ$, where $\uN_n$ is closed under the action of $\underline{R}_*$. Similarly for an $RO(C_p)$-graded $\underline{R}_\star$-submodule.
\end{definition}

Lewis defines ideals of a Green functor in \cite{Lewis}.

\begin{definition}[{\cite[Definition 2.1(c)]{Lewis}}] 
    Let $\uR$ be a $C_p$-Green functor. A \emph{left ideal} $\uI$ of $\uR$ is a submodule of $\uR$ considered as a left module over itself. Analogously, one can define \emph{right ideals} and \emph{two sided ideals.}
\end{definition}

Graded ideals for $\ZZ$-graded and $RO(C_p)$-graded Green functors are defined similarly.

\begin{definition}
    Let $\uR_*$ be a $\ZZ$-graded $C_p$-Green functor. A \emph{left ideal} $\uI_*$ of $\uR_*$ is an $\uR_*$-submodule considered as a left $\uI_*$-module. Analogously, one can define \emph{right ideals} and \emph{two sided ideals.}
\end{definition}
\begin{definition}
    Let $\uR_\star$ be a $RO(C_p)$-graded $C_p$-Green functor. A \emph{left ideal} $\uI_\star$ of $\uR_\star$ is an $\uR_\star$-submodule considered as a left $\uI_\star$-module. Analogously, one can define \emph{right ideals} and \emph{two sided ideals.}
\end{definition}

Note that, as is true classically, if $\uI_\star$ is an $RO(C_p)$-graded left ideal of $\uR_\star$, then $\uI_0$ must be a left ideal of $\uR_0$. This is because by the definition of an $\uR_\star$-subfunctor, there is an inclusion map $\uI_0 \to \uR_0$ and the module structure map on $\uR_0 \, \square \, \uI_0$ must land in $\uI_0$. Therefore $\uI_0$ is an ideal of $\uR_0$.

Classically, we have relative tensor products of abelian groups defined from a coequalizer diagram. The following is the equivariant analogue. Lewis defines the non-graded case in \cite{Lewis}, and Lewis and Mandell extend this to the $RO(C_p)$-graded case in \cite{LM}.

\begin{definition}[{\cite{Lewis}, \cite[Definition 3.6(a)]{LM}}]
Let $\underline{L}$ and $\underline{M}$ be right and left $\underline{R}$-modules, respectively, for $\underline{R}$ a $C_p$-Green functor. Define $\underline{L} \, \square_{\underline{R}} \, \underline{M}$ as the coequalizer of the following diagram in the category of $C_p$-Mackey functors
% https://q.uiver.app/#q=WzAsMixbMCwwLCJcXHVuZGVybGluZXtMfSBcXCwgXFxzcXVhcmUgXFwsIFxcdW5kZXJsaW5le1J9IFxcLCBcXHNxdWFyZSBcXCwgXFx1bmRlcmxpbmV7TX0iXSxbMSwwLCJcXHVuZGVybGluZXtMfSBcXCwgXFxzcXVhcmUgXFwsIFxcdW5kZXJsaW5le019Il0sWzAsMSwiXFxyaG8gXFwsIFxcc3F1YXJlIFxcLCBcXGlkIiwwLHsib2Zmc2V0IjotMn1dLFswLDEsIlxcaWQgXFwsIFxcc3F1YXJlIFxcLCBcXGxhbWJkYSIsMix7Im9mZnNldCI6Mn1dXQ==
\[\begin{tikzcd}
	{\underline{L} \, \square \, \underline{R} \, \square \, \underline{M}} & {\underline{L} \, \square \, \underline{M}}
	\arrow["{\rho \, \square \, \id}", shift left=2, from=1-1, to=1-2]
	\arrow["{\id \, \square \, \lambda}"', shift right=2, from=1-1, to=1-2]
\end{tikzcd}\]
\noindent for $\rho$ and $\lambda$ being the right and left module actions, respectively. The definition is similar for $\ZZ$-graded and $RO(C_p)$-graded Green functors.
\end{definition}

\subsection{Mackey fields}

We are now ready to discuss an equivariant analogue to fields. Lewis defines Mackey fields in \cite{Lewis}. In this subsection we will naturally extend this definition to define graded Mackey fields.

\begin{definition}[{\cite[Definition 2.6(f)]{Lewis}}] 
A \emph{$C_p$-Mackey field} $\underline{F}$ is a nonzero, commutative $C_p$-Green functor with no nontrivial proper ideals.
\end{definition}

One may guess that the constant Mackey functor over a field is always a Mackey field, but this is not always true.

\begin{example}
    The constant $C_2$-Mackey functor over the field with two elements, $\underline{\FF}_2$, is not a $C_2$-Mackey field. Recall that $\underline{\FF}_2(C_2/C_2) = \FF_2$, $\underline{\FF}_2(C_2/e) = \FF_2$, $tr = 0$, and $res = \id$. This has the nontrivial ideal $\uI$ such that $\uI(C_2/C_2) = 0$, $\uI(C_2/e) = \FF_2$, $tr = 0$, and $res = 0$.
\end{example}

\begin{example}
    The constant $C_2$-Mackey functor over the field with $p$ elements, $\underline{\FF}_p$, for $p$ an odd prime is a Mackey field. Recall that $\underline{\FF}_p(C_2/C_2) = \FF_p$, $\underline{\FF}_p(C_2/e) = \FF_p$, $tr = 2$, and $res = \id$. In fact, $\underline{\FF}_p$ is a $C_q$-Mackey field for any distinct primes $p$ and $q$.
\end{example}

There are many other examples of Mackey fields. In fact any $C_p$-Mackey functor $\uM$ where $\uM(C_p/e) = 0$ and $\uM(C_p/C_p) = F$ for $F$ a field is a $C_p$-Mackey field. The following is another interesting example.

\begin{example}
    The fixed point $C_2$-Mackey functor on $\CC$, say $\underline{FP}(\CC)$, where the $C_2$-action on $\CC$ is complex conjugation is a $C_2$-Mackey field. This Mackey functor is defined by the following characteristics: $\underline{FP}(\CC)(C_2/C_2) = \CC^{C_2} = \RR$, $\underline{FP}(\CC)(C_2/e) = \CC$, $tr = 2Re$, and $res = inc$.
\end{example}

One may have noticed that for every example of a Mackey field $\uF$, $\uF(C_p/C_p)$ has been a field. This is not special to these examples, but is a feature of Mackey fields.

\begin{proposition}[{\cite[Proposition 3.9(f)]{Lewis}}]
    If $\uF$ is a $C_p$-Mackey field, then $\uF(C_p/C_p)$ is a field.
\end{proposition}

Note that it is not necessarily true that $\underline{F}(C_p/e)$ is a field. It turns out that $C_p$-Mackey fields can be ``sorted" by induction theory by whether $\uF(C_p/e)$ is zero or not.

We will now create a natural definition of graded Mackey fields.

\begin{definition}
    A \emph{$\ZZ$-graded $C_p$-Mackey field} is a nonzero, commutative $\ZZ$-graded $C_p$-Green functor with no nontrivial $\ZZ$-graded ideals.
    An $RO(C_p)$\emph{-graded Mackey field} is defined similarly.
\end{definition}

\subsection{Equivariant homotopy groups of Eilenberg-Mac Lane spectra} \label{Sec: Homotopy Groups}

In this section, we will compute $\underline{\pi}_\star(H\underline{F})$ for any $C_p$-Mackey field $\underline{F}$, $p$ prime. Note that Ferland and Lewis compute $\underline{\pi}_\star(H\underline{M})$ for $\uM$ any $C_p$-Mackey functor in Section 8 of \cite{Ferland-Lewis}. We found that computing $\underline{\pi}_\star(H\underline{M})$ is much simpler in the case when $\uM$ is a $C_p$-Mackey field, so it is worth revisiting. Let us first recall the definition of equivariant homology and cohomology.

Define $a_+$ as the disjoint union of the discrete space $a$ and a $C_p$-trivial point. For $C_p$-spectra $E$ and $D$ we regard $[D,E]^{C_p}$ as a $C_p$-Mackey functor by defining $[D,E]^{C_p}(a) = [\Sigma^\infty a_+ \smashy D, E]^{C_p}$ for $a$ a finite $C_p$-set. Similarly for a $C_p$-space $X$, $[\Sigma^\infty X,E]^{C_p}(a) = [\Sigma^\infty a_+ \smashy \Sigma^\infty X, E]^{C_p}$ is a $C_p$-Mackey functor.

\begin{definition}[{\cite{LMM, LMS}}] \label{EquivariantHomology}  Assume $E$ and $D$ are $C_p$-spectra, and $X$ a $C_p$-space. Then $C_p$-equivariant $E$\emph{-cohomology} and $E$\emph{-homology} of $X$ and $D$ are given by

    \begin{center}
    \begin{tabular}{ccc}
       $\underline{E}^\star (X) = [\Sigma^\infty_+ X, S^\star \smashy E]^{C_p}$,  & &  $\underline{E}_\star (X) = [S^\star , \Sigma^\infty_+ X \smashy E]^{C_p}$, \\
       $\underline{\tilde E}^\star (X) = [\Sigma^\infty X, S^\star \smashy E]^{C_p}$,  &  &  $\underline{\tilde E}_\star (X) = [S^\star , \Sigma^\infty X \smashy E]^{C_p}$, \\
       $\underline{E}^\star (D) = [D, S^\star \smashy  E]^{C_p}$,  & and & $\underline{E}_\star  (D) = [S^\star , D \smashy E]^{C_p}$.
    \end{tabular}
\end{center}

\end{definition}

\noindent Note that  $\underline{\pi}_\star(E \smashy D) = \underline{E}_\star(D)$. For a Mackey functor $\uM$, $\underline{H}_\star(-;\uM) \cong H\uM_\star(-)$ and $\underline{H}^\star(-;\uM) \cong H\uM^\star(-)$, meaning $E = H\uM$ is the Eilenberg-Mac Lane spectrum of $\uM$ in the notation in the above definition.

The following result of Oru\c{c} in \cite{Oruc} is more general than we state it here. For the purposes of this paper we will focus on $C_p$-Mackey fields.

\begin{proposition}[{\cite[Proposition 3.11]{Oruc}}] \label{OrucCoHomology}
 For $\uF$ a $C_p$-Mackey field such that $\uF(C_p/e)=0$, $\underline{M}$ a module over $\uF$, $E$ a $C_p$-spectrum, and $\alpha \in RO(C_p)$ we have
\begin{align*}
    \underline{H}_\alpha(E ; \underline{M}) & \cong J_{C_p/C_p}(H_{\dim(\alpha^{C_p})}(\Phi^{C_p}(E) ; \underline{M}(C_p/C_p))), \, \text{and} \\
    \underline{H}^\alpha(E ; \underline{M}) & \cong J_{C_p/C_p}(H^{\dim(\alpha^{C_p})}(\Phi^{C_p}(E) ; \underline{M}(C_p/C_p))).
\end{align*}
\noindent For $\uF'$ a $C_p$-Mackey field such that $\uF'(C_p/e) \neq 0$, $\uM'$ a module over $\uF'$, $E$ a $C_p$-spectrum, and $\alpha \in RO(C_p)$ we have
\begin{align*}
    \underline{H}_\alpha(E ; \underline{M}') & \cong J_{C_p/e}(H_{\dim(\alpha)}(E ; \underline{M}'(C_p/e)) \otimes \ZZ), \, \text{and} \\
    \underline{H}^\alpha(E ; \underline{M}') & \cong J_{C_p/e}(H^{\dim(\alpha)}(E ; \underline{M}'(C_p/e)) \otimes \ZZ),
\end{align*}
\noindent where $\ZZ$ has a trivial $C_p$-action unless both $p = 2$ and $\alpha = k - k\sigma$ for $k$ an odd integer, where $\sigma$ is the sign representation. The non-trivial $C_2$-action on $\ZZ$ is multiplication by $-1$.
\end{proposition}

Let $\SS_{C_p}$ and $\SS$ be the units in the category of ring $C_p$-spectra and the category of ring spectra, respectively. We can write $\underline{\pi}_\star(H\underline{F})$ as $\underline{H}_\star(\SS_{C_p}; \underline{F})$. Recall the following property of the geometric fixed point functor:
\begin{center}
    $\Phi^{C_p}(\Sigma^\infty X) \simeq \Sigma^\infty (X^{C_p}).$
\end{center}
The $C_p$-geometric fixed points of $\SS_{C_p}$ are as follows, $\Phi^{C_p}(\SS_{C_p}) = \Phi^{C_p}(\Sigma^\infty_{C_p} S^0) = \Sigma^\infty((S^0)^{C_p}) = \Sigma^\infty S^0 = \SS$. Therefore \Cref{OrucCoHomology} shows that for $\uF$ and $\uF'$ $C_p$-Mackey fields such that $\uF(C_p/e) = 0$ and $\uF'(C_p/e)\neq 0$,
\begin{equation} \label{C_pHomotopyGroups} \underline{\pi}_\alpha(H\underline{F}) \cong J_{C_p/C_p}(H_{\dim(\alpha^{C_p})}(\SS; \underline{F}(C_p/C_p)))
\end{equation}
\begin{equation}\label{eHomotopyGroups}
\underline{\pi}_\alpha(H\underline{F}') \cong J_{C_p/e}(H_{\dim(\alpha)}(\SS_{C_p}; \underline{F}'(C_p/e))\otimes \ZZ)
\end{equation}
\noindent for any $\alpha \in RO(C_p)$. As stated in \Cref{OrucCoHomology} $\ZZ$ has a non-trivial action only when $p = 2$ and $\alpha = k - \sigma k$ for $k$ an odd integer.

\begin{proposition}[{\cite{Lewis}}] \label{transfer nonzero}
    If $\uF$ is a $C_p$-Mackey field such that $\uF(C_p/e) \neq 0$, then the transfer map is nonzero.
\end{proposition}

We can now describe what any $C_p$-Mackey field must look like.

\begin{proposition} \label{All Mackey fields}
    Any $C_p$-Mackey field $\uF$ has of one of the following two forms:
% https://q.uiver.app/#q=WzAsNCxbMCwwLCJrIl0sWzAsMSwiMCJdLFsyLDAsIlJee0NfcH0iXSxbMiwxLCJSIl0sWzEsMCwiMCIsMix7ImN1cnZlIjoxfV0sWzAsMSwiMCIsMix7ImN1cnZlIjoxfV0sWzMsMiwidHIiLDIseyJjdXJ2ZSI6MX1dLFsyLDMsImluYyIsMix7ImN1cnZlIjoxfV1d
\[\begin{tikzcd}
	k && {R^{C_p}} \\
	0 && R
	\arrow["0"', curve={height=6pt}, from=1-1, to=2-1]
	\arrow["inc"', curve={height=6pt}, from=1-3, to=2-3]
	\arrow["0"', curve={height=6pt}, from=2-1, to=1-1]
	\arrow["tr"', curve={height=6pt}, from=2-3, to=1-3]
\end{tikzcd}\]
\noindent for $k$ and $R^{C_p}$ fields, $R$ a ring with $C_p$-action, $tr(x) = \sum_{i = 0}^{p-1} \gamma^i x$ for $x \in R$, and $tr \neq 0$.
\end{proposition}

\begin{proof}

    For any Mackey field $\uF$, $\underline{\pi}_0(H\uF) \cong \uF$ by definition of the Eilenberg-Mac Lane spectrum $H\uF$.

    If $\uF(C_p/e) = 0$ and $\uF(C_p/C_p) = k$ a field, then \Cref{C_pHomotopyGroups} shows us that $\underline{\pi}_0(H\uF) \cong J_{C_p/C_p}(k)$. By \Cref{J-Mackey Functor Example} this is the left hand diagram above.

    If $\uF(C_p/C_p) \neq 0$, then since $\uF$ is a Mackey field, it is also a commutative Green functor so $\uF(C_p/e)$ must be a ring, say $R$. \Cref{eHomotopyGroups} shows us that $\underline{\pi}_0(H\uF) \cong J_{C_p/e}(R)$. Recall that in degree $0$ the copy of $\ZZ$ has a trivial $C_p$ action for any prime $p$. By \Cref{J-Mackey Functor Example} $J_{C_p/e}(R)$ is the right hand diagram above, and $tr(x) = \sum_{i = 0}^{p-1} \gamma^i x$ for $x \in R$. Further, the transfer map of $J_{C_p/e}(R)$ must be non-zero by \Cref{transfer nonzero}.
\end{proof}

Let $\uF$ be a $C_p$-Mackey field such that $\uF(C_p/e) = 0$. Recall that by definition of the Eilenberg-Mac Lane spectrum $\underline{\pi}_0(H\uF) \cong \uF$, and $\underline{\pi}_n(H\uF) = \underline{0}$ for any nonzero integer $n$. Therefore \cref{C_pHomotopyGroups} shows us that $\underline{\pi}_{\alpha}(H\uF) \cong \underline{\pi}_{0}(H\uF) \cong \uF$ for any $\alpha \in RO(C_p)$ such that $\dim(\alpha^{C_p}) = 0$. Further, $\underline{\pi}_{\alpha}(H\uF) \cong \underline{0}$ for any $\alpha \in RO(C_p)$ such that $\dim(\alpha^{C_p}) \neq 0$. From this, we get the following description:

\begin{center}
$\underline{\pi}_\alpha(H\underline{F}) \cong \begin{cases}
    \uF & \deg(\alpha^{C_p}) = 0 \\
    \underline{0} & \text{else.}
    \end{cases}$
\end{center}

The $\ZZ$-graded homotopy groups of the Eilenberg-Mac Lane spectrum of a Mackey field gives a $\ZZ$-graded Mackey field since $\underline{\pi}_*(H\uF)$ is $\uF$ in degree $0$ and $\underline{0}$ in all other degrees. The $RO(C_p)$-graded case is more complicated.

\begin{proposition} \label{CpGradedMackeyField}
    Let $p$ be prime. For $\underline{F}$ a $C_p$-Mackey field such that $\uF(C_p/e) = 0$, $H\underline{F}_\star$ is an $RO(C_p)$-graded Mackey field.
\end{proposition}

\begin{proof}
Let $\uF$ be a $C_p$-Mackey field such that $\uF(C_p/C_p) = k$ a field, and $\uF(C_p/e) = 0$. By way of contradiction, say there exists a nontrivial, proper $RO(C_p)$-graded ideal $\uI_\star$ of $H\uF_\star$. The definition of a graded ideal says that $\uI_\alpha$ is a submodule of $H\uF_\alpha$ for all $\alpha \in RO(C_p)$. One consequence is that $\underline{I}_\alpha(C_p/H) \subseteq H\underline{F}_\alpha(C_p/H)$ for all $\alpha\in RO(C_p)$ and $H = e, C_p$. The module structure map $H\uF_\alpha \, \square \, \uI_\beta \to \uI_{\alpha + \beta}$ comes from the module structure of $H\uF_\star$ as a module over itself. Therefore the module structure is induced from the multiplication map $H\uF_\alpha \, \square \, H\uF_\beta \to H\uF_{\alpha + \beta}$. If $H\uF_\alpha = \underline{0}$ then $H\uF_\alpha \, \square \, H\uF_\beta = \underline{0}$ so $H\uF_\alpha \, \square \, H\uF_\beta \to H\uF_{\alpha + \beta}$ is the $\underline{0}$ map, but $H\uF_{\alpha + \beta}$ could be non-zero. If $H\uF_\alpha \neq \underline{0} \neq H\uF_\beta$ then by \Cref{BoxProduct} $(H\uF_\alpha \, \square \, H\uF_\beta)(C_p/e) = 0$ and $(H\uF_\alpha \, \square \, H\uF_\beta)(C_p/C_p) = k \otimes k$. Note that if $H\uF_\alpha \neq \underline{0} \neq H\uF_\beta$ then by the above computation, $\deg(\alpha^{C_p}) = 0 = \deg(\beta^{C_p})$ so $\deg((\alpha + \beta)^{C_p}) = 0$ and $H\uF_{\alpha + \beta} = \uF \neq \underline{0}$. Therefore the multiplication map $H\uF_\alpha \, \square \, H\uF_\beta \to H\uF_{\alpha + \beta}$ is the $0$ map on the $C_p/e$ level and the multiplication map, $k \otimes k \to k$, on the $C_p/C_p$ level.

Recall that $\uI_0$ must be an ideal of $H\uF_0 = \uF$, and since $\uF$ is a Mackey field then $\uI_0$ is either $\uF$ or $\underline{0}$. We will show that for either situation $\uI_\star$ must be $\underline{0}_\star$ or $H\uF_\star$, which is a contradiction since $\uI_\star$ is assumed to be a nontrivial, proper ideal.

If $\uI_0 = \uF$, choose $\alpha \in RO(C_p)$ such that $H\uF_\alpha \neq \underline{0}$, so $H\uF_\alpha = \uF$. Consider the module structure map $H\uF_\alpha \, \square \, \uI_0 \cong \uF \, \square \, \uF \to \uI_\alpha$. As discussed above, this is induced by the multiplication map $H\uF_\alpha \, \square \, H\uF_0 \cong \uF \, \square \, \uF \to H\uF_{\alpha} \cong \uF$ which is the $0$ map on the $C_p/e$ level and the multiplication map $k \otimes k \to k$ on the $C_p/C_p$ level. So, $\uI_\alpha(C_p/C_p)$ must be $k$. Therefore $\uI_\alpha = \uF$ for all $\alpha$ such that $H\uF_\alpha \neq \underline{0}$ so $\uI_\star \cong H\uF_\star$.

    If $\uI_0 = \underline{0},$ choose $\alpha \in RO(C_p)$ such that $H\uF_\alpha \neq \underline{0}$,  so $H\uF_\alpha = \uF.$ Consider the module structure map $H\uF_\alpha \, \square \, \uI_{-\alpha} \cong \uF \, \square \, \uI_{-\alpha} \to \uI_0 = \underline{0}$. For any $\alpha \in RO(C_p)$, if $\dim(\alpha^{C_p}) = 0$ then $\dim(-\alpha^{C_p}) = 0$ and similarly if $\dim(\alpha^{C_p}) \neq 0$ then $\dim(-\alpha^{C_p}) \neq 0$. This demonstrates that $H\uF_\alpha = H\uF_{-\alpha}$ for all $\alpha \in RO(C_p)$. As discussed above, the map $\uF \, \square \, \uI_{-\alpha} \to \underline{0}$ is induced by the multiplication map $H\uF_\alpha \, \square \, H\uF_{-\alpha} \cong \uF \, \square \, \uF \to H\uF_{0} \cong \uF$ which is the $0$ map on the $C_p/e$ level and the multiplication map $k \otimes k \to k$ on the $C_p/C_p$ level. Since $\uI_{-\alpha}$ is a submodule of $H\uF_{-\alpha}$ then $\uI_{-\alpha}(C_p/C_p) \leq H\uF_{-\alpha}(C_p/C_p) = k$. Since $k$ is a field, there is no non-trivial subgroup of $k$ which multiplies with $k$ to $0$, therefore for $\uF \, \square \, \uI_{-\alpha} \to \underline{0}$ to be a multiplication map $\uI_{-\alpha}(C_p/C_p)$ must be $0$, and thus $\uI_{-\alpha} = \underline{0}$. Further, $\uI_\alpha$ must be $\underline{0}$ for all $\alpha$ such that $H\uF_\alpha \neq \underline{0}$ so $\uI_\star \cong \underline{0}_\star$.

    Therefore, $H\uF_\star$ is an $RO(C_p)$-graded Mackey field. \end{proof}

Let $\uF$ be a $C_2$-Mackey field such that $\uF(C_2/e) = R \neq 0$. Recall by the definition of the Eilenberg-Mac Lane spectrum and \cref{eHomotopyGroups} shows us that $\underline{\pi}_{k-k\sigma}(H\uF) \cong \underline{\pi}_{0}(H\uF) = \uF$ for $k$ an even integer, as the copy of $\ZZ$ has a trivial $C_2$-action in those degrees. In degrees $k - k\sigma$ for $k$ an odd integer, we have $J_{C_2/e}(R \otimes \ZZ)$ where $\ZZ$ has the non-trivial $C_2$-action of multiplication by $-1$ (see \Cref{J-Mackey Functor Example}). In all other degrees we have $J_{C_2/e}(0) \cong \underline{0}$. From this and \Cref{J-Mackey Functor Example} we get the following description:

\begin{center}
$\underline{\pi}_\alpha(H\underline{F}) \cong \begin{cases} \uF & \alpha = k - k\sigma, \, \, k \, \text{even} \\
    \begin{tikzcd}[column sep=small,row sep=small]
	{(R \otimes \ZZ)^{C_2}} \\
	\\
	{R \otimes \ZZ}
	\arrow["inc"', curve={height=12pt}, from=1-1, to=3-1]
	\arrow["{1 + \gamma}"', curve={height=12pt}, from=3-1, to=1-1]
\end{tikzcd} & \alpha = k - k\sigma, \, \, k \, \text{odd} \\
    \underline{0} & \text{else}
    \end{cases}$
\end{center}

\noindent where the non-trivial $C_2$-action, $\gamma$, is diagonal on $R \otimes \ZZ$ and the non-trivial $C_2$-action on $\ZZ$ is multiplication by $-1$.

Let $\uF$ be a $C_p$-Mackey field such that $\uF(C_p/e) = R \neq 0$ for $p$ an odd prime. A similar argument as above gives the following description:

\begin{center}
    $\underline{\pi}_\alpha(H\underline{F}) \cong \begin{cases} \uF & \dim(\alpha) = 0 \\
    \underline{0} & \text{else}.
    \end{cases}$
\end{center}

\noindent In fact, this is also the computation of the homotopy groups of $H\uF$ for a $C_2$-Mackey field if and only if $\uF(C_2/e)$ is characteristic two.

The $\ZZ$-graded homotopy groups of the Eilenberg-Mac Lane spectrum of a Mackey field gives a Mackey field since $\underline{\pi}_*(H\uF)$ is $\uF$ in degree $0$ and $\underline{0}$ in all other degrees. The $RO(C_p)$-graded case is more complicated. We will split into two cases, not because the proofs are dissimilar but because the statement of one is more broad than the other.

\begin{proposition}
    For $\uF$ a $C_p$-Mackey field such that $\uF(C_p/e) \neq 0$, and $p$ an odd prime, $H\uF_\star$ is an $RO(C_p)$-graded Mackey field.
\end{proposition}

\begin{proof}
Let $\uF$ be a $C_p$-Mackey field such that $\uF(C_p/e) = R \neq 0$ is a commutative ring, and $\uF(C_p/C_p) = R^{C_p}$ a field. By way of contradiction, say there exists a nontrivial, proper $RO(C_p)$-graded ideal $\uI_\star$ of $H\uF_\star$. The definition of a graded ideal says that $\uI_\alpha$ is a submodule of $H\uF_\alpha$ for all $\alpha \in RO(C_p)$. One consequence is that $\underline{I}_\alpha(C_p/H) \subseteq H\underline{F}_\alpha(C_p/H)$ for all $\alpha\in RO(C_p)$ and $H = e, C_p$. The module structure map $H\uF_\alpha \, \square \, \uI_\beta \to \uI_{\alpha + \beta}$ comes from the module structure of $H\uF_\star$ as a module over itself. Therefore the module structure is induced from the multiplication map $H\uF_\alpha \, \square \, H\uF_\beta \to H\uF_{\alpha + \beta}$. If $H\uF_\alpha = \underline{0}$ then $H\uF_\alpha \, \square \, H\uF_\beta = \underline{0}$ so $H\uF_\alpha \, \square \, H\uF_\beta \to H\uF_{\alpha + \beta}$ is the $\underline{0}$ map, but $H\uF_{\alpha + \beta}$ could be non-zero. If $H\uF_\alpha \neq \underline{0} \neq H\uF_\beta$ then by \Cref{BoxProduct} $(H\uF_\alpha \, \square \, H\uF_\beta)(C_p/e) = R \otimes R$ and $(H\uF_\alpha \, \square \, H\uF_\beta)(C_p/C_p) = \big(R^{C_p} \otimes R^{C_p} \oplus Im(tr)\big)/_{FR}$. For this case, the multiplication map $H\uF_\alpha \, \square \, H\uF_\beta \to H\uF_{\alpha + \beta}$ is defined as
% https://q.uiver.app/#q=WzAsNCxbMCwwLCJcXGJpZyhSXntDX3B9IFxcb3RpbWVzIFJee0NfcH0gXFxvcGx1cyBJbSh0cilcXGJpZykvX3tGUn0iXSxbMCwyLCJSIFxcb3RpbWVzIFIiXSxbMiwyLCJSIl0sWzIsMCwiUl57Q19wfSJdLFsxLDAsIiIsMCx7ImN1cnZlIjoyfV0sWzAsMSwiIiwwLHsiY3VydmUiOjJ9XSxbMSwyLCJcXG11IiwyXSxbMCwzLCJcXHBoaSJdLFszLDIsIiIsMSx7ImN1cnZlIjoyfV0sWzIsMywiIiwxLHsiY3VydmUiOjJ9XV0=
\[\begin{tikzcd}[column sep=small, row sep=small]
	{\big(R^{C_p} \otimes R^{C_p} \oplus Im(tr)\big)/_{FR}} && {R^{C_p}} \\
	\\
	{R \otimes R} && R
	\arrow["\phi", from=1-1, to=1-3]
	\arrow[curve={height=12pt}, from=1-1, to=3-1]
	\arrow[curve={height=12pt}, from=1-3, to=3-3]
	\arrow[curve={height=12pt}, from=3-1, to=1-1]
	\arrow["\mu"', from=3-1, to=3-3]
	\arrow[curve={height=12pt}, from=3-3, to=1-3]
\end{tikzcd}\]
where $\mu$ is the multiplication map, and $\phi(x \otimes y, tr(a \otimes b)) = xy + \sum_{i = 0}^{p-1} \gamma^i (ab)$. These maps respect the transfer and restriction maps as well as the Frobenius reciprocity relation.

Recall that $\uI_0$ must be an ideal of $H\uF_0 = \uF$, and since $\uF$ is a Mackey field then $\uI_0$ is either $\uF$ or $\underline{0}$. We will show that for either situation $\uI_\star$ must be $\underline{0}_\star$ or $H\uF_\star$, which is a contradiction since $\uI_\star$ is assumed to be a nontrivial, proper ideal.

If $\uI_0 = \uF$, choose $\alpha \in RO(C_p)$ such that $H\uF_\alpha \neq \underline{0}$, so $H\uF_\alpha = \uF$. Consider the module structure map $H\uF_\alpha \, \square \, \uI_0 \cong \uF \, \square \, \uF \to \uI_\alpha$. Since $\uI_\alpha$ is a submodule of $H\uF_\alpha \cong \uF$ then $\uI_\alpha(C_p/H) \subseteq \uF(C_p/H)$. As discussed above, the map $H\uF_\alpha \, \square \, \uI_0 \cong \uF \, \square \, \uF \to \uI_\alpha$ is induced by the multiplication map $H\uF_\alpha \, \square \, H\uF_0 \cong \uF \, \square \, \uF \to H\uF_{\alpha} \cong \uF$, which is the multiplication map $\mu\colon R \otimes R \to R$ on the $C_p/e$ level and $\phi$ as defined above on the $C_p/C_p$ level. Since $R^{C_p}$ is a field and $R$ is a commutative ring, both $\mu$ and $\phi$ are surjective as $\mu(1 \otimes x) = x$ for all $x \in R$ and $\phi(1 \otimes x, 0) = x$ for all $x \in R^{C_p}$, therefore $\uI_\alpha(C_p/C_p) = R^{C_p}$ and $\uI_\alpha(C_p/e) = R$. Therefore $\uI_\alpha = \uF$ for all $\alpha$ such that $H\uF_\alpha \neq \underline{0}$, so $\uI_\star \cong H\uF_\star$.

If $\uI_0 = \underline{0},$ choose $\alpha \in RO(C_p)$ such that $H\uF_\alpha \neq \underline{0}$, so $H\uF_\alpha = \uF.$ Consider the module structure map $H\uF_\alpha \, \square \, \uI_{-\alpha} \cong \uF \, \square \, \uI_{-\alpha} \to \uI_0 \cong \underline{0}$. For any $\alpha \in RO(C_p)$, if $\dim(\alpha) = 0$ then $\dim(-\alpha) = 0$ and similarly if $\dim(\alpha) \neq 0$ then $\dim(-\alpha) \neq 0$. This demonstrates that $H\uF_\alpha = H\uF_{-\alpha}$ for all $\alpha \in RO(C_p)$. Since $\uI_{-\alpha}$ is a submodule of $H\uF_{-\alpha} = \uF$, then $\uI_{-\alpha}(C_p/H) \subseteq \uF(C_p/H)$. Say $\uI_{-\alpha}(C_p/e) = K \leq R$, then the induced map, $\mu|_{R \otimes K}(1 \otimes k) = k$ must be $0$ for all $k \in K$, therefore $K = 0$. Say $\uI_{-\alpha}(C_p/C_p) = L \leq R^{C_p}$, then since $\uI_{-\alpha}$ is a submodule of $\uF$ there must be an inclusion map $\uI_{-\alpha} \to \uF$ which respects the restriction and transfer maps,
% https://q.uiver.app/#q=WzAsNCxbMCwwLCJMIl0sWzAsMiwiMCJdLFsyLDAsIlJee0NfcH0iXSxbMiwyLCJSIl0sWzEsMywiMCIsMl0sWzAsMiwiaW5jIl0sWzAsMSwiMCIsMix7ImN1cnZlIjoyfV0sWzEsMCwiMCIsMix7ImN1cnZlIjoyfV0sWzIsMywiaW5jIiwyLHsiY3VydmUiOjJ9XSxbMywyLCJ0ciIsMix7ImN1cnZlIjoyfV1d
\[\begin{tikzcd}[column sep=small, row sep=small]
	L && {R^{C_p}} \\
	\\
	0 && R
	\arrow["inc", from=1-1, to=1-3]
	\arrow[curve={height=12pt}, from=1-1, to=3-1]
	\arrow[curve={height=12pt}, from=1-3, to=3-3]
	\arrow[curve={height=12pt}, from=3-1, to=1-1]
	\arrow["0"', from=3-1, to=3-3]
	\arrow[curve={height=12pt}, from=3-3, to=1-3]
\end{tikzcd}\]
but the only way that this map of Mackey functors can respect the restriction maps is if $L = 0$. Therefore $\uI_{-\alpha}$ must be $\underline{0}$ for all $\alpha$ such that $H\uF_\alpha \neq \underline{0}$ so $\uI_\star \cong \underline{0}_\star$.

Therefore, $H\uF_\star$ is an $RO(C_p)$-graded Mackey field.
\end{proof}

The proof is the same for the following proposition.

\begin{proposition}
Let $\uF$ be a $C_2$-Mackey field, where $\uF(C_2/e) = R \neq 0$. Consider $\ZZ$ to be the $\ZZ[C_2]$-module where the $C_2$-action is multiplication by $-1$. If $R$ is such that $R \otimes \ZZ \cong R$ as $\ZZ[C_2]$-modules, then $H\uF_\star$ is an $RO(C_2)$-graded Mackey field.    
\end{proposition}

\section{Equivariant (topological) Hochschild homology}\label{Sec:Hochschild Homology}

Throughout this paper, we will study equivariant analogues of Hochschild homology (HH) and topological Hochschild homology (THH). In this section, we will recall these equivariant analogues, namely, Hochschild homology for Green functors \cite[Definition 2.25]{Witt} and twisted THH \cite[Definition 8.2]{twTHH}, which take as input Green functors and equivariant ring spectra, respectively. Both of these definitions are defined for any cyclic group $C_n$, but for the purpose of this paper we will consider only the cyclic groups $C_p$, $p$ prime.

\subsection{Hochschild homology for Green functors}

Green functors are an equivariant analogue to rings; therefore, it is natural to want an equivariant analogue of Hochschild homology which takes $\uR$-algebras as input, for $\uR$ a Green functor. In \cite{Witt} the authors define such an equivariant analogue, namely \emph{Hochschild homology for Green functors}. 

For this section, let $C_p = \langle \gamma \rangle$ be the cyclic group of order $p$ where $\gamma = e^{2\pi i/p}$. Recall from \Cref{MackeyGroupAction} that one can define a $C_p$-action on a $C_p$-Mackey functor. Let $\alpha_k\colon \uM^{\square \, k+1} \to \uM^{\square \, k+1}$ be the map that rotates the last copy of $\uM$ to the front and then acts on that $\uM$ by $\gamma$.

\begin{definition}[{\cite[Definition 2.20]{Witt}}] 
    Let $C_p = \langle \gamma \rangle$, $\uR$ a commutative $C_p$-Green functor, $\uM$ an $\uR$-algebra, and let all box products be over $\uR$. The $C_p$\emph{-twisted cyclic bar complex} of $\uM$, denoted $B_\bullet^{cy,C_p}(\uM)$, is a simplicial $C_p$-Mackey functor such that $B_k^{cy,C_p}(\uM) = \uM^{\square \, k+1}$, where the face and degeneracy maps $d_i\colon \uM^{\square \, k + 1} \to \uM^{\square \, k}$ and $s_i\colon \uM^{\square \, k+1} \to \uM^{\square \, k+2}$ are defined as follows:

    \begin{center}
        \begin{tabular}{l}
             $d_i = \begin{cases}
                 \id^{i}\, \square \, \mu\, \square \, \id^{k -i - 1} & 0 \leq i < k \\
                 (\mu\, \square \, \id^{k - 1}) \circ \alpha_k & i = k
             \end{cases}$ \\
             $s_i = \id^{i+1}\, \square \, \eta \, \square \, \id^{k - i}$ \, \, \,  $0 \leq i \leq k$
        \end{tabular}
    \end{center}
    \noindent where $\mu$ and $\eta$ are the multiplication and unit maps of $\uM$.
\end{definition}

There is an equivalence between the category of simplicial Mackey functors and the category of non-negatively graded dg Mackey functors by applying the Dold-Kan correspondence at each orbit. The homology of a simplicial Mackey functor is the homology of the associated normalized dg Mackey functor; details can be found in Section 4 of \cite{Witt}. 

\begin{definition}[{\cite[Definition 2.25]{Witt}}] Let $C_p = \langle \gamma \rangle$, $\uR$ a commutative $C_p$-Green functor, and $\uM$ an $\uR$-algebra. The \emph{Hochschild homology} of $\uM$, is defined by:
    \[
\underline{\HH}_i^{\uR,C_p}(\uM) = H_i(B_\bullet^{cy,C_p}(\uM)).
    \]
\end{definition}

Adamyk, Gerhardt, Hess, Klang, and Kong define Hochschild homology of graded Green functors in \cite{AGHKK}. Let $\alpha_k\colon \uM_\star^{\square \, k+1} \to \uM_\star^{\square \, k+1}$ be defined as first moving the last copy of $\uM_\star$ to the front by using the rotating isomorphism (see \Cref{rotating isomorphism}), and then acting on that copy of $\uM_\star$ by $\gamma$.

\begin{definition} [{\cite[Definition 4.1.7]{AGHKK}}] Let $C_p = \langle \gamma \rangle$, $\uR_\star$ a commutative $RO(C_p)$-graded Green functor, $\uM_\star$ an $\uR_\star$-algebra, and all box products are over $\uR_\star$. The $C_p$\emph{-twisted cyclic bar complex} of $\uM_\star$, denoted $B^{cy,C_p}_\bullet(\uM_\star)$, is a simplicial $RO(C_p)$-graded Mackey functor such that $B_k^{cy,C_p}(\uM_\star) = \uM_\star^{\square \, k+1}$, where the face and degeneracy maps $d_i\colon \uM_\star^{\square \, k + 1} \to \uM_\star^{\square \, k}$ and $s_i\colon \uM_\star^{\square \, k+1} \to \uM_\star^{\square \, k+2}$ are defined as follows:

    \begin{center}
        \begin{tabular}{l}
             $d_i = \begin{cases}
                 \id^{i}\, \square \, \mu\, \square \, \id^{k -i - 1} & 0 \leq i < k \\
                 (\mu\, \square \, \id^{k - 1}) \circ \alpha_k & i = k
             \end{cases}$ \\
             $s_i = \id^{i+1}\, \square \, \eta \, \square \, \id^{k - i}$ \, \, \,  $0 \leq i \leq k$
        \end{tabular}
    \end{center}
    \noindent where $\mu$ and $\eta$ are the multiplication and unit maps of $\uM_\star$.
\end{definition}

\begin{definition}[{\cite[Definition 4.1.8]{AGHKK}}] Let $C_p = \langle \gamma \rangle$, let $\uR_\star$ be a commutative $RO(C_p)$-graded Green functor, and let $\uM_\star$ be an $\uR_\star$-algebra. The \emph{Hochschild homology} of $\uM_\star$, is defined by:
    \[
\underline{\HH}_i^{\uR_\star,C_p}(\uM_\star) = H_i(B_\bullet^{cy,C_p}(\uM_\star)).
    \]
\end{definition}

Lewis and Mandell's paper \cite{LM} allows us to do homological algebra in the equivariant setting. As is true classically, there is a Tor functor perspective for Hochschild homology for Green functors. We recall the definition of Tor in this setting, as seen in \cite{LM} this definition is for a general finite group $G$ but we will only be considering the group $C_p$.

\begin{definition}[{\cite{LM}}] Let $\uR_\star$ be an $RO(C_p)$-graded Green functor. For $\uM_\star$ and $\uN_\star$ left and right $\uR_\star$-modules respectively, $\underline{\Tor}_{s,\star}^{\uR_\star}(\uN_\star, \uM_\star)$ is the $s^{\text{th}}$ left derived functor of $\uN_\star \, \square_{\uR_\star} \, \uM_\star$.
\end{definition}

In order to discuss a result connecting this equivariant Tor and Hochschild homology for Green functors, we must define the following notation for some important modules.

Let $\uR$ be a $C_p$-Green functor, $\uN$ a left $\uR$-module, and let the left $\uR$-module map for $\uN$ be denoted $\mu$. Let us define ${}^\gamma\uN$ as $\uN$ with the left $\uR$-module map ${}^\gamma \mu$, defined by the following diagram:
% https://q.uiver.app/#q=WzAsMyxbMCwwLCJcXHVSIFxcLCBcXHNxdWFyZSBcXCwgXFx1TiJdLFsyLDAsIlxcdU4iXSxbMCwyLCJcXHVSIFxcLCBcXHNxdWFyZSBcXCwgXFx1TiJdLFsyLDEsIlxcbXUiLDJdLFswLDIsIlxcaWQgXFwsIFxcc3F1YXJlIFxcLCBcXGdhbW1hIiwyXSxbMCwxLCJ7fV5cXGdhbW1hIFxcbXUiXV0=
\[\begin{tikzcd}[column sep=small, row sep=small]
	{\uR \, \square \, \uN} && \uN \\
	\\
	{\uR \, \square \, \uN}
	\arrow["{{}^\gamma \mu}", from=1-1, to=1-3]
	\arrow["{\gamma \, \square \, \id}"', from=1-1, to=3-1]
	\arrow["\mu"', from=3-1, to=1-3]
\end{tikzcd}\]
\noindent where $\gamma\colon \uR \to \uR$ acts on $\uR$ by $\gamma$, and note that the box products are not over $\uR$ in this diagram. This definition extends to ${}^\gamma \uM$ for $\uM$ an $\uR$-algebra.

We are now ready to discuss the following result which will aid us in our computation in \Cref{A Computation}.

\begin{proposition}[{\cite[Proposition 4.3.2]{AGHKK}}] \label{HH and Tor} Let $C_p = \langle \gamma \rangle$, and $\uR_\star$ be an $RO(C_p)$-graded commutative Green functor. If $\uM_\star$ is an $\uR_\star$-algebra and is flat as an $\uR_\star$-module, there is a natural isomorphism
    \[
    \underline{\HH}_*^{\uR_\star,C_p}(\uM_\star) \cong \underline{\Tor}_{*,\star}^{\uM_\star \, \square_{\uR_\star} \, \uM_\star^{\text{op}}}(\uM_\star, {}^\gamma \uM_\star).
    \]
\end{proposition}

\subsection{Twisted topological Hochschild homology}\label{Sec:TwTHH}

In \cite{twTHH}, Angeltveit, Blumberg, Gerhardt, Hill, Lawson, and Mandell define an equivariant analogue to THH which takes as input a ring $C_n$-spectrum, namely $C_n$\emph{-twisted topological Hochschild homology} for a general cyclic group $C_n$. The authors define $C_n$-twisted THH of a ring $C_n$-spectrum $A$ to be the norm $N_{C_n}^{S^1}A$, and they also show that twisted THH can be defined using a twisted analogue of the cyclic bar complex. For the purposes of this paper we will be focused on the cyclic groups $C_p$, $p$ prime.

For this section, let $C_p = \langle \gamma \rangle$ be the cyclic group of order $p$ where $\gamma = e^{2\pi i/p}$, let $A$ be an associative ring $C_p$-spectrum, and $\alpha_k\colon A^{\smashy k+1} \to A^{\smashy k+1}$ rotates the last copy of $A$ to the front and acts on that copy of $A$ by $\gamma$. 

\begin{definition} [{\cite[Definition 8.1]{twTHH}}] \label{twistedcyclicbar}
    Let $\gamma = e^{2\pi i/p}$ be the chosen generator of $C_p$, and let $A$ be an associative ring $C_p$-spectrum indexed over the trivial universe. The $C_p$\emph{-twisted cyclic bar complex} for $A$, denoted $B_\bullet^{cy,C_p}(A)$, is a simplicial object such that $B_k^{cy,C_p}(A) = A^{\smashy k+1}$. The face and degeneracy maps, $d_i\colon A^{\smashy k+1} \to A^{\smashy k}$ and $s_i\colon A^{\smashy k+1} \to A^{\smashy k+2}$ are defined as follows:

    \begin{center}
        \begin{tabular}{l}
             $d_i = \begin{cases}
                 \id^{i} \smashy \mu \smashy \id^{k -i - 1} & 0 \leq i < k \\
                 (\mu \smashy \id^{k - 1}) \circ \alpha_k & i = k
             \end{cases}$ \\
             $s_i = \id^{i+1} \smashy \eta \smashy \id^{k - i}$  \, \, \,  $0 \leq i \leq k$
        \end{tabular}
    \end{center}
    \noindent where $\mu$ and $\eta$ are the multiplication and unit maps of $A$.
\end{definition}

This twisted cyclic bar complex, along with some change of universe functors denoted $\I$, are one way to define twisted THH.

\begin{definition}[{\cite[Definition 8.2]{twTHH}}]
    Let $U$ be a complete $S^1$-universe, let $\widetilde{U} \coloneqq i_{C_p}^*U$ be the pullback of the universe to $C_p$, and let $A$ be an associative ring $C_p$-spectrum indexed on $\widetilde{U}$. The \emph{$C_p$-twisted topological Hochschild homology} of $A$ is $\THH_{C_p}(A) = \I^{U}_{\RR^\infty}|B_\bullet^{cy,C_p}(\I^{\RR^\infty}_{\widetilde{U}}A)|$.
\end{definition}

There is another useful perspective of twisted THH, in order to define this alternative perspective we must define the following notation.

Let $R$ be a ring spectrum, $E$ a left $R$-module, and let the left $R$-module map for $E$ be denoted $\mu$. Similarly to the definition above, we can define ${}^\gamma E$ as $E$ with the left $R$-module map ${}^\gamma \mu$, defined by the following diagram:
% https://q.uiver.app/#q=WzAsMyxbMCwwLCJBIFxcc21hc2h5IFIiXSxbMiwwLCJBIl0sWzAsMiwiQSBcXHNtYXNoeSBSIl0sWzIsMSwiXFxtdSIsMl0sWzAsMiwiXFxnYW1tYSBcXCwgXFxzcXVhcmUgXFwsIFxcaWQiLDJdLFswLDEsInt9XlxcZ2FtbWEgXFxtdSJdXQ==
\[\begin{tikzcd}[column sep=small, row sep=small]
	{R \smashy E} && E \\
	\\
	{R \smashy E}
	\arrow["{{}^\gamma \mu}", from=1-1, to=1-3]
	\arrow["{\gamma \, \square \, \id}"', from=1-1, to=3-1]
	\arrow["\mu"', from=3-1, to=1-3]
\end{tikzcd}\]
\noindent where $\gamma\colon R \to R$ acts on $R$ by $\gamma$. This definition extends to ${}^\gamma A$ for $A$ an $R$-algebra.

With this definition we can write the twisted cyclic bar complex, $B_{\bullet}^{cy,C_p}(A)$, as the \emph{twisted cyclic nerve}, $B_{\bullet}^{cy,C_p}(A, {}^\gamma A)$ \cite[Definition 2.20]{Witt}. The twisted cyclic nerve is defined as $B_{k}^{cy,C_p}(A, {}^\gamma A) = {}^\gamma A \smashy A^{\smashy k}$, with face and degeneracy maps as follows:
 \begin{center}
        \begin{tabular}{l}
             $d_i = \begin{cases}
                 \id^{i} \smashy \mu \smashy \id^{k -i - 1} & 0 \leq i < k \\
                 ({}^\gamma \mu \smashy \id^{k - 1}) \circ t & i = k
             \end{cases}$ \\
             $s_i = \id^{i+1} \smashy \eta \smashy \id^{k - i}$  \, \, \,  $0 \leq i \leq k$
        \end{tabular}
    \end{center}
    \noindent where $\mu$ and $\eta$ are the multiplication and unit maps of $A$, and $t$ rotates the last copy of $A$ to the front. As a $C_p$-spectrum, $\THH_{C_p}(A) = \THH(A, {}^\gamma A)$ \cite[Proposition 4.2.6]{AGHKK}.

Whenever discussing twisted topological Hochschild homology of a $C_p$-spectrum $A$, $\THH_{C_p}(A)$, we mean the left derived twisted topological Hochschild homology of $A$. In particular, this is the same as considering twisted THH of the cofibrant replacement of $A$.

The work of Adamyk, Gerhardt, Hess, Klang, and Kong in \cite[Theorem 4.2.7]{AGHKK} shows that there is an equivariant analogue of the B\"okstedt spectral sequence which demonstrates a relationship between Hochschild homology for Green functors and $C_n$-twisted THH. First, we recall that for $G$ an abelian group, and $\gamma \in G$, one can define a left $\gamma$-action on any genuine orthogonal $G$-spectrum, denoted $\ell_\gamma\colon X \to X$ (for more details see \cite[Section 3.1]{SchwedeGlobalHomotopy}). We will say that $\gamma$ acts on a $G$-spectrum trivially if $\ell_\gamma$ is equivariantly homotopic to the identity map.

For the purposes of this paper, we are only considering $C_p$-Mackey functors, but the following result as seen in \cite{AGHKK} is for a general cyclic group $C_n$.

\begin{theorem}[{\cite[Theorem 4.2.7]{AGHKK}}] \label{EquivBokSS}
     Let $C_p$ be a finite subgroup of $S^1$ such that $C_p = \langle \gamma\rangle$. Let $A$ be an associative ring $C_p$-spectrum, and $E$ a commutative ring $C_p$-spectrum such that $\gamma$ acts trivially on $E$. If $\underline{E}_\star(A)$ is flat over $\underline{E}_\star$, then there is a twisted B\"okstedt spectral sequence

\begin{center}
    $E^2_{s,\star} = \underline{\HH}_s^{\underline{E}_\star,C_p}(\underline{E}_\star(A)) \Rightarrow \underline{E}_{s + \star}(i_{C_p}^* \THH_{C_p}(A))$
\end{center}

\noindent where $d^r \colon E_{i, \alpha} \to E_{i - r, \alpha + r - 1}$.
\end{theorem}

Classically, for $R$ a commutative ring spectrum, $\THH(R) \simeq R \otimes S^1$ \cite{MSV}. There is a similar interpretation for twisted THH in the commutative case which we will use throughout this paper. To introduce this version of the definition we will first recall that the category of orthogonal $C_p$-ring spectra and the category of unbased spaces with a free $C_p$-action are tensored over the category of unbased spaces with a free $C_p$-action. For $R$ a commutative ring $C_p$-spectrum indexed over the trivial universe $\RR^\infty$, consider the functor $R \otimes_{C_p} (-)$ on the category of unbased spaces with a free $C_p$-action to be the coequalizer of the following diagram
\begin{equation}\label{coequalizer for THHCn}
% https://q.uiver.app/#q=WzAsMixbMCwwLCJSIFxcb3RpbWVzIENfcCBcXG90aW1lcyAoLSkiXSxbMiwwLCJSIFxcb3RpbWVzICgtKSJdLFswLDEsInIiLDAseyJvZmZzZXQiOi0yfV0sWzAsMSwiXFxlbGwiLDIseyJvZmZzZXQiOjJ9XV0=
\begin{tikzcd}
	{R \otimes C_p \otimes (-)} && {R \otimes (-)}
	\arrow["\id \otimes r", shift left=2, from=1-1, to=1-3]
	\arrow["\ell \otimes \id"', shift right=2, from=1-1, to=1-3]
\end{tikzcd}
\end{equation}

\noindent where $r$ is the $C_p$-action on $(-)$ and $\ell$ is the induced $C_p$-action on $R$.

The authors of \cite{twTHH} show that for $U$ a complete $S^1$-universe and $\widetilde{U} = \iota_{C_p}^*U$, and $R$ a commutative ring $C_p$-spectrum indexed over $\widetilde{U}$, 
\[\THH_{C_p}(R) \simeq \I_{\RR^\infty}^{U}(\I_{\widetilde{U}}^{\RR^\infty}(R) \otimes_{C_p} S^1).
\]
This characterization will be heavily used in \Cref{Sec:Algebraic structure} in order to demonstrate the algebraic structure of twisted THH.

\section{Computation of twisted THH of the Real bordism spectrum}\label{A Computation}

Very few computations of twisted THH are in the literature \cite{AGHKK}. Namely, in \cite[Theorem 4.3.4]{AGHKK} they compute $H{\underline{\FF}_2}_\star(\THH_{C_2}(MU_\RR))$. In this section, we will compute $H\uF_\star(\THH_{C_2}(MU_\RR))$ for $\uF$ the $C_2$-Mackey field such that $\uF(C_2/C_2) = \FF_2$, $\uF(C_2/e) = 0$, and $MU_\RR$ the Real bordism spectrum.

We will use the twisted B\"okstedt spectral sequence defined in \cite{AGHKK} (see \Cref{EquivBokSS}) to do this computation. To use this spectral sequence, we need that the non-trivial element of $C_2$ acts trivially on $H\uF$, and that $H\uF_\star(MU_\RR)$ is flat over $\underline{\pi}_\star(H\uF)$.

First, we will study which $C_p$-Mackey fields $\uF$ give an Eilenberg-Mac Lane spectrum, $H\uF$, with a trivial $C_2$-action. Then, using the fact that $MU_\RR$ is a real-oriented spectrum, we will study for which $C_p$-Mackey fields $\uF$ that $H\uF_\star(MU_\RR)$ is flat over $H\uF_\star$. We finish this section with the computation.

\subsection{The trivial action}\label{the trivial action}

Recall that in \Cref{Sec:TwTHH} we discussed that one can define a $C_p$-spectrum to have a trivial $\gamma$-action, for $\gamma \in C_p$, if $\ell_\gamma$ is equivariantly homotopic to the identity map. Let $E$ be a $C_p$-spectrum. If the Weyl action on $\underline{\pi}_\star(E)$ is trivial, then the chosen generator $\gamma$ of $C_p$ induces the identity map on the $RO(C_p)$-graded homotopy groups of $E$. Furthermore, if the only element that induces the identity map in $E^{\star}E$ is the unit $1$, then $\ell_\gamma$ must be equivariantly homotopic to the identity map.

Computations in \Cref{Sec: Homotopy Groups} show that for $\uF$ a $C_p$-Mackey field such that $\uF(C_p/C_p) = k$ and $\uF(C_p/e) = 0$, whenever $\underline{\pi}_\alpha(H\uF) \neq \underline{0}$ then $\underline{\pi}_\alpha(H\uF) = \uF$ which has a trivial Weyl action. Therefore every level of $\underline{\pi}_\star(H\uF)$ has a trivial $C_p$-action.

To compute $H\uF^\star H\uF$ we consider the following interesting fact which Oru\c{c} proves within the proof of Theorem 3.11 in \cite{Oruc}.

\begin{proposition}[{\cite{Oruc}}] \label{phiHR} 
Let $\uF$ be a $C_p$-Mackey field such that $\uF(C_p/C_p) = k$ and $\uF(C_p/e) = 0$. Then if we consider $\Phi^{C_p}(H\uF)$ as a non-equivariant ring spectrum, it is isomorphic to the non-equivariant Eilenberg-Mac Lane spectrum $Hk$.
\end{proposition}

Using this proposition along with \Cref{OrucCoHomology} we can see that $H\underline{F}^\alpha H\underline{F}(C_p/C_p) \cong Hk^{\dim(\alpha^{C_p})}Hk$ and $H\uF^\alpha H\uF(C_p/e) = 0$. Since we only need to consider $H\uF^\alpha H\uF$ for $\dim(\alpha^{C_p}) = 0$ then we only need to consider $Hk^0Hk$. 

Considering the case when $k = \FF_2$, then the question reduces to which elements of $\FF_2$ induce the identity on $H\FF_2^0 H\FF_2 \cong \FF_2$. The only element of $\FF_2$ which induces the identity on $\FF_2$ is the unit $1$. Therefore if $\uF(C_2/C_2) = \FF_2$, and $\uF(C_2/e) = 0$, then the action of $\gamma$ on $H\uF$ is trivial.

\subsection{Real oriented}

We will now recall what it means for a spectrum to be real oriented, and discuss some properties that real oriented spectra possess.

Consider $\CC \PP^n$ and $\CC \PP^\infty$ as pointed $C_2$-spaces under the action of complex conjugation, where the base point is $\CC \PP^0$. Note that the $C_2$-fixed point spaces of $\CC \PP^n$ and $\CC \PP^\infty$ are $\RR \PP^n$ and $\RR \PP^\infty$, respectively. The inclusion map $\CC\PP^1 \to \CC\PP^\infty$ induces a restriction map, $\underline{\widetilde{ E}}^{\rho}(\CC \PP^\infty)(C_2/C_2) \to \underline{\widetilde{ E}}^{\rho}(\CC \PP^1)(C_2/C_2)$, which is used in the following definition.

\begin{definition}[{\cite{ArakiOrientations}}]
Let $E$ be a $C_2$-equivariant homotopy commutative ring spectrum. A \emph{real orientation} of $E$ is a class $x \in \underline{\widetilde{ E}}^{\rho}(\CC \PP^\infty)(C_2/C_2)$ whose restriction to 
\[
\underline{\widetilde{ E}}^{\rho}(\CC \PP^1)(C_2/C_2) \cong \underline{E}^0(pt)(C_2/C_2)
\]

\noindent is the unit, where $\rho = 1 + \sigma$ is the regular representation. The spectrum $E$ is \emph{real oriented} if it has a real orientation.
\end{definition}

The following corollary builds off of this work of Araki.

\begin{corollary}[{\cite[Corollary 5.18]{HHR}}]
\label{RealOrientedToPolynomial}
 If $E$ is a real oriented ring $C_2$-spectrum, then there is a weak equivalence of ring spectra
\[
MU_\RR \smashy E \simeq E \smashy \underset{i \geq 1}{\bigwedge} S^0[S^{i\rho}]
\]
\noindent where $S^0[S^{i\rho}] = \underset{j \geq 0}{\bigvee} (S^{i\rho})^j$.
\end{corollary}

Let $\uF$ be the $C_p$-Mackey field such that $\uF(C_p/C_p) = k$ and $\uF(C_p/e) = 0$. Let us consider which finite fields $k$ make $H\underline{F}$ a real oriented spectrum. First we need the following definitions and results.

\begin{definition} \label{Concentrated} A $C_p$-spectrum $E$ is \emph{concentrated over} $C_p$ if $\underline{\pi}_\star^{C_p}(E) \neq \underline{0}$. Similarly, $E$ is \emph{concentrated over} $e$ if $\underline{\pi}^H_\star(E)$ is not $\underline{0}$ for any $H \leq C_p$.
\end{definition}

\begin{proposition}[{\cite[Remark 3.7]{Oruc}}]  \label{Determined&Concentrated} 
For any $C_p$-Mackey field $\underline{F}$ such that $\uF(C_p/e) = 0$, the Eilenberg-Mac Lane spectrum $H\underline{F}$ is concentrated over $C_p$.
\end{proposition}

\begin{definition}
    [{\cite[Definition 6.10]{HillSlicePrimer}}] A $C_p$-spectrum $E$ is \emph{geometric} if $\underline{\pi}_\star(E)(C_p/e) = 0$.
\end{definition}

We will use geometric spectra in this section and in \Cref{More Computations}. There is a particularly useful property that geometric spectra have. In order to state this property, we first need the following definition.

\begin{definition}[{\cite[II.2.10]{LMS}}]
Let $N$ be a normal subgroup of $G$. Denote $\P_N$ as the family of proper subgroups of $N$. Let $E\P_N$ denote the \emph{classifying space} of $\P_N$ such that $E\P_N^H$ is empty and for any proper subgroup $H$ of $G$, $E\P_N^H$ is weakly contractible. 

We will always assume that $E\P_N$ is a $G$-CW complex and let $\tilde E\P_N$ be the mapping cone of $E\P_N \to *$. For ease of notation, we will write $\tilde E\P$ for $\tilde E\P_G$.
\end{definition}

\begin{proposition}[{\cite[Proposition 6.11]{HillSlicePrimer}}] \label{GeometricSpectra} Let $E$ be a $C_p$-spectrum. Then $E$ is a geometric spectrum if and only if the natural map
\[
E \to \widetilde{E}\mathcal{P} \smashy E
\]
\noindent is a $C_p$-weak equivalence.
\end{proposition}

The following two propositions are classical results which we will use in our next proof.

\begin{proposition}[{\cite[II]{LMS}}] \label{geom=fixed}
For $E$ a $C_p$-spectrum concentrated over $C_p$ we have
\begin{center}
    $\Phi^{C_p}(E) \simeq E^{C_p}.$
\end{center}
\end{proposition}

Let $E$ and $D$ be $C_p$-spectra. We let $[D,E]^{C_p}$ denote the homotopy classes of maps of $C_p$-spectra.

\begin{proposition}[{\cite[II.9.2-6]{LMS}}] \label{SwitchEquivariance}
Let $E$ and $D$ be $C_p$-spectra. If $E$ is concentrated over $C_p$, then
\begin{center}
    $[D,E]^{C_p} \cong [\Phi^{C_p}(D), E^{C_p}]$.
\end{center}

\end{proposition}

We are now ready to prove the following proposition.

\begin{proposition} \label{RealOrientedField}
Let $\underline{F}$ be the $C_2$-Mackey field where $\underline{F}(C_2/C_2) = k$ is a finite field and $\uF(C_2/e) = 0$. Then, $H\underline{F}$ is real oriented if and only if $k$ is characteristic 2.
\end{proposition}

\begin{proof}
For ease of notation let $E \coloneqq H \underline{F}$. Using \Cref{EquivariantHomology}, we have
\begin{align*}
    \underline{ \widetilde{ E}}^{\rho}(\CC \PP^\infty)(C_2/C_2) & \cong [ \Sigma^\infty \CC\PP^\infty, \Sigma^{\rho}E]^{C_2}(C_2/C_2) \\
    & \cong [\Sigma^\infty(C_2/C_2)_+ \smashy \Sigma^\infty \CC\PP^\infty, \Sigma^{\rho}E]^{C_2} \\
    & \cong [ \Sigma^\infty \CC\PP^\infty, \Sigma^{\rho}E]^{C_2}.
\end{align*}
By \Cref{Determined&Concentrated}, $E$ is concentrated over $C_2$, which means that $\underline{\pi}_\star^{C_2}(E) \neq 0$. This implies that $S^\rho \smashy E$ is also concentrated over $C_2$. Using \Cref{geom=fixed}, \Cref{SwitchEquivariance}, and the properties of geometric fixed points, we can continue our calculation in the following way:
\begin{align*}
    [ \Sigma^\infty \CC\PP^\infty, \Sigma^{\rho}E]^{C_2} & \cong [ \Phi^{C_2}(\Sigma^\infty \CC\PP^\infty), (S^\rho \smashy E)^{C_2}]^{e} \\
     & \cong [ \Sigma^\infty ((\CC\PP^\infty)^{C_2}), \Phi^{C_2}(S^\rho \smashy E) ]^e \\
     & \cong [ \Sigma^\infty \RR\PP^\infty, S^1 \smashy \Phi^{C_2}(E) ]^e.
\end{align*}
\noindent Since $[\Sigma^\infty \RR\PP^\infty, S^1 \smashy \Phi^{C_2}(E)]^e$ is non-equivariant, we can use \Cref{phiHR} to state
\begin{align*}
    [\Sigma^\infty \RR\PP^\infty, S^1 \smashy \Phi^{C_2}(E)]^e & \cong [\Sigma^\infty \RR\PP^\infty, S^1 \smashy Hk]^e \\
    & \cong \widetilde{ H}^1(\RR\PP^\infty ; k).
\end{align*}
\noindent By a similar argument, we have that
\begin{center}
    $\underline{ \widetilde{ E}}^{\rho}(\CC \PP^1)(C_2/C_2) \cong \widetilde{ H}^1(\RR\PP^1 ; k)$.
\end{center}

It is a classical computation that $H^1(\RR\PP^1 ; k) = k$. Using the Universal Coefficient Theorem we have that $H^1(\RR\PP^\infty; k) \cong \Hom(\FF_2, k)$ is $k$ when the characteristic of $k$ is $2$, and $0$ else. If $k$ is not characteristic $2$ then there exists no $\bar x \in H^1(\RR\PP^\infty ; k)$ that maps to the unit in $k$, so $E$ would not be real oriented.

Therefore let us assume that $k$ is characteristic $2$. Recall that the inclusion map $\CC\PP^1 \to \CC\PP^\infty$ induces a restriction map, $f\colon \underline{\widetilde{ E}}^{\rho}(\CC \PP^\infty)(C_2/C_2) \to \underline{\widetilde{ E}}^{\rho}(\CC \PP^1)(C_2/C_2)$. Similarly, the inclusion $\RR\PP^1 \to \RR\PP^\infty$ induces a restriction map, $h\colon \underline{\widetilde{ E}}^{\rho}(\RR \PP^\infty)(C_2/C_2) \to \underline{\widetilde{ E}}^{\rho}(\RR \PP^1)(C_2/C_2)$. These maps and the above isomorphisms give the following commutative diagram by naturality:

% https://q.uiver.app/#q=WzAsNCxbMCwwLCJbXFxTaWdtYV5cXGluZnR5IFxcQ0MgXFxQUF5cXGluZnR5LCBcXFNpZ21hXlxccmhvIEVdXntDXzJ9ICJdLFsxLDAsIltcXFNpZ21hXlxcaW5mdHkgXFxSUiBcXFBQXlxcaW5mdHksIFxcU2lnbWEgSGtdIl0sWzAsMSwiW1xcU2lnbWFeXFxpbmZ0eSBcXENDIFxcUFBeMSwgXFxTaWdtYV5cXHJobyBFXV57Q18yfSJdLFsxLDEsIltcXFNpZ21hXlxcaW5mdHkgXFxSUiBcXFBQXjEsIFxcU2lnbWEgSGtdLiJdLFswLDIsImYiLDJdLFsxLDMsImgiXSxbMCwxLCJcXGNvbmciLDEseyJzdHlsZSI6eyJib2R5Ijp7Im5hbWUiOiJub25lIn0sImhlYWQiOnsibmFtZSI6Im5vbmUifX19XSxbMiwzLCJcXGNvbmciLDEseyJzdHlsZSI6eyJib2R5Ijp7Im5hbWUiOiJub25lIn0sImhlYWQiOnsibmFtZSI6Im5vbmUifX19XV0=
\[\begin{tikzcd}
	{[\Sigma^\infty \CC \PP^\infty, \Sigma^\rho E]^{C_2} } & {[\Sigma^\infty \RR \PP^\infty, \Sigma Hk]^e \cong k} \\
	{[\Sigma^\infty \CC \PP^1, \Sigma^\rho E]^{C_2}} & {[\Sigma^\infty \RR \PP^1, \Sigma Hk]^e \cong k.}
	\arrow["\cong"{description}, draw=none, from=1-1, to=1-2]
	\arrow["f"', from=1-1, to=2-1]
	\arrow["h", from=1-2, to=2-2]
	\arrow["\cong"{description}, draw=none, from=2-1, to=2-2]
\end{tikzcd}\]
\noindent Since this diagram commutes and $h$ is an isomorphism, then so is $f$. Therefore $f$ sends the unit to the unit, meaning $E$ is real oriented.
\end{proof}

\subsection{The computation}

Now that we know some examples of $C_2$-Mackey fields $\uF$ which have an Eilenberg-Mac Lane spectrum that is real oriented, we can use \Cref{RealOrientedToPolynomial} to obtain the following result.

\begin{lemma} \label{MUR}
Let $\underline{F}$ be a $C_2$-Mackey field such that $\uF(C_2/e) = 0$. If $H\underline{F}$ is real oriented, then $H\underline{F}_\star(MU_\RR)$ is a free $H\underline{F}_\star$-module, that is,
\begin{center}
    $H\underline{F}_\star(MU_\RR) \cong H\underline{F}_\star [b_1, b_2, \ldots ]$
\end{center}
\noindent where $\deg(b_i) = i\rho$.

\end{lemma}

\begin{proof}
Since $H\underline{F}$ is real oriented, we can use \Cref{RealOrientedToPolynomial} to show that
\begin{center}
    $MU_\RR \smashy H\underline{F} \simeq H\underline{F} \smashy \underset{i \geq 1}{\bigwedge} S^0[S^{i\rho}]$
\end{center}
\noindent which gives an isomorphism of $RO(C_2)$-graded Green functors
\begin{center}
$\underline{\pi}_\star(MU_\RR \smashy H\underline{F}) \cong \underline{\pi}_\star(H\underline{F})[b_1, b_2, \ldots]$
\end{center}
\noindent for $\deg(b_i) = i\rho$. \end{proof}

There is a classical standard argument which is a result of Cartan and Eilenberg's Theorem X.6.1 in \cite{Cartan-Eilenberg}. The argument is that if $k$ is a commutative ring, and $A$ is a commutative $k$-algebra that is flat as a module over $k$, then
\[
\Tor_*^{A^e}(A,A) \cong A \otimes_k \Tor_*^A(k,k).
\]

\noindent Using the homological algebra from \cite{LM} we can extend Cartan and Eilenberg's argument to the equivariant setting. That is, if $\uR_\star$ is a commutative $G$-Green functor, and $\uM_\star$ is a commutative $\uR_\star$-algebra that is flat as a module over $\uR_\star$, then
\[
\underline{\Tor}_{*,\star}^{\uM_\star \, \square_{\uR_\star} \, \uM_\star^{\text{op}}}(\uM_\star, \uM_\star) \cong \uM_\star \, \square_{\uR_\star} \, \underline{\Tor}_{*,\star}^{\uM_\star}(\uR_\star, \uR_\star).
\]

\noindent We will use this in the following calculation.

Let $\uF$ be the $C_2$-Mackey field where $\underline{F}(C_2/C_2) = \FF_2$, and $\uF(C_2/e) = 0$, then the results in \Cref{the trivial action} and \Cref{RealOrientedField} demonstrate that  $H\underline{F}$ has a trivial $C_2$-action and is real oriented.

\begin{theorem}\label{thm:MUR computation}

For $\underline{F}$ the $C_2$-Mackey field where $\underline{F}(C_2/C_2) = \FF_{2}$, and $\uF(C_2/e) = 0$,
\[
\underline{H}_\star(\THH_{C_2}(MU_\RR) ; \underline{F}) \cong H\underline{F}_\star [b_1, b_2, \ldots] \, \square_{H\underline{F}_{\star}} \, \mathlarger{\mathlarger{\Lambda}}_{H\underline{F}_{\star}}(z_1, z_2, \ldots) 
\]
\noindent as an $H\underline{F}_{\star}$-module. Here $|b_i| = i\rho$ and $|z_i| = 1+i\rho$. 

\end{theorem}

\begin{proof}

\Cref{RealOrientedField} shows that $H\underline{F}$ has a trivial $C_2$-action. In order to use the B\"okstedt spectral sequence, we need to show that $H\underline{F}_\star(MU_\RR)$ is flat over $H\underline{F}_\star$. The following isomorphism of $RO(C_2)$-graded Green functors is given by \Cref{MUR}:
\[  \underline{\pi}_\star(H\underline{F} \smashy  MU_\RR) \cong H\underline{F}_{\star}[ b_1, b_2, \ldots ]
\]
\noindent where $\deg(b_i) = i\rho$. Therefore $H\underline{F}_\star(MU_\RR)$ is flat over $H\underline{F}_\star$. Since $H\uF$ has a trivial $C_2$-action, then ${}^\gamma H\underline{F}_\star(MU_\RR)$ has the same left $H\underline{F}_\star$-module action as $H\underline{F}_\star(MU_\RR)$. Since the appropriate conditions hold, we can use the twisted B\"okstedt spectral sequence (see \Cref{EquivBokSS}):
\[
    E_{s,\star}^2 = \underline{\HH}_s^{H\underline{F}_{\star},C_2}(\underline{H}_\star(MU_\RR ; \underline{F})) \Rightarrow \underline{H}_{s + \star}(i_{C_2}^* \THH_{C_2}(MU_\RR) ; \underline{F}).
\]
\noindent The $E_2$-term is 
\begin{align*}
    E_{s,\star}^2 & = \underline{\HH}_s^{H\underline{F}_{\star},C_2}(\underline{H}_\star(MU_\RR ; \underline{F})) \\
    & \cong \underline{\HH}_s^{H\underline{F}_{ \star},C_2}(H\underline{F}_{\star} [b_1, b_2, \ldots]) & & \text{\Cref{MUR}} \\
    & \cong \underline{\Tor}_{*,\star}^{H\underline{F}_{\star} [b_1, \ldots]^e}(H\underline{F}_{\star} [b_1, \ldots], {}^\gamma H\underline{F}_{\star} [b_1, \ldots]) & & \text{\Cref{HH and Tor}} \\
    & \cong \underline{\Tor}_{*,\star}^{H\underline{F}_{\star} [b_1, \ldots]^e}(H\underline{F}_{\star} [b_1, \ldots], H\underline{F}_{\star} [b_1, \ldots]) \\
     & \cong H\underline{F}_{\star} [b_1, \ldots] \, \square_{H\underline{F}_{\star}} \, \underline{\Tor}_{*,\star}^{H\underline{F}_{\star} [b_1, \ldots]}(H\underline{F}_{\star}, H\underline{F}_{\star}) & & \text{\cite{Cartan-Eilenberg}} \\
   & \cong H\underline{F}_\star [b_1, b_2, \ldots] \, \square_{H\underline{F}_{\star}} \, \mathlarger{\mathlarger{\Lambda}}_{H\underline{F}_{\star}}(z_1, z_2, \ldots)
\end{align*}
\noindent the last isomorphism is a similar computation to the classical one using the Koszul complex, where $\deg(b_i) = (0,i\rho)$ and $\deg(z_i) = (1,i\rho)$.

Recall that $d^r\colon E^r_{s,\alpha} \to E^r_{s-r, \alpha + r - 1}$. Our spectrum $MU_\RR$ is commutative, so by \cite[Proposition 4.2.8]{AGHKK} we can view this as a spectral sequence of $H\underline{F}_\star$-algebras. Consider the differential $d^2$. We know that all the $d^2$ differentials are determined by what the differential does on the generators of the $E^2$ page, thus since the only generators are in the columns where $s = 0, 1$ then all of the differentials on the $E^2$-page are zero and the spectral sequence collapses.
\end{proof}

\section{Algebraic structure on twisted THH}\label{Sec:Algebraic structure}

In this section we will demonstrate that for $R$ a commutative $C_p$-ring spectrum, $\THH_{C_p}(R)$ is a commutative $R$-algebra for any prime $p$ in the category of $C_p$-spectra. We will then discuss the lack of coproduct structure on $\THH_{C_p}(R)$. First, we must discuss the simplicial model of twisted THH.

\subsection{Simplicial model of twisted THH}\label{SubSec: simplicial structures}

In this section, we will explore a simplicial model for $C_p$-twisted THH. Let $R$ be a commutative ring $C_p$-spectrum and recall that, by convention, our spectra are orthogonal spectra. For this introduction we will suppress the change of universe functors for ease of notation. We start this section by constructing $C_p$-equivariant simplicial models of the circle $pS^1_\bullet$. \Cref{pS^1 gives cyclic bar} tells us that this simplicial model of the circle gives the equivalence $\THH_{C_p}(R) \simeq R \otimes_{C_p} S^1$ from \cite{twTHH}. To define the product structure, we define $m\THH_{C_p}(R)$ to be $|R \otimes_{C_p} mpS^1_\bullet|$, and demonstrate that $m\THH_{C_p}(R)$ is $C_p$-weakly equivalent to $\THH_{C_p}(R)$ in \Cref{2THHtoTHH}. 

Recall the following simplicial model of $S^1$ from \cite{Loday}, where $C_{n+1} = \{ 1, \gamma, \ldots, \gamma^{n} \}$ indicates the number of elements on each level.
% https://q.uiver.app/#q=WzAsNyxbMSwxLCJDXzIiXSxbMiwxXSxbMCwxXSxbMSwwLCJcXG92ZXJzZXR7XFx1bmRlcnNldHt9e1xcdmRvdHN9fXtDXzN9Il0sWzIsMF0sWzAsMF0sWzEsMiwiZSJdLFswLDMsInNfMCIsMSx7Im9mZnNldCI6LTV9XSxbMywwLCJkXzEiLDFdLFswLDMsInNfMSIsMSx7Im9mZnNldCI6NX1dLFs0LDEsImRfMiIsMSx7ImxhYmVsX3Bvc2l0aW9uIjo0MCwib2Zmc2V0Ijo1LCJzaG9ydGVuIjp7InNvdXJjZSI6MTAsInRhcmdldCI6NDB9fV0sWzUsMiwiZF8wIiwxLHsibGFiZWxfcG9zaXRpb24iOjQwLCJvZmZzZXQiOi01LCJzaG9ydGVuIjp7InNvdXJjZSI6MTAsInRhcmdldCI6NDB9fV0sWzAsNiwiZF8xIiwxLHsib2Zmc2V0IjotNX1dLFs2LDAsInNfMCIsMV0sWzAsNiwiZF8wIiwxLHsib2Zmc2V0Ijo1fV1d
\[\begin{tikzcd}
	{} & {\overset{\underset{}{\vdots}}{C_3}} & {} \\
	{} & {C_2} & {} \\
	& e
	\arrow["{d_0}"{description, pos=0.4}, shift left=5, shorten <=1pt, shorten >=6pt, from=1-1, to=2-1]
	\arrow["{d_1}"{description}, from=1-2, to=2-2]
	\arrow["{d_2}"{description, pos=0.4}, shift right=5, shorten <=1pt, shorten >=6pt, from=1-3, to=2-3]
	\arrow["{s_0}"{description}, shift left=5, from=2-2, to=1-2]
	\arrow["{s_1}"{description}, shift right=5, from=2-2, to=1-2]
	\arrow["{d_1}"{description}, shift left=5, from=2-2, to=3-2]
	\arrow["{d_0}"{description}, shift right=5, from=2-2, to=3-2]
	\arrow["{s_0}"{description}, from=3-2, to=2-2]
\end{tikzcd}\]
\noindent Let us call this model $S^1_\bullet$. The face and degeneracy maps are as follows:

\begin{center}
\begin{tabular}{lcl}
    $d_i(\gamma^{j}) = \begin{cases} \gamma^j & j \leq i < n \\
    \gamma^{j - 1} & j > i
    \end{cases}$ & & $s_i(\gamma^{j}) = \begin{cases} \gamma^{j} & j \leq i \\
    \gamma^{j+1} & j > i.
    \end{cases}$ \\
    $d_n(\gamma^{j}) = \begin{cases} \gamma^{j} & j < n \\
    1 & j = n
    \end{cases}$
\end{tabular}
\end{center}

\noindent There are also maps $t\colon C_{n + 1} \to C_{n + 1}$ such that $t(\gamma^{j})$ is $\gamma^{j+1}$ for $j < n$ and $1$ for $j = n$. It is notable that $d_n = d_0 \circ t\colon C_{n+1} \to C_{n+1}$ for all $n$.

We will need additional models of the circle for our proofs and will construct these using the simplicial edgewise subdivision functor defined by B\"okstedt, Hsiang, and Madsen in \cite{Bokstedt-Hsiang-Madsen(Spaces)}. The simplicial $r$-fold edgewise subdivision functor, $sd_r(-)$, is defined so that for a simplicial object $X_\bullet$,

\begin{center}
    \begin{tabular}{l}
    $sd_r(X_\bullet)_n = X_{(n + 1)r - 1}$
    \end{tabular}
\end{center}

\noindent with face and degeneracy maps $\bar d_i$ and $\bar s_i$ defined by

\begin{center}
    \begin{tabular}{l}
    $\bar d_i = d_i \circ d_{i + n + 1} \circ \ldots \circ d_{i + (r - 1)(n+1)}$ \\
    $\bar s_i = s_{i + (r-1)(n+2)} \circ \ldots \circ s_{i + (n + 2)} \circ s_i$
    \end{tabular}
\end{center}

\noindent for $d_i$ and $s_i$ the face and degeneracy maps of the simplicial object $X_\bullet$.

\begin{remark}\label{dn=d0t}
    Recall in $S^1_\bullet$ there is an equality of maps $d_n = d_0 \circ t\colon C_{n+1} \to C_{n+1}$. It is also true that in  $sd_r(S^1_\bullet)$, $\bar d_n = \bar d_0 \circ t$. To see this, consider that $\bar d_0 = d_0 \circ d_{n+1} \circ \ldots \circ d_{(r-1)n + r - 1}$ and using the simplicial relation $d_i \circ d_j = d_{j-1} \circ d_i$ (for $i < j$), we can move $d_0$ to the front and get that $\bar d_0 = d_n \circ d_{2n+1} \circ \ldots \circ d_{(r-1)n + r - 2} \circ d_0$. Consider 
        \begin{align*}
            \bar d_0 \circ t &= d_n \circ d_{2n+1} \circ \ldots \circ d_{(r-1)n + r - 2} \circ d_0 \circ t \\
            &= d_n \circ d_{2n+1} \circ \ldots \circ d_{(r-1)n + r - 2} \circ d_{rn+r-1}
        \end{align*}
\noindent which is $\bar d_n$ by definition.
\end{remark}

Let us start by understanding the $2$-fold edgewise simplicial subdivision of the circle, $sd_2(S^1_\bullet)$.

\begin{example}
Let us refer to $sd_2(S^1_\bullet)$ as $2S^1_\bullet$. By definition, 
\begin{align*}
    sd_2(S^1_\bullet)_n &= S^1_{2n +1}, \\
    \bar d_i &= d_i \circ d_{i + n + 1}, \quad \text{and} \\
    \bar s_i &= s_{i + n + 2} \circ s_i.
\end{align*}
\noindent Therefore $2S^1_\bullet$ is
% https://q.uiver.app/#q=WzAsNyxbMSwxLCJDXzQiXSxbMiwxXSxbMCwxXSxbMSwwLCJcXG92ZXJzZXR7XFx1bmRlcnNldHt9e1xcdmRvdHN9fXtDXzZ9Il0sWzIsMF0sWzAsMF0sWzEsMiwiQ18yIl0sWzAsMywiXFxiYXIgc18wIiwxLHsib2Zmc2V0IjotNX1dLFszLDAsIlxcYmFyIGRfMSIsMV0sWzAsMywiXFxiYXIgc18xIiwxLHsib2Zmc2V0Ijo1fV0sWzQsMSwiXFxiYXIgZF8yIiwxLHsibGFiZWxfcG9zaXRpb24iOjQwLCJvZmZzZXQiOjUsInNob3J0ZW4iOnsic291cmNlIjoxMCwidGFyZ2V0Ijo0MH19XSxbNSwyLCJcXGJhciBkXzAiLDEseyJsYWJlbF9wb3NpdGlvbiI6NDAsIm9mZnNldCI6LTUsInNob3J0ZW4iOnsic291cmNlIjoxMCwidGFyZ2V0Ijo0MH19XSxbMCw2LCJcXGJhciBkXzEiLDEseyJvZmZzZXQiOi01fV0sWzYsMCwiXFxiYXIgc18wIiwxXSxbMCw2LCJcXGJhciBkXzAiLDEseyJvZmZzZXQiOjV9XV0=
\[\begin{tikzcd}
	{} & {\overset{\underset{}{\vdots}}{C_6}} & {} \\
	{} & {C_4} & {} \\
	& {C_2.}
	\arrow["{\bar d_0}"{description, pos=0.4}, shift left=5, shorten <=1pt, shorten >=6pt, from=1-1, to=2-1]
	\arrow["{\bar d_1}"{description}, from=1-2, to=2-2]
	\arrow["{\bar d_2}"{description, pos=0.4}, shift right=5, shorten <=1pt, shorten >=6pt, from=1-3, to=2-3]
	\arrow["{\bar s_0}"{description}, shift left=5, from=2-2, to=1-2]
	\arrow["{\bar s_1}"{description}, shift right=5, from=2-2, to=1-2]
	\arrow["{\bar d_1}"{description}, shift left=5, from=2-2, to=3-2]
	\arrow["{\bar d_0}"{description}, shift right=5, from=2-2, to=3-2]
	\arrow["{\bar s_0}"{description}, from=3-2, to=2-2]
\end{tikzcd}\]
\noindent The only nondegenerate elements are $1,\gamma \in C_2$ and $\gamma,\gamma^3 \in C_4$, where the boundary of the $1$-cell $\gamma$ is defined by

\begin{center}
    \begin{tabular}{l}
    $\bar d_0(\gamma) = 1$ and $\bar d_1(\gamma) = \gamma,$
    \end{tabular}
\end{center}

\noindent and the boundary of the $1$-cell $\gamma^3$ is defined by

\begin{center}
    \begin{tabular}{l}
    $\bar d_0(\gamma^3) = \gamma$ and $\bar d_1(\gamma^3) = 1$.
    \end{tabular}
\end{center}

\noindent If we refer to the non-degenerate $0$-cells as $v_0$ and $v_1$, then $2S^1_\bullet$ looks like:

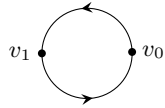
\begin{center}
\begin{tikzpicture}[>=stealth]
    \draw[ 
        decoration={
            markings, 
            mark=at position 0.005 with {\fill circle (1.5pt) node[right] {\footnotesize$v_0$};}, 
            mark=at position 0.27 with {\arrow[thick]{>}},
            mark=at position 0.5 with {\fill circle (1.5pt) node[left] {\footnotesize$v_1$};}, 
            mark=at position 0.77 with {\arrow[thick]{>}},
        },
        postaction={decorate}
    ]
    (3,0) circle (0.6);
\end{tikzpicture}
\end{center}

\noindent where the $C_2$-action on $2S^1_\bullet$ is induced from applying the functor $sd_2(-)$ to $S^1_\bullet$. This $C_2$-action on $2S^1_k = C_{2k + 2}$ sends $\gamma^i$ to $\gamma^{j}$ where $j = i + k + 1 \mod 2k+2)$. Therefore the $C_2$-action on $2S^1_\bullet$ is counter clockwise rotation by $\pi$.

\end{example}

If we consider $2S^1_\bullet$ non-equivariantly, it is not the same as $dS^1_\bullet$ as defined in \cite{Angeltveit-Rognes}. Non-equivariantly, $2S^1_\bullet$ is equivalent to $d'S^1_\bullet$ as defined in \cite[Remark 3.6]{Angeltveit-Rognes}. 

We can use this process to define $mS^1_\bullet$ for any positive integer $m$, which will have the $C_m$-action of counter clockwise rotation by $2\pi/m$. Notice that in order for $mS^1_\bullet$ to have a simplicial $C_n$-action of counter clockwise rotation by $2\pi/n$ then $m$ must be a multiple of $n$. Consider two examples of $C_3$-equivariant simplicial models of the circle; $3S^1_\bullet$ and $6S^1_\bullet$:

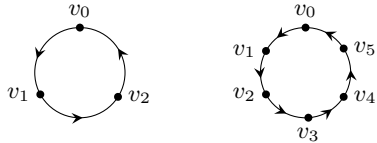
\centerline{\begin{tikzpicture}[>=stealth]
    \draw[ 
        decoration={markings,
        mark=at position 0.09 with {\arrow[thick]{>}},
        mark=at position 0.25 with {\fill circle (1.5pt) node[above] {\footnotesize$v_0$};}, 
        mark=at position 0.45 with {\arrow[thick]{>}}, 
        mark=at position 0.58 with {\fill circle (1.5pt) node[left] {\footnotesize$v_1$};}, 
        mark=at position 0.76 with {\arrow[thick]{>}},
        mark=at position 0.91 with {\fill circle (1.5pt) node[right] {\footnotesize$v_2$};}, 
        },
        postaction={decorate}
        ]
        (1,0) circle (0.6);
    \draw[ 
        decoration={markings, 
        mark=at position 0.01 with {\arrow[thick]{>}},
        mark=at position 0.09 with {\fill circle (1.5pt) node[right] {\footnotesize$v_5$};},
        mark=at position 0.18 with {\arrow[thick]{>}},
        mark=at position 0.25 with {\fill circle (1.5pt) node[above] {\footnotesize$v_0$};}, 
        mark=at position 0.36 with {\arrow[thick]{>}},
        mark=at position 0.42 with {\fill circle (1.5pt) node[left] {\footnotesize$v_1$};}, 
        mark=at position 0.52 with {\arrow[thick]{>}},
        mark=at position 0.58 with {\fill circle (1.5pt) node[left] {\footnotesize$v_2$};}, 
        mark=at position 0.68 with {\arrow[thick]{>}},
        mark=at position 0.76 with {\fill circle (1.5pt) node[below] {\footnotesize$v_3$};},
        mark=at position 0.85 with {\arrow[thick]{>}},
        mark=at position 0.91 with {\fill circle (1.5pt) node[right] {\footnotesize$v_4$};}, 
        },
        postaction={decorate}
        ]
        (4,0) circle (0.6);
\end{tikzpicture}}

\noindent where the induced $C_3$-action on both of these simplicial objects is counter clockwise rotation by $2\pi/3$.

In \Cref{Sec:TwTHH} we recalled the definition of twisted THH as defined by Angeltveit, Blumberg, Gerhardt, Hill, Lawson, and Mandell in \cite{twTHH}. We also discussed the different perspectives these authors gave us on twisted THH including the geometric realization of the twisted cyclic bar construction, and a tensor product. Namely, for $U$ a complete $S^1$-universe, $\widetilde{U} = \iota_{C_p}^*U$, and $R$ a commutative ring $C_p$-spectrum indexed over $\widetilde{U}$, 
\[
\THH_{C_p}(R) \simeq \I_{\RR^\infty}^{U}(\I_{\widetilde{U}}^{\RR^\infty}(R) \otimes_{C_p} S^1).
\]
To simplify notation, we will often let $R$ be a commutative ring $C_p$-spectrum indexed on the trivial universe $\RR^\infty$. To generalize one can consider a commutative ring $C_p$-spectrum $\widetilde{R}$ indexed over $\widetilde{U}$ and let $R = \I_{\widetilde{U}}^{\RR^\infty} \widetilde{R}.$ The following proposition demonstrates which simplicial model of the circle is suitable for this tensor product perspective.

For the rest of this section, let the chosen generator of the group $C_p$ be $\gamma \coloneqq e^{2\pi i/p}$.

\begin{proposition} \label{pS^1 gives cyclic bar}
Let $R$ be a commutative ring $C_p$-spectrum indexed on the trivial universe $\RR^\infty$, for $p$ prime. Then $R \otimes_{C_p} pS^1_\bullet \cong B^{cy, C_p}_\bullet(R)$, the $C_p$-twisted cyclic bar construction.
\end{proposition}

\begin{proof}

Let $\mu$ and $\eta$ be the multiplication and unit maps of $R$ respectively. To show that these simplicial objects are equivalent we will first show that every level is the same and then we will show that they have equivalent face and degeneracy maps.

The $k$-simplicies of $R \otimes_{C_p} pS^1_\bullet$ are defined by the following coequalizer diagram
% https://q.uiver.app/#q=WzAsMyxbMCwwLCIoXFxJX1Vee1xcUlJeXFxpbmZ0eX0gUikgXFxvdGltZXMgQ19wIFxcb3RpbWVzIHBTXjFfayJdLFsyLDAsIihcXElfVV57XFxSUl5cXGluZnR5fSBSKSBcXG90aW1lcyBwU14xX2siXSxbNCwwLCIoXFxJX1Vee1xcUlJeXFxpbmZ0eX0gUikgXFxvdGltZXNfe0NfcH0gcFNeMV9rIl0sWzAsMSwiXFxpZCBcXG90aW1lcyByIiwwLHsib2Zmc2V0IjotMn1dLFswLDEsIlxcZWxsIFxcb3RpbWVzIFxcaWQiLDIseyJvZmZzZXQiOjJ9XSxbMSwyXV0=
\[\begin{tikzcd}
	{R \otimes C_p \otimes pS^1_k} && {R \otimes pS^1_k} && {R \otimes_{C_p} pS^1_k}
	\arrow["{\id \otimes r}", shift left=2, from=1-1, to=1-3]
	\arrow["{\ell \otimes \id}"', shift right=2, from=1-1, to=1-3]
	\arrow[from=1-3, to=1-5]
\end{tikzcd}\]
\noindent where the map $r$ is the $C_p$-action on $pS^1_k$ and $\ell$ is the induced $C_p$-action on $R$.

Let $C_p = \langle \gamma \rangle$, and $pS^1_k = C_{pk + p} = \{1, x, \ldots, x^{pk + p - 1} \}$. The induced $C_p$-action on the set of elements $C_{pk + p}$ is defined by $\gamma x^i = x^{j}$ such that $j = i+k+1 \mod pk + p)$. There is a $C_{pk + p}$-action on $pS^1_k$ induced by $t\colon C_{pk + p} \to C_{pk + p}$ defined by $t(x^i) = x^{j}$ such that $j = i + 1 \mod pk + p)$.

As $C_p$-sets, $C_p \otimes pS^1_k = C_p \times C_{pk+p}$. Thus $R \otimes C_p \otimes pS^1_k$ can be written as 
\[
\underset{t = 0}{\overset{p-1}{\bigwedge}}( \underset{s = 0}{\overset{pk + p - 1}{\bigwedge}} R_{x^s,\gamma^t}). 
\]
For ease of notation let us write $R_{s,t}$ instead of $R_{x^s,\gamma^t}$. Similarly, $R \otimes pS^1_k$ can be written as 
\[
\underset{s = 0}{\overset{pk + p - 1}{\bigwedge}} R_{s}.
\]
With this notation, $\id \otimes r\colon R_{s,t} \mapsto R_{j}$ such that $j = s + t(k + 1) \mod pk + p)$, and $\ell \otimes \id \colon R_{s,t} \mapsto {}_{\gamma^t}R_{s}$ where ${}_{\gamma^t}R$ indicates $R$ which has been acted on by $\gamma^t$. By definition of the coequalizer, $R \otimes_{C_p} pS^1_k$ is the quotient space of $R \otimes pS^1_k$ where the quotient forces these two actions to agree. Recall that $\gamma^t x^s = x^j$ for $j = s + t(k + 1) \mod pk + p)$. Therefore $R \otimes_{C_p} pS^1_k = R^{\smashy k + 1}$.

We will now show that the face and degeneracy maps from $R \otimes_{C_p} pS^1_k$ are equivalent to the face and degeneracy maps from $B_k^{cy,C_p}(R)$. We will start by considering the face and degeneracy maps of $pS^1_k$ and induce the corresponding face and degeneracy maps of $R \otimes_{C_p} pS^1_{k}$.

Recall that $pS^1_k  = sd_p(S^1_\bullet)_k$ with face and degeneracy maps $\bar d_i$ and $\bar s_i$ defined as follows:

\begin{center}
    \begin{tabular}{l}
    $\bar d_i = d_i \circ d_{i + k + 1} \circ \ldots \circ d_{i + (p - 1)(k+1)}$ \\
    $\bar s_i = s_{i + (p-1)(k+2)} \circ \ldots \circ s_{i + (k + 2)} \circ s_i$
    \end{tabular}
\end{center}

\noindent where $d_i$ and $s_i$ are the face and degeneracy maps of $S^1_\bullet$, and $0 \leq i \leq k$. 

Let us start by finding what the induced face maps are on $R \otimes_{C_p} pS^1_k$, say $\delta_i\colon R \otimes_{C_p} pS^1_k \to R \otimes_{C_p} pS^1_{k-1}$, for $0 \leq i \leq k$ and $k > 0$. The map $(\id \otimes \bar d_i)\colon R \otimes pS^1_k \to R \otimes pS^1_{k-1}$ applies the multiplication map to $R_{i} \smashy R_{i+1}$ as well as $R_{i + n(k + 1)} \smashy R_{i + 1 + n(k + 1)}$ for all $0 \leq n < p$. Therefore $\delta_i\colon R \otimes_{C_p} pS^1_k \to R \otimes_{C_p} pS^1_{k-1}$ is the map $\id^{\smashy i} \smashy \mu \smashy \id^{\smashy k - i-1}$ for $0 \leq i < k$. 

Before figuring out what $\delta_k$ must be, recall that $pS^1_k$ has a $C_{pk+p}$-action induced by the map $t$. Consider $(\id \otimes t)\colon R \otimes pS^1_k \to R \otimes pS^1_k$, this map rotates the last copy of $R$ to the front. This map also rotates $R_k$ into the position $R_{k+1}$ was in. This is important as in the quotient $R \otimes_{C_p} pS^1_k$ we have the following two maps must be equivalent:
\[
    (\ell \otimes \id)(R_{0,1}) ={}_{\gamma}R_{0}, \quad \text{and} \quad (\id \otimes r)(R_{0,1}) = R_{k+1}.
\]
\noindent So the map that is induced on $R \otimes_{C_p} pS^1_k$ rotates the last copy of $R$ to the front and acts on that copy of $R$ by $\gamma$. Let us suggestively refer to this induced map as $\alpha_k$.

Recall from \Cref{dn=d0t} that $\bar d_k = \bar d_0 \circ t$, so the last face map $\delta_k$ is induced from $\id \otimes \bar d_k = (\id \otimes \bar d_0) \circ (\id \otimes t)\colon R \otimes pS^1_k \to R \otimes pS^1_{k-1}$. The universal property of the coequalizer shows that the maps $(\id \otimes \bar d_0)$ and $(\id \otimes t)$ induce maps on $R \otimes_{C_p} pS^1_k$, namely $\delta_0$ and $\alpha_k$ respectively. Further, by the uniqueness property, the map induced from their composition, $(\id \otimes \bar d_0) \circ (\id \otimes t)$ must be equivalent to the composition of the induced maps. Meaning, $\delta_k = \delta_0 \circ \alpha_k$. 

We similarly induce the degeneracy maps of $R \otimes_{C_p} pS^1_k$, say $\sigma_i\colon R \otimes_{C_p} pS^1_k \to R \otimes_{C_p} pS^1_{k+1}$, for $0 \leq i \leq k$ and $k \geq 0$. By a similar argument as above we can show that these can be written as $\sigma_i = \id^{\smashy i+1} \smashy \eta \smashy \id^{\smashy k - i}$.

Therefore, using \Cref{twistedcyclicbar} we can see that $R \otimes_{C_p} pS^1_\bullet$ is isomorphic to $B^{cy, C_p}_\bullet(R)$. \end{proof}

Let $U$ be a complete $S^1$-universe and $R$ a commutative ring $C_p$-spectrum indexed on $\widetilde{U} \coloneqq i_{C_p}^*(U)$, then a result of this proposition is that $|\I^{U}_{\RR^\infty}((\I_{\widetilde{U}}^{\RR^\infty}R) \otimes_{C_p} pS^1_\bullet)| \cong \THH_{C_p}(R)$. We can construct a similar structure $m\THH_{C_p}(R) \coloneqq |\I^{U}_{\RR^\infty}((\I_{\widetilde{U}}^{\RR^\infty}R) \otimes_{C_p} mpS^1_\bullet)|$.

We want an equivariant analogue of Angeltveit and Rognes' result \cite[Lemma 3.8]{Angeltveit-Rognes}, namely we want to show that $m\THH_{C_p}(R)$ is $C_p$-weakly equivalent to $\THH_{C_p}(R)$ (see \Cref{2THHtoTHH}). To obtain this we first need some definitions and lemmas. 

\begin{definition} \label{def: double bar}
    Let $R$ be a commutative ring $C_p$-spectrum indexed on the trivial universe, $p$ odd. The \emph{double bar complex} for $R$, denoted $B_\bullet(R,R,R)$, is a simplicial object such that $B_k(R,R,R) = R^{\smashy k+2}$. The face and degeneracy maps, $d_i\colon R^{\smashy k+2} \to R^{\smashy k+1}$ and $s_i\colon R^{\smashy k+2} \to R^{\smashy k+3}$ are defined as follows:
    \begin{align*}
        d_i & = \id^i \smashy \mu \smashy \id^{k - i} \\
        s_i & = \id^{\smashy i+1} \smashy \eta \smashy \id^{\smashy k - i + 1}
    \end{align*}
    \noindent for $0 \leq i \leq k$, where $\mu$ is the multiplication map and $\eta$ is the unit map for $R$.
\end{definition}

We will be considering smash products of simplicial $C_p$-spectra throughout this paper, so let us discuss what we mean. Let $X_\bullet$ and $Y_\bullet$ be two simplicial $C_p$-spectra. When we write $(X \smashy Y)_\bullet$, we mean the simplicial $C_p$-spectrum, where $(X \smashy Y)_k = X_k \smashy Y_k$ and the face and degeneracy maps from $(X \smashy Y)_k$ are the smash product of the respective face and degeneracy maps from $X_k$ and $Y_k$.

In general, for some commutative ring $C_p$-spectrum $R$ with $(R,R)$-bimodules $M$ and $N$, $M \smashy_{(R \smashy R)} N$ is the coequalizer of the bimodule action maps. Similarly, for $M$ and $N$ $R$-modules, one can construct $M \smashy_R N$ as the coequalizer of the right and left $R$-module action maps.

Let $R$ be a commutative ring $C_p$-spectrum. Consider the map 
\[
f\colon R \smashy (R \smashy 
R) \to R
\]
where $f$ rotates the last copy of $R$ to the front and then applies the $(R,R)$-bimodule map which multiplies the middle $R$ with the left and right copies of $R$. Namely, $f = \mu \circ (\id \smashy \mu) \circ t$ where $t$ rotates the last $R$ to the front. Also consider the map 
\[
h\colon (R \smashy R) \smashy B_k(R,R,R) \to B_k(R,R,R)
\]
where $h$ rotates the first copy of $R$ in $(R \smashy R)$ to the end and then applies the $(R,R)$-bimodule map on $B_k(R,R,R)$ which multiplies the first two copies of $R$ and the last two copies of $R$. Namely, $h = (\mu \smashy \id^{\smashy k}) \circ (\id^{\smashy k+2} \smashy \mu) \circ t'$ where $t'$ rotates the first copy of $R$ to the back. For the next lemma we will need the coequalizer of the following diagram in the category of spectra:
% https://q.uiver.app/#q=WzAsMixbMCwwLCJBIFxcc21hc2h5IEEiXSxbMiwwLCJBIFxcc21hc2h5IEJfayhBLEEsQSkiXSxbMCwxLCJmIiwwLHsib2Zmc2V0IjotMn1dLFswLDEsImciLDIseyJvZmZzZXQiOjJ9XV0=
\[\begin{tikzcd}
	{R \smashy (R \smashy R) \smashy B_k(R,R,R)} && {R \smashy B_k(R,R,R)}
	\arrow["f \smashy \id^{\smashy k+2}", shift left=2, from=1-1, to=1-3]
	\arrow["\id \smashy h"', shift right=2, from=1-1, to=1-3]
\end{tikzcd}\]
\noindent which we will write as $R \smashy_{(R \smashy R)} B_k(R,R,R)$.
Since $R$ and $B_k(R,R,R)$ are free $(R,R)$-bimodules, we can consider the map $R \smashy B_k(R,R,R) \to R^{\smashy k + 1}$ which multiplies the first copy of $R$ with the first and last copies of $R$ in $B_k(R,R,R)$.
This map respects the maps from the coequalizer, therefore the universal property of the coequalizer gives the reduction map: $R \smashy_{(R \smashy R)} B_k(R,R,R) \to R^{\smashy k + 1}$.

Recall from the definition of ${}^\gamma E$ for $E$ a $C_p$-spectrum (see \Cref{Sec:TwTHH}) that the spectra ${}^\gamma R$ and $R$ are the same except for their left $R$-module map. Therefore we can construct $R \smashy_{(R \smashy R)} B_k(R,R,{}^\gamma R)$ in a similar way as we constructed $R \smashy_{(R \smashy R)} B_k(R,R,R)$. 
Similarly to above, we have a reduction map: $R \smashy_{(R \smashy R)} B_k(R,R,{}^\gamma R) \to {}^\gamma R \smashy R^{\smashy k}$, since the maps in the coequalizer diagram are in such a way that the right module action of $R$ and the left module action of ${}^\gamma R$ are maintained.

Let us similarly construct, $B_k(R,R,R) \smashy_{(R \smashy R)} B_k(R,R,R)$. Consider the map
\[
k\colon B_k(R,R,R) \smashy (R \smashy R)  \to B_k(R,R,R)
\]
where $k$ rotates the first copy of $R$ in $(R \smashy R)$ to the front and then applies the $(R,R)$-bimodule map which multiplies the first two copies of $R$ and the last two copies of $R$. Namely, $k = (\mu \smashy \id^{\smashy k}) \circ (\id^{\smashy k+2} \smashy \mu) \circ t''$ where $t''$ rotates the first copy of $R$ in $(R \smashy R)$ to the front. We now define $B_k(R,R,R) \smashy_{(R \smashy R)} B_k(R,R,R)$ as the coequalizer of the following diagram in the category of spectra:
% https://q.uiver.app/#q=WzAsMixbMiwwLCJCX2soUixSLFIpIFxcc21hc2h5IEJfayhSLFIsUikuIl0sWzAsMCwiQl9rKFIsUixSKSBcXHNtYXNoeSAoUiBcXHNtYXNoeSBSKSBcXHNtYXNoeSBCX2soUixSLFIpIl0sWzEsMCwiayBcXHNtYXNoeSBcXGlkXntcXHNtYXNoeSBrKzJ9IiwwLHsib2Zmc2V0IjotMn1dLFsxLDAsIlxcaWRee1xcc21hc2h5IGsrMn0gXFxzbWFzaHkgaCIsMix7Im9mZnNldCI6Mn1dXQ==
\[\begin{tikzcd}
	{B_k(R,R,R) \smashy (R \smashy R) \smashy B_k(R,R,R)} && {B_k(R,R,R) \smashy B_k(R,R,R)}
	\arrow["{k \smashy \id^{\smashy k+2}}", shift left=2, from=1-1, to=1-3]
	\arrow["{\id^{\smashy k+2} \smashy h}"', shift right=2, from=1-1, to=1-3]
\end{tikzcd}\]
\noindent where $h$ is the same map as defined above. We construct $B_k(R,R,R) \smashy_{(R \smashy R)} B_k(R,R,{}^\gamma R)$ and $B_k(R,R,R)^{\smashy_R (n-1)} \smashy_{(R \smashy R)} B_k(R,R,{}^\gamma R)$ in a similar way, where $B_k(R,R,R) \smashy_R B_k(R,R,R)$ is the coequalizer of the right $R$-module action on the left copy of $B_k(R,R,R)$ and the left $R$-module action on the right copy of $B_k(R,R,R)$.
The reduction maps of these coequalizers are as follows: $B_k(R,R,R) \smashy_{(R \smashy R)} B_k(R,R,{}^\gamma R) \to {}^\gamma R \smashy R^{\smashy 2k+1}$, $B_k(R,R,R) \smashy_R B_k(R,R,R) \to R^{\smashy 2k + 3}$, and $B_k(R,R,R)^{\smashy_R (n-1)} \smashy_{(R \smashy R)} B_k(R,R,{}^\gamma R) \to {}^\gamma R \smashy R^{\smashy n(k+1) - 1}$.

\begin{lemma} \label{bar complexes isoms}
    Let $R$ be a commutative ring $C_p$-spectrum indexed over the trivial universe $\RR^\infty$, $p$ prime. Then we have the following isomorphisms of simplicial $C_p$-spectra:
\begin{align*}
    B_\bullet^{cy,C_p}(R) & \cong R \smashy_{(R \smashy R)} B_\bullet(R,R,{}^\gamma R), \, \text{and} \\
    sd_n(B_\bullet^{cy,C_p}(R)) & \cong (B(R,R,R)^{\smashy_R (n-1)} \smashy_{(R \smashy R)} B(R,R,{}^\gamma R))_\bullet.
\end{align*}
\end{lemma}

\begin{proof}
Let $\mu$ and $\eta$ be the multiplication and unit maps of $R$ respectively, and let ${}^\gamma \mu$ be the left $R$-module structure map of ${}^\gamma R$ as defined in \Cref{Sec:TwTHH}.

The first isomorphism is very similar to the classical result that $B_\bullet^{cy}(R) \cong R \smashy_{(R \smashy R)} B_\bullet(R,R,R)$. The arguments for each isomorphism are similar, we will prove the second isomorphism as the proof requires more details and the proof of the first isomorphism is essentially a simplification of the proof of the second isomorphism.

For the second isomorphism we will start by defining a simplicial ring $C_p$-spectrum $X_\bullet$, which we will show is the image of the reduction map from $(B(R,R,R)^{\smashy_R (n-1)} \smashy_{(R \smashy R)} B(R,R,{}^\gamma R))_\bullet$ (discussed above), then we will demonstrate that $sd_n(B_\bullet^{cy, C_p}(R)) \cong X_\bullet$ which will finish the argument.

Let $X_\bullet$ be a simplicial ring $C_p$-spectrum such that $X_k = {}^\gamma R \smashy R^{\smashy nk + n - 1}$, with face and degeneracy maps $\delta_i\colon X_k \to X_{k-1}$ and $\sigma_i\colon X_k \to X_{k+1}$ defined as follows where we group the copies of $R$ and ${}^\gamma R$ in $X_k$ as $({}^\gamma R \smashy R^{\smashy k}) \smashy (R^{\smashy k+1})^{\smashy n - 1}$:
\begin{align*}
    \delta_i & = (\id^{\smashy i} \smashy \mu \smashy \id^{\smashy k-i-1}) \smashy (\id^{\smashy i} \smashy \mu \smashy \id^{\smashy k-i-1})^{\smashy n - 1} \\
    \delta_k & = (({}^\gamma \mu \smashy \id^{\smashy k-1}) \smashy (\mu \smashy \id^{\smashy k-1})^{\smashy n-1}) \circ t \\
    \sigma_j & = (\id^{\smashy j + 1} \smashy \eta \smashy \id^{\smashy k-j}) \smashy (\id^{\smashy j + 1} \smashy \eta \smashy \id^{\smashy k-j})^{\smashy n-1}
\end{align*}
\noindent where $t$ is the map that rotates the last copy of $R$ to the front, $0 \leq i < k,$ and $0 \leq j \leq k$.

Consider $(B(R,R,R)^{\smashy_R (n-1)} \smashy_{(R \smashy R)} B(R,R,{}^\gamma R))_\bullet$, constructed as described above. Since $R$, $B_\bullet(R,R,R),$ and $B_\bullet(R,R,{}^\gamma R)$ are free $(R,R)$-bimodules, we have the following level-wise reduction map (as discussed above), this reduction map is from $(B(R,R,R)^{\smashy_R (n-1)} \smashy_{(R \smashy R)} B(R,R,{}^\gamma R))_k = B_k(R,R,R)^{\smashy_R (n-1)} \smashy_{(R \smashy R)} B_k(R,R,{}^\gamma R)$ to $X_k$:
\begin{center}
    $\psi_k\colon (R^{\smashy k+2})^{\smashy_R (n-1)} \smashy_{(R \smashy R)} (R^{\smashy k+1} \smashy {}^\gamma R) \to {}^\gamma R \smashy R^{\smashy kn + n - 1}.$
\end{center}
The face and degeneracy maps of $(B(R,R,R)^{\smashy_R (n-1)} \smashy_{(R \smashy R)} B(R,R,{}^\gamma R))_\bullet$, say $\delta_i'$ and $\sigma_i'$, on the $k^\th$ level are defined by the face and degeneracy maps on $B_k(R,R,R)$ and $B_k(R,R,{}^\gamma R)$ (see \Cref{def: double bar}). The face and degeneracy maps 
    \begin{align*}
        \delta_i' \colon & B_k(R,R,R)^{\smashy_R (n-1)} \smashy_{(R \smashy R)} B_k(R,R,{}^\gamma R) \to B_{k-1}(R,R,R)^{\smashy_R (n-1)} \smashy_{(R \smashy R)} B_{k-1}(R,R,{}^\gamma R) \quad \text{and} \\
        \sigma_i' \colon & B_k(R,R,R)^{\smashy_R (n-1)} \smashy_{(R \smashy R)} B_k(R,R,{}^\gamma R) \to B_{k+1}(R,R,R)^{\smashy_R (n-1)} \smashy_{(R \smashy R)} B_{k+1}(R,R,{}^\gamma R)
    \end{align*}
\noindent are defined as follows where we group the copies of $R$ and ${}^\gamma R$ in $B_{k}(R,R,R)^{\smashy_R (n-1)} \smashy_{(R \smashy R)} B_{k}(R,R,{}^\gamma R)$ as $(R^{\smashy k+2})^{\smashy_R (n - 1)} \smashy_{R \smashy R} (R^{\smashy k+1} \smashy {}^\gamma R)$:
\begin{align*}
    \delta_i' & = (\id^{\smashy i} \smashy \mu \smashy \id^{\smashy k-i})^{\smashy n - 1} \smashy (\id^{\smashy i} \smashy \mu \smashy \id^{\smashy k-i}) \\
    \delta_k' & = (\id^{\smashy k} \smashy \mu)^{\smashy n-1} \smashy (\id^{\smashy k} \smashy {}^\gamma \mu) \\
    \sigma_j' & = (\id^{\smashy j + 1} \smashy \eta \smashy \id^{\smashy k-j+1})^{\smashy n - 1} \smashy (\id^{\smashy j + 1} \smashy \eta \smashy \id^{\smashy k-j+1})
\end{align*}
\noindent for $0 \leq i < k,$ and $0 \leq j \leq k$.

Since the level-wise reduction map $\psi_k$ does not affect the middle $k$ copies of $R$ within each copy of $B_k(R,R,R)$ and $B_k(R,R,{}^\gamma R)$ then $\psi_k$ commutes with all degeneracy maps as well as the $i^\th$ face maps for $0 < i < k$. The other two face maps require more discussion. The right $R$-module structure of ${}^\gamma R$ is the same as the right $R$-module structure of $R$, therefore $\psi_{k-1} \circ \delta_0' = \delta_0 \circ \psi_k$. Consider the maps $\psi_{k-1} \circ \delta_k'$, and $\delta_k \circ \psi_k$. These two maps are equivalent as they multiply the same copies of $R$ and ${}^\gamma R$. Therefore $(B(R,R,R)^{\smashy_R (n-1)} \smashy_{(R \smashy R)} B(R,R,{}^\gamma R))_\bullet \cong X_\bullet$. We still need to show that $sd_n(B_\bullet^{cy, C_p}(R)) \cong X_\bullet$.

Let $d_i$ and $s_i$ be the face and degeneracy maps of $B^{cy,C_p}_\bullet(R)$ (see \Cref{twistedcyclicbar}). By definition of the simplicial edgewise subdivision, and simplicial identities, the simplicial object $sd_n(B^{cy,C_p}_\bullet(R))$ is determined by the following:
\begin{align*}
    sd_n(B^{cy,C_p}_\bullet(R))_k & = B^{cy,C_p}_{(k+1)n-1}(R) \cong ({}^\gamma R \smashy R^{k}) \smashy (R^{\smashy k + 1})^{\smashy n - 1} \\
        \bar d_i & = d_i \circ d_{i+k+1} \circ \ldots \circ d_{i + (n-1)(k+1)} \\
        & = (\id^{\smashy i} \smashy \mu \smashy \id^{\smashy k-i -1}) \smashy (\id^{\smashy i} \smashy \mu \smashy \id^{\smashy k-i -1})^{\smashy n- 1} \\
        \bar d_k & = d_k \circ d_{2k + 1} \circ \ldots \circ d_{k + (n-2)(k+1)} \circ d_{(k+1)n - 1} \\
       & = d_k \circ d_{2k + 1} \circ \ldots \circ d_{k + (n-2)(k+1)} \circ d_0 \circ \alpha  \\
       & = (({}^\gamma \mu \smashy \id^{\smashy k-1}) \smashy ( \mu \smashy \id^{\smashy k-1})^{\smashy n-1}) \circ t \\
        \bar s_j & = s_{j + (n-1)(k+2)} \circ \ldots \circ s_{j + (k+2)} \circ s_j \\
        & = (\id^{\smashy j + 1} \smashy \eta \smashy \id^{\smashy k-j}) \smashy (\id^{\smashy j + 1} \smashy \eta \smashy \id^{\smashy k-j})^{\smashy n-1}
\end{align*}
\noindent where $t$ is the map that rotates the last copy of $R$ to the front, $0 \leq i < k,$ and $0 \leq j \leq k$.

These maps are equivalent to the face and degeneracy maps of $X_\bullet$ and $X_k \cong sd_n(B_\bullet^{cy,C_p}(R))_k$. This gives the following equivalence of $C_p$-spectra: $sd_n(B_\bullet^{cy, C_p}(R)) \cong X_\bullet$. Therefore $sd_n(B^{cy,C_p}_\bullet(R)) \cong (B(R,R,R)^{\smashy_R (n-1)} \smashy_{(R \smashy R)} B(R,R,{}^\gamma R))_\bullet$. \end{proof}

We will now prove the classical bar lemma in the case for a commutative ring $C_p$-spectrum, we will also need that the geometric realization is a $C_p$-weak equivalence on each level of spectra for a later result. Here $R_\bullet$ is defined to be $R$ at every simplicial level and all face and degeneracy maps are the identity.

\begin{lemma}[Bar lemma] \label{Bar Lemma}
     Let $R$ be a commutative ring $C_p$-spectrum indexed over the trivial universe $\RR^\infty$, $p$ prime. The map of simplicial $(R,R)$-bimodules
    \[
    B_\bullet(R,R,R) \to R_\bullet
    \]
    \noindent is a simplicial $C_p$-homotopy equivalence of right $R$-modules. Applying geometric realization to this map produces the map $B(R,R,R) \to R$ which is a $C_p$-weak equivalence on each level of the spectra.
\end{lemma}

\begin{proof}
    Define the map $\eta\colon R_\bullet \to B_\bullet(R,R,R)$ as an iteration of units. Specifically, $\eta_k \colon R_k = R \to B_k(R,R,R) = R^{\smashy k+2}$ is $k + 1$ iterations of the unit map. Define the map $\epsilon\colon B_\bullet(R,R,R) \to R_\bullet$ as an iteration of products on the left-most copies of $R$. Specifically, $\epsilon_k \colon B_k(R,R,R) = R^{\smashy k+2} \to R_k = R$ is $k + 1$ iterations of the multiplication map. Then their composition $\epsilon \circ \eta\colon R_\bullet \to R_\bullet$ is the identity.
    
    Define $h_i\colon B_q(R,R,R) \to B_{q+1}(R,R,R)$ for $0 \leq i \leq q + 1$ to be $i$ iterations of the product on the left-most copies of $R$ and then $i+1$ iterations of the unit map. Therefore $h_0$ is equivalent to the identity map on $B_q(R,R,R)$ and $h_{q+1}$ is equivalent to $\eta \circ \epsilon$.

    Then $h_i$ is a simplicial $C_p$-homotopy equivalence between the identity map on $B_\bullet(R,R,R)$ and the composition $\eta \circ \epsilon$. This homotopy equivalence includes the single copy of $R$ into the right most copy of $R$ in $B_k(R,R,R)$, so this is an equivalence of right $R$-modules. 

This homotopy equivalence gives a $C_p$-weak equivalence on the geometric realization, and therefore a $C_p$-weak equivalence on each level of the spectra. \end{proof}

Hirschhorn gives a definition of Reedy cofibrant for a general Reedy category and model structure in \cite{HirschhornModelBook}. There are many classical results pertaining to Reedy cofibrant objects. One reference which goes into more detail in the equivariant case is Section 2.2 in \cite{Malkiewich-cycTHH}. We will use the equivariant model structure from \cite[Theorem 5.4.15]{EquivariantStructureSmashPowers}.

For a simplicial object $X_\bullet$ the \emph{latching object} $L_nX$ is the subspace of $X_n$ consisting of all the images of the degeneracy maps $s_i\colon X_{n-1} \to X_n$, $0 \leq i \leq n-1$, and a \emph{latching map} of $X$ is the natural map $L_nX \to X_n$ \cite[Definition 15.2.5(1)]{HirschhornModelBook}. For a simplicial $C_p$-spectrum $X_\bullet$, if every latching map $L_nX \to X_n$ is a cofibration, then $X_\bullet$ is \emph{Reedy cofibrant} \cite[Definition 15.3.3 (2)]{HirschhornModelBook}. We will use the notion of a \emph{Reedy h-cofibration} where each of the latching maps are $h$-cofibrations, meaning they satisfy the free homotopy extension property. In the chosen model structure, the forgetful map from commutative ring $C_p$-spectra to $C_p$-spectra preserves cofibrations \cite[Theorem 5.4.16]{EquivariantStructureSmashPowers}. This implies that for a cofibrant, commutative ring $C_p$-spectrum $R$, the underlying $C_p$-spectrum is also cofibrant. Further, the underlying $C_p$-spectrum is $h$-cofibrant by \cite[Theorem 5.4.18]{EquivariantStructureSmashPowers}. This tells us that the maps we will use later, namely $* \colon \Sigma^\infty* \to {}^\gamma R$, $\eta\colon \SS \to R$, and their smash products, are all $h$-cofibrations. Here $\Sigma^\infty*$ is the zero-spectrum.

We will give a sketch of why $B_\bullet^{cy,C_p}(R)$ is Reedy $h$-cofibrant, for $R$ a cofibrant commutative ring $C_p$-spectrum. The argument is similar for $sd_n(B_\bullet^{cy,C_p}(R))$ and $B_\bullet(R,R,R)$.

Consider $B_\bullet^{cy,C_p}(R)$. Recall that the degeneracy maps $s_i\colon B_{k-1}^{cy,C_p}(R) \to B_{k}^{cy,C_p}(R)$ are $s_i = \id^{\smashy i} \smashy \eta \smashy \id^{\smashy k - i - 1}$ for $0 \leq i \leq k-1$. Consider the following diagram of $h$-cofibrations:
% https://q.uiver.app/#q=WzAsNCxbMCwxLCJBIFxcc21hc2h5IFxcU1MgXFxzbWFzaHkgXFxTUyJdLFsyLDAsIkEgXFxzbWFzaHkgQSBcXHNtYXNoeSBcXFNTIl0sWzIsMiwiQSBcXHNtYXNoeSBcXFNTIFxcc21hc2h5IEEiXSxbNCwxLCJBIFxcc21hc2h5IEEgXFxzbWFzaHkgQSJdLFswLDIsIlxcaWQgXFxzbWFzaHkgXFxpZCBcXHNtYXNoeSBcXGV0YSIsMl0sWzAsMSwiXFxpZCBcXHNtYXNoeSBcXGV0YSBcXHNtYXNoeSBcXGlkIl0sWzEsMywiXFxpZCBcXHNtYXNoeSBcXGlkIFxcc21hc2h5IFxcZXRhIl0sWzIsMywiXFxpZCBcXHNtYXNoeSBcXGV0YSBcXHNtYXNoeSBcXGlkIiwyXV0=
\[\begin{tikzcd}[column sep=small, row sep=small]
	&& {R \smashy R \smashy \SS} \\
	{R \smashy \SS \smashy \SS} &&&& {R \smashy R \smashy R} \\
	&& {R \smashy \SS \smashy R}
	\arrow["{\id \smashy \id \smashy \eta}", from=1-3, to=2-5]
	\arrow["{\id \smashy \eta \smashy \id}", from=2-1, to=1-3]
	\arrow["{\id \smashy \id \smashy \eta}"', from=2-1, to=3-3]
	\arrow["{\id \smashy \eta \smashy \id}"', from=3-3, to=2-5]
\end{tikzcd}\]
\noindent The pushout product of this diagram of maps is the latching map $L_2 X \to X_2$. This latching map is a $h$-cofibration as it is the pushout product of $h$-cofibrations. The other latching maps $L_n X \to X_n$ are built similarly as the pushout product (indicated by $\boxwedge$) of cofibrant maps in the following way: $(\SS \to R)^{\boxwedge \, k} \, \boxwedge \, (\Sigma^\infty* \to R)$. Therefore $L_nX \to X_n$ is a $h$-cofibration for all $n$. Therefore $B_\bullet^{cy,C_p}(R)$ is Reedy $h$-cofibrant. The argument is similar to show that $sd_n(B_\bullet^{cy,C_p}(R))$ and $B_\bullet(R,R,R)$ are Reedy $h$-cofibrant.

The following classical result is in \cite{HirschhornModelBook} and was extended equivariantly in \cite{Malkiewich-cycTHH}.

\begin{proposition}[{\cite[Theorem 15.11.11]{HirschhornModelBook} \cite[Proposition 2.4]{Malkiewich-cycTHH}}] \label{prop: Reedy maps give equiv} Let $X_\bullet$ and $Y_\bullet$ be simplicial orthogonal $C_p$-spectra. If $X_\bullet$ is Reedy $h$-cofibrant, then $|X_\bullet|$ is $h$-cofibrant. Further, if $X_\bullet$ and $Y_\bullet$ are both Reedy $h$-cofibrant, then any map $X_\bullet \to Y_\bullet$ that is a weak equivalence on each simplicial level induces a $C_p$-equivalence $|X_\bullet| \to |Y_\bullet|$ on each level of the spectra.
\end{proposition}

For the proof of the following lemma we will be using a multi-simplicial set. For $X$ a multi-simplicial set the \emph{diagonal simplicial set} $d(X)$ on the $k^\th$ level is the $n$-simplicial set $X(k,\ldots, k)$, and can also be viewed as the composite functor
\[
\Delta^\text{op} \overset{D^{n-1}}{\longrightarrow} (\Delta^\text{op})^{\times n} \overset{X}{\longrightarrow} sSets
\]
where $D^{n-1}$ is the diagonal map and $sSets$ is the category of simplicial sets \cite[IV.1]{GoerssJardine}. A multi-simplicial set, when geometrically realized is homeomorphic to the geometric realization of its diagonal \cite{QuillenKTheoryI}.

\begin{lemma}
    Let $U$ be a complete $S^1$-universe, $R$ be a cofibrant, commutative ring $C_p$-spectrum indexed over the trivial universe $\RR^\infty$, $p$ prime. Then $|\I_{\RR^\infty}^U sd_n(B_\bullet^{cy,C_p}(R))| \to |\I_{\RR^\infty}^U B_\bullet^{cy,C_p}(R)|$ is a $C_p$-weak equivalence.
\end{lemma}

\begin{proof} 
Recall from \Cref{bar complexes isoms} that $sd_n(B_\bullet^{cy,C_p}(R)) \cong (B(R,R,R)^{\smashy_R (n-1)} \smashy_{(R \smashy R)} B(R,R,{}^\gamma R))_\bullet$ as simplicial $C_p$-spectra. This is the diagonal spectrum of the $n$-simplicial $C_p$-spectrum $B_\bullet(R,R,R)^{\smashy_R (n-1)} \smashy_{(R \smashy R)} B_\bullet(R,R,{}^\gamma R)$ therefore by \cite[pg 10]{QuillenKTheoryI} we have the following homeomorphism
\[
|B_\bullet(R,R,R)^{\smashy_R (n-1)} \smashy_{(R \smashy R)} B_\bullet(R,R,{}^\gamma R)| \cong |(B(R,R,R)^{\smashy_R (n-1)} \smashy_{(R \smashy R)} B(R,R,{}^\gamma R))_\bullet|.
\]

\noindent Since this is a multi-simplicial $C_p$-spectrum we can therefore geometrically realize one simplicial direction at a time. 

Consider the following composition of maps, first the homeomorphism discussed above, followed by $n-1$ iterations of the map from the bar lemma on the left-most component (see \Cref{Bar Lemma}) and the identity map on the rest of the components:
    \begin{align*}
        |(B(R,R,R)^{\smashy_R (n-1)} \smashy_{(R \smashy R)} B(R,R,{}^\gamma R))_\bullet| & \cong |B_\bullet(R,R,R)^{\smashy_R (n-1)} \smashy_{(R \smashy R)} B_\bullet(R,R,{}^\gamma R)|  \\
        & \cong |B_\bullet(R,R,R)| \smashy_R |B_\bullet(R,R,R)^{\smashy_R (n-2)} \smashy_{(R \smashy R)} B_\bullet(R,R,{}^\gamma R)| \\
    & \to R \smashy_R |B_\bullet(R,R,R)^{\smashy_R (n-2)} \smashy_{(R \smashy R)} B_\bullet(R,R,{}^\gamma R)| \\
    & \vdots  \\
    & \to R \smashy_{(R \smashy R)} |B_\bullet(R,R,{}^\gamma R))|.
    \end{align*}
\noindent Recall from \Cref{bar complexes isoms} that $R \smashy_{(R \smashy R)} B_\bullet(R,R,{}^\gamma R)$ is isomorphic to $B_\bullet^{cy,C_p}(R)$, therefore $|B_\bullet^{cy,C_p}(R)| \cong R \smashy_{(R \smashy R)} |B_\bullet(R,R,{}^\gamma R)|$. The first two maps are homeomorphisms which respect the $C_p$-actions, therefore they are $C_p$-weak equivalences. The maps following all apply the map from the bar lemma on the left-most component which is a $C_p$-weak equivalence on each level of the spectra, with the identity map on $B_\bullet(R,R,R)^{\smashy_R k} \smashy_{(R \smashy R)} B_\bullet(R,R,{}^\gamma R)$ for some $0 \leq k \leq n-2$. Since the identity map on $B_\bullet(R,R,R)^{\smashy_R k} \smashy_{(R \smashy R)} B_\bullet(R,R,{}^\gamma R)$ is a map of Reedy $h$-cofibrant objects which is a weak equivalence on each simplicial level, then \Cref{prop: Reedy maps give equiv} says after passing to the geometric realization of this map is a $C_p$-weak equivalence on each level of the spectra.

Therefore this composite map, $|\I_{\RR^\infty}^U sd_n(B_\bullet^{cy,C_p}(R))| \to |\I_{\RR^\infty}^U B_\bullet^{cy,C_p}(R)|$, is a $C_p$-weak equivalence.
\end{proof}

Recall that $m\THH_{C_p}(R) \cong |\I_{\RR^\infty}^U sd_m(B_\bullet^{cy,C_p}(R))|$, and whenever discussing twisted THH of a $C_n$-spectrum $R$, we mean the left derived twisted THH of $R$. In particular, this is the same as considering twisted THH of the cofibrant replacement of $R$.

An equivariant analogue to \cite[Lemma 3.8]{Angeltveit-Rognes} immediately follows.

\begin{theorem}\label{2THHtoTHH}

Let $U$ be a complete $S^1$-universe, and let $\widetilde{U} \coloneqq i^*_{C_p} U$. Let $R$ be a commutative ring $C_p$-spectrum indexed on the $C_p$-universe $\widetilde{U}$, for $p$ prime. Then there is a $C_p$-weak equivalence

\begin{center}
    $\pi_m\colon m\THH_{C_p}(R) \to \THH_{C_p}(R)$.
\end{center}

\end{theorem}

We are now ready to discuss the structure of twisted THH.

\subsection{Algebraic structure} \label{SubSec:Algebra structure}

In this section, we will show that for $R$ a commutative ring $C_p$-spectrum, $\THH_{C_p}(R)$ is a commutative $R$-algebra in the category of $C_p$-spectra for any prime $p$. The process for proving this is similar to the process that Angeltveit and Rognes use to show that for $A$ a commutative ring spectrum, $\THH(A)$ is a commutative $A$-algebra in the category of spectra \cite{Angeltveit-Rognes}. Note that Angeltveit and Rognes also show that $\THH(A)$ is an $A$-Hopf algebra in the stable homotopy category \cite{Angeltveit-Rognes}.

Recall that in \cite{Angeltveit-Rognes} the simplicial map $\eta\colon e \to S^1_\bullet$ is the inclusion of the point, which induces the unit map $\eta\colon A \to \THH(A)$ by applying the functor $A \otimes (-)$. The equivariant analogue to this simplicial map is $\eta\colon C_p \to pS^1_\bullet$ which includes the $p$ points into $pS^1_\bullet$. This induces the unit map $\eta\colon R \to \THH_{C_p}(R)$ by applying the functor $R \otimes_{C_p} (-)$. The intuition here is that we need $C_p$-equivariant analogues to the classical spaces used. First, instead of a point we require the $C_p$-orbit of a point. Second, instead of $S^1$ we need a $C_p$-equivariant model of the circle such that after applying the functor $R \otimes_{C_p} (-)$ to the circle we get $\THH_{C_p}(R)$.

\begin{example}
The inclusion map $C_3 \to 3S^1_\bullet$ can be pictured as follows:

    
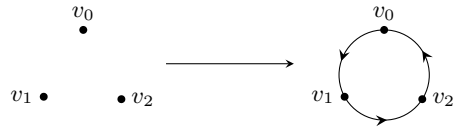
\centerline{\begin{tikzpicture}[>=stealth]
    \draw[
        decoration={markings,
        mark=at position 0.25 with {\fill circle (1.5pt) node[above] {\footnotesize$v_0$};}, 
        mark=at position 0.58 with {\fill circle (1.5pt) node[left] {\footnotesize$v_1$};}, 
        mark=at position 0.91 with {\fill circle (1.5pt) node[right] {\footnotesize$v_2$};}, 
        },
        postaction={decorate},
        draw=white
       ]
       (-1,0) circle (0.6);

\draw[->] (0.1,0.15) -- (1.8,0.15);

    \draw[ 
        decoration={markings,
        mark=at position 0.09 with {\arrow[thick]{>}},
        mark=at position 0.25 with {\fill circle (1.5pt) node[above] {\footnotesize$v_0$};}, 
        mark=at position 0.45 with {\arrow[thick]{>}}, 
        mark=at position 0.58 with {\fill circle (1.5pt) node[left] {\footnotesize$v_1$};}, 
        mark=at position 0.76 with {\arrow[thick]{>}},
        mark=at position 0.91 with {\fill circle (1.5pt) node[right] {\footnotesize$v_2$};}, 
        },
        postaction={decorate}
        ]
        (3,0) circle (0.6);
\end{tikzpicture}}
\end{example}

Let us consider the pushout of two copies of the simplicial map $\eta\colon C_p \to pS^1_\bullet$, for $p$ prime, and let us call this pushout $pS^1_\bullet \wedgey_{C_p} pS^1_\bullet.$ Define the fold map $\phi\colon pS^1_\bullet \wedgey_{C_p} pS^1_\bullet \to pS^1_\bullet$ as folding the $1$-cells together that share the same boundary. This chosen notation is meant to evoke that this is an equivariant analogue of the classical wedge.

\begin{example}
    The spaces $2S^1_\bullet \wedgey_{C_2} 2S^1_\bullet$ and $3S^1_\bullet \wedgey_{C_3} 3S^1_\bullet$ can be depicted with the following diagrams, respectively:

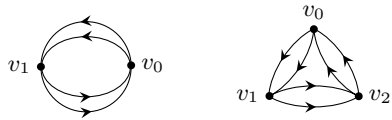
\centerline{
\begin{tabular}{ccc}
    \begin{tikzpicture}[>=stealth]
    \draw[ 
        decoration={
            markings, 
            mark=at position 0.005 with {\fill circle (1.5pt) node[right] {\footnotesize$v_0$};}, 
            mark=at position 0.27 with {\arrow[thick]{>}}, 
            mark=at position 0.5 with {\fill circle (1.5pt) node[left] {\footnotesize$v_1$};}, 
            mark=at position 0.77 with {\arrow[thick]{>}},
        },
        postaction={decorate}
    ]
    (3,0) circle (0.6);
    
    \draw[
        decoration={
            markings, 
            mark=at position 0.27 with {\arrow[thick]{>}}, 
            mark=at position 0.77 with {\arrow[thick]{>}},
        },
        postaction={decorate}
    ]
    (3,0) ellipse (.59 and 0.4);
\end{tikzpicture}
 & & 
\begin{tikzpicture}[>=stealth]
    \draw[ 
        decoration={
            markings,
            mark=at position 0.17 with {\arrow[thick]{>}},
            mark=at position 0.32 with {\fill circle (1.5pt) node[above] {\footnotesize$v_0$};}, 
            mark=at position 0.50 with {\arrow[thick]{>}}, 
            mark=at position 0.64 with {\fill circle (1.5pt) node[left] {\footnotesize$v_1$};}, 
            mark=at position 0.83 with {\arrow[thick]{>}},
            mark=at position 0.999 with {\fill circle (1.5pt) node[right] {\footnotesize$v_2$};}, 
        },
        postaction={decorate}
    ]
    (.6,-.3) to [bend left=30,looseness=0.8] (0,.6) to [bend left=30,looseness=0.8] (-.6, -.3) to [bend left=30,looseness=0.8] (.6,-.3) ;
    
    \draw[ 
        decoration={
            markings,
            mark=at position 0.17 with {\arrow[thick]{>}},
            mark=at position 0.50 with {\arrow[thick]{>}}, 
            mark=at position 0.83 with {\arrow[thick]{>}},
        },
        postaction={decorate}
    ]
    (.6,-.3) to [bend right=30,looseness=0.8] (0,.6) to [bend right=30,looseness=0.8] (-.6, -.3) to [bend right=30,looseness=0.8] (.6,-.3) ;
\end{tikzpicture}
\end{tabular}}
\noindent where the $C_2$-action on the first diagram is counter clockwise rotation by $\pi$, and the $C_3$-action on the second diagram is counter clockwise rotation by $2\pi/3$.
\end{example}

In order to show that this fold map induces the product map $\phi\colon \THH_{C_p}(R) \smashy_R \THH_{C_p}(R) \to \THH_{C_p}(R)$, we need to show $R \otimes_{C_p} (pS^1_\bullet \wedgey_{C_p} pS^1_\bullet)$ is isomorphic to $(R \otimes_{C_p} pS^1_\bullet) \smashy_R (R \otimes_{C_p} pS^1_\bullet)$ as simplicial $C_p$-spectra. This question reduces to whether the functor $R \otimes_{C_p} (-)$ (from \cref{coequalizer for THHCn}) from unbased spaces with a free $C_p$-action to $C_p$-spectra preserves pushouts.

\begin{proposition}
    Let $R$ be a commutative ring $C_p$-spectrum and consider $C_p$ as an unbased space with a free $C_p$-action, $p$ prime. The functor $R \otimes_{C_p} (-)$ from the category of unbased spaces with a free $C_p$-action to the category of commutative ring $C_p$-spectra preserves pushouts.
\end{proposition}

\begin{proof}
The tensor products $R \otimes C_p \otimes (-)$ and $R \otimes (-)$ both preserve colimits. All colimits commute with other colimits, and the coequalizer is a colimit so $R \otimes_{C_p} (-)$ commutes with all colimits and therefore preserves pushouts.
\end{proof}

Now we can say that the simplicial maps:
\begin{center}
\begin{tabular}{l}
    $\eta\colon C_p \to pS^1_\bullet$ \\
    $\phi\colon pS^1_\bullet \wedgey_{C_p} pS^1_\bullet \to pS^1_\bullet$
    \end{tabular}
\end{center}
\noindent induce the following maps of commutative ring $C_p$-spectra:
\begin{equation} \label{SimpProduct} 
    \begin{tabular}{l}
    $\eta\colon R \to \THH_{C_p}(R)$ \\
    $\phi\colon \THH_{C_p}(R) \smashy_R \THH_{C_p}(R) \to \THH_{C_p}(R)$
    \end{tabular}
\end{equation}
\noindent which are the unit and product maps respectively. We will use these maps to show that $\THH_{C_p}(R)$ is a commutative $R$-algebra.

In order to check associativity of this product map we will need to understand the pushout of the span $pS^1_\bullet \wedgey_{C_p} pS^1_\bullet \leftarrow C_p \to pS^1_\bullet$, which is $pS^1_\bullet \wedgey_{C_p} pS^1_\bullet \wedgey_{C_p} pS^1_\bullet$. This should be thought of as an equivariant analogue to the wedge of three circles.

\begin{example} \label{3waywedge}
    Consider $3S^1_\bullet \wedgey_{C_3} 3S^1_\bullet \wedgey_{C_3} 3S^1_\bullet$ which can be pictured as:

    
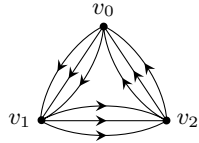
\centerline{
    \begin{tikzpicture}[>=stealth, scale=1.4]
    \draw[ 
        decoration={
            markings,
            mark=at position 0.17 with {\arrow[thick]{>}},
            mark=at position 0.32 with {\fill circle (1.5pt) node[above] {\footnotesize$v_0$};}, 
            mark=at position 0.50 with {\arrow[thick]{>}}, 
            mark=at position 0.64 with {\fill circle (1.5pt) node[left] {\footnotesize$v_1$};}, 
            mark=at position 0.83 with {\arrow[thick]{>}},
            mark=at position 0.999 with {\fill circle (1.5pt) node[right] {\footnotesize$v_2$};}, 
        },
        postaction={decorate}
    ]
    (.6,-.3) to [bend left=30,looseness=0.8] (0,.6) to [bend left=30,looseness=0.8] (-.6, -.3) to [bend left=30,looseness=0.8] (.6,-.3) ;

    \draw[ 
        decoration={
            markings,
            mark=at position 0.17 with {\arrow[thick]{>}},
            mark=at position 0.50 with {\arrow[thick]{>}}, 
            mark=at position 0.83 with {\arrow[thick]{>}},
        },
        postaction={decorate}
    ]
    (.6,-.3) to (0,.6) to (-.6, -.3) to (.6,-.3) ;

    \draw[ 
        decoration={
            markings,
            mark=at position 0.17 with {\arrow[thick]{>}},
            mark=at position 0.50 with {\arrow[thick]{>}}, 
            mark=at position 0.83 with {\arrow[thick]{>}},
        },
        postaction={decorate}
    ]
    (.6,-.3) to [bend right=30,looseness=0.8] (0,.6) to [bend right=30,looseness=0.8] (-.6, -.3) to [bend right=30,looseness=0.8] (.6,-.3) ;
\end{tikzpicture}
    }

    \noindent Here $\phi \wedgey \id\colon 3S^1_\bullet \wedgey_{C_3} 3S^1_\bullet \wedgey_{C_3} 3S^1_\bullet \to 3S^1_\bullet \wedgey_{C_3} 3S^1_\bullet$ folds the outer copy of $3S^1_\bullet$ with the middle copy of $3S^1_\bullet$ and leaves the inner copy of $3S^1_\bullet$ alone. Similarly, $\id \wedgey \phi\colon 3S^1_\bullet \wedgey_{C_3} 3S^1_\bullet \wedgey_{C_3} 3S^1_\bullet \to 3S^1_\bullet \wedgey_{C_3} 3S^1_\bullet$ folds the inner copy of $3S^1_\bullet$ with the middle copy of $3S^1_\bullet$ and leaves the outer copy of $3S^1_\bullet$ alone.
\end{example}

\begin{theorem} \label{THHAlgebra}
Let $p$ be prime. For a commutative ring $C_p$-spectrum $R$, $\THH_{C_p}(R)$ is a commutative $R$-algebra in the category of commutative ring $C_p$-spectra.
\end{theorem}

\begin{proof} We begin by checking associativity of the product map $\phi\colon \THH_{C_p}(R) \smashy_R \THH_{C_p}(R) \to \THH_{C_p}(R)$. For ease of notation, let $T \coloneqq \THH_{C_p}(R)$. We need to verify that the following diagram commutes:
% https://q.uiver.app/#q=WzAsNCxbMCwwLCJUIFxcc21hc2h5X1IgVCBcXHNtYXNoeV9SIFQiXSxbMiwwLCJUIFxcc21hc2h5X1IgVCJdLFswLDIsIlQgXFxzbWFzaHlfUiBUIl0sWzIsMiwiVCJdLFswLDEsIlxcaWQgXFxzbWFzaHkgXFxwaGkiXSxbMCwyLCJcXHBoaSBcXHNtYXNoeSBcXGlkIiwyXSxbMiwzLCJcXHBoaSIsMl0sWzEsMywiXFxwaGkiXV0=
\[\begin{tikzcd}[column sep=small,row sep=small]
	{T \smashy_R T \smashy_R T} && {T \smashy_R T} \\
	\\
	{T \smashy_R T} && T.
	\arrow["{\id \smashy \phi}", from=1-1, to=1-3]
	\arrow["{\phi \smashy \id}"', from=1-1, to=3-1]
	\arrow["\phi", from=1-3, to=3-3]
	\arrow["\phi"', from=3-1, to=3-3]
\end{tikzcd}\]
\noindent It is sufficient to show that the following diagram of $C_p$-simplicial spaces commutes:
% https://q.uiver.app/#q=WzAsNCxbMCwwLCJwU14xX1xcYnVsbGV0IFxcd2VkZ2V5X3tDX3B9IHBTXjFfXFxidWxsZXQgXFx3ZWRnZXlfe0NfcH0gcFNeMV9cXGJ1bGxldCJdLFsyLDAsInBTXjFfXFxidWxsZXQgXFx3ZWRnZXlfe0NfcH0gcFNeMV9cXGJ1bGxldCJdLFswLDIsInBTXjFfXFxidWxsZXQgXFx3ZWRnZXlfe0NfcH0gcFNeMV9cXGJ1bGxldCJdLFsyLDIsInBTXjFfXFxidWxsZXQiXSxbMCwxLCJcXGlkIFxcd2VkZ2V5IFxccGhpIl0sWzAsMiwiXFxwaGkgXFx3ZWRnZXkgXFxpZCIsMl0sWzIsMywiXFxwaGkiLDJdLFsxLDMsIlxccGhpIl1d
\[\begin{tikzcd}[column sep=small,row sep=small]
	{pS^1_\bullet \wedgey_{C_p} pS^1_\bullet \wedgey_{C_p} pS^1_\bullet} && {pS^1_\bullet \wedgey_{C_p} pS^1_\bullet} \\
	\\
	{pS^1_\bullet \wedgey_{C_p} pS^1_\bullet} && {pS^1_\bullet.}
	\arrow["{\id \wedgey \phi}", from=1-1, to=1-3]
	\arrow["{\phi \wedgey \id}"', from=1-1, to=3-1]
	\arrow["\phi", from=1-3, to=3-3]
	\arrow["\phi"', from=3-1, to=3-3]
\end{tikzcd}\]
\noindent The maps $\id \wedgey \phi$ and $\phi \wedgey \id$ fold the inner two and outer two copies of $pS^1_\bullet$ together, respectively, where $\phi$ folds the remaining two copies of $pS^1_\bullet$ together. Therefore this diagram commutes.

To check unitality and commutativity of the product map, we need to show that the following diagrams commute:
% https://q.uiver.app/#q=WzAsNyxbMCwwLCJwU14xX1xcYnVsbGV0IFxcd2VkZ2V5X3tDX3B9IENfcCJdLFsyLDAsInBTXjFfXFxidWxsZXQgXFx3ZWRnZXlfe0NfcH0gcFNeMV9cXGJ1bGxldCJdLFs0LDAsIkNfcCBcXHdlZGdleV97Q19wfSBwU14xX1xcYnVsbGV0Il0sWzIsMywicFNeMV9cXGJ1bGxldCJdLFswLDUsInBTXjFfXFxidWxsZXQgXFx3ZWRnZXlfe0NfcH0gcFNeMV9cXGJ1bGxldCJdLFs0LDUsInBTXjFfXFxidWxsZXQgXFx3ZWRnZXlfe0NfcH0gcFNeMV9cXGJ1bGxldCJdLFsyLDgsInBTXjFfXFxidWxsZXQiXSxbMCwxLCJcXGlkIFxcd2VkZ2V5IFxcZXRhIl0sWzIsMSwiXFxldGEgXFx3ZWRnZXkgXFxpZCIsMl0sWzAsMywiXFxzaW1lcSIsMl0sWzIsMywiXFxzaW1lcSJdLFsxLDMsIlxccGhpIl0sWzUsNiwiXFxwaGkiXSxbNCw2LCJcXHBoaSIsMl0sWzQsNSwiXFx0YXUiXV0=
\[\begin{tikzcd}[column sep=small,row sep=tiny]
	{pS^1_\bullet \wedgey_{C_p} C_p} && {pS^1_\bullet \wedgey_{C_p} pS^1_\bullet} && {C_p \wedgey_{C_p} pS^1_\bullet} \\
	\\
	\\
	&& {pS^1_\bullet} \\
	{pS^1_\bullet \wedgey_{C_p} pS^1_\bullet} &&&& {pS^1_\bullet \wedgey_{C_p} pS^1_\bullet} \\
	\\
	\\
	&& {pS^1_\bullet}
	\arrow["{\id \wedgey \eta}", from=1-1, to=1-3]
	\arrow["\simeq"', from=1-1, to=4-3]
	\arrow["\phi", from=1-3, to=4-3]
	\arrow["{\eta \wedgey \id}"', from=1-5, to=1-3]
	\arrow["\simeq", from=1-5, to=4-3]
	\arrow["\tau", from=5-1, to=5-5]
	\arrow["\phi"', from=5-1, to=8-3]
	\arrow["\phi", from=5-5, to=8-3]
\end{tikzcd}\]
\noindent where $pS^1_\bullet \wedgey_{C_p} C_p \cong pS^1_\bullet$ is the pushout of the span $pS^1_\bullet \leftarrow C_p \to C_p$. The map $\id \wedgey \eta\colon pS^1_\bullet \wedgey_{C_p} C_p \to pS^1_\bullet \wedgey_{C_p} pS^1_\bullet$ is the identity on $pS^1_\bullet$ and includes $C_p$ into the second copy of $pS^1_\bullet$, similarly $\eta \wedgey \id\colon C_p \wedgey_{C_p} pS^1_\bullet \to pS^1_\bullet \wedgey_{C_p} pS^1_\bullet$ is the identity on $pS^1_\bullet$ and includes $C_p$ into the first copy of $pS^1_\bullet$. The map $\tau$ swaps the first and second copies of $pS^1_\bullet$. 

For the unitality diagram, note that $\phi\colon pS^1_\bullet \wedgey_{C_p} pS^1_\bullet \to pS^1_\bullet$ is the fold map and the maps $\id \wedgey \eta$ and $\eta \wedgey \id$ both have an image of one copy of $pS^1_\bullet$, so this diagram commutes. For the commutativity diagram, the fold map has the same image no matter the position of the two copies of $pS^1_\bullet$.

Therefore, $\THH_{C_p}(R)$ is a commutative $R$-algebra.
\end{proof}

\subsection{Coproduct Structure Failure} \label{SubSec:Coproduct Structure}

Classically, for $A$ a commutative ring spectrum, $\THH(A)$ has a coassociative coproduct structure which is not in general cocommutative in the stable homotopy category \cite{Angeltveit-Rognes}. Classically one can consider the collapse map from $S^1$ to the point to produce the counit and a simplicial pinch map from a double model of the circle to a wedge of two circles to produce the coproduct map.

In our equivariant model, to get a counit, we would need a $C_p$-equivariant simplicial map from $pS^1_\bullet$ to $C_p$. This is impossible as $pS^1_\bullet$ is connected and $C_p$ is $p$ 0-cells which has no $C_p$-fixed points, therefore there is no $C_p$-equivariant simplicial way to collapse $pS^1_\bullet$ to a point in $C_p$. 

Since $2pS^1_\bullet$ is the $C_p$-equivariant analogue to the double model of the circle and $pS^1_\bullet \wedgey_{C_p} pS^1_\bullet$ is the $C_p$-equivariant analogue to the wedge of two circles, to obtain a coproduct we want a $C_p$-equivariant simplicial pinch map on $2pS^1_\bullet$ whose image is $pS^1_\bullet \wedgey_{C_p} pS^1_\bullet$. Classically, the pinch map identifies opposite 0-cells. Equivariantly, the pinch map identifies the $C_2$-orbits of 0-cells in $2pS^1_\bullet$ which were induced from applying $sd_2(-)$ to $pS^1_\bullet$. The pinch map on $2pS^1_\bullet$ has an interesting image which is not $pS^1_\bullet \wedgey_{C_p} pS^1_\bullet$, as demonstrated by the following two examples.

\begin{example} \label{p=2 problem}
    The pinch map for $p = 2$ is from $4S^1_\bullet$ to a space that looks like $2S^1_\bullet \wedgey_{C_2} 2S^1_\bullet$ but with a different $C_2$-action. The pinch map is pictured here:
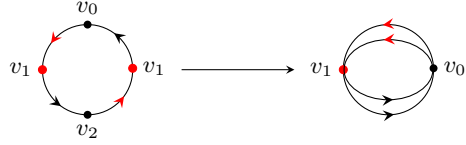
\begin{center}
\begin{tikzpicture}[>=stealth]
    \draw[ 
        decoration={markings, 
        mark=at position 0.005 with {\filldraw[fill=red,draw=red] (0,0) circle (1.5pt) node[right] {\footnotesize$v_1$};}, 
        mark=at position 0.15 with {\arrow[thick]{>}}, 
        mark=at position 0.25 with {\fill circle (1.5pt) node[above] {\footnotesize$v_0$};}, 
        mark=at position 0.40 with {\arrow[thick,draw=red]{>}}, 
        mark=at position 0.5 with {\filldraw[fill=red,draw=red] (0,0) circle (1.5pt) node[left] {\footnotesize$v_1$};}, 
        mark=at position 0.65 with {\arrow[thick]{>}}, 
        mark=at position 0.75 with {\fill circle (1.5pt) node[below] {\footnotesize$v_2$};}, 
        mark=at position 0.9 with {\arrow[thick,draw=red]{>}}, 
        },
        postaction={decorate}
        ]
        (-4,0) circle (0.6);

    \draw[ 
        decoration={
            markings, 
            mark=at position 0.005 with {\fill circle (1.5pt) node[right] {\footnotesize$v_0$};}, 
            mark=at position 0.27 with {\arrow[thick,draw=red]{>}}, 
            mark=at position 0.5 with {\filldraw[fill=red,draw=red] circle (1.5pt) node[left] {\footnotesize$v_1$};}, 
            mark=at position 0.77 with {\arrow[thick]{>}},
        },
        postaction={decorate}
    ]
    (0,0) circle (0.6);
    
    \draw[
        decoration={
            markings, 
            mark=at position 0.27 with {\arrow[thick,draw=red]{>}}, 
            mark=at position 0.77 with {\arrow[thick]{>}},
        },
        postaction={decorate}
    ]
    (0,0) ellipse (.59 and 0.4);

    \draw[->] (-2.75,0) -- (-1.25,0);
\end{tikzpicture}

\end{center}

\noindent the $C_2$-action on the image swaps the edges with red arrows with each other and the edges with black arrows with each other. Whereas, the $C_2$-action on $2S^1_\bullet \wedgey_{C_2} 2S^1_\bullet$ is counter-clockwise rotation by $\pi$.
\end{example}

\begin{example}
The pinch map for $p = 3$ is from $6S^1_\bullet$ to a space that looks like $3S^1_\bullet \wedgey_{C_3} 3S^1_\bullet$ but with a different $C_3$-action. The pinch map is pictured here:

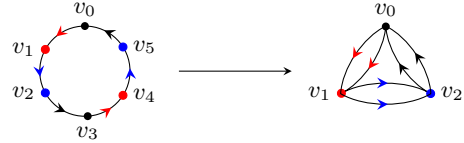
\centerline{
\begin{tikzpicture}[>=stealth]
\draw[ 
        decoration={markings, 
        mark=at position 0.01 with {\arrow[thick,draw=blue]{>}},
        mark=at position 0.09 with {\filldraw[fill=blue,draw=blue] (0,0) circle (1.5pt) node[right] {\footnotesize$v_5$};},
        mark=at position 0.18 with {\arrow[thick]{>}},
        mark=at position 0.25 with {\fill circle (1.5pt) node[above] {\footnotesize$v_0$};}, 
        mark=at position 0.36 with {\arrow[thick,draw=red]{>}},
        mark=at position 0.42 with {\filldraw[fill=red,draw=red] (0,0) circle (1.5pt) node[left] {\footnotesize$v_1$};}, 
        mark=at position 0.52 with {\arrow[thick,draw=blue]{>}},
        mark=at position 0.58 with {\filldraw[fill=blue,draw=blue] (0,0) circle (1.5pt) node[left] {\footnotesize$v_2$};}, 
        mark=at position 0.68 with {\arrow[thick]{>}},
        mark=at position 0.76 with {\fill circle (1.5pt) node[below] {\footnotesize$v_3$};},
        mark=at position 0.85 with {\arrow[thick,draw=red]{>}},
        mark=at position 0.91 with {\filldraw[fill=red,draw=red] (0,0) circle (1.5pt) node[right] {\footnotesize$v_4$};}, 
        },
        postaction={decorate}
        ]
        (-4,0) circle (0.6);

\draw[->] (-2.75,0) -- (-1.25,0);

    \draw[ 
        decoration={
            markings,
            mark=at position 0.17 with {\arrow[thick]{>}},
            mark=at position 0.32 with {\fill circle (1.5pt) node[above] {\footnotesize$v_0$};}, 
            mark=at position 0.50 with {\arrow[thick,draw=red]{>}}, 
            mark=at position 0.64 with {\filldraw[fill=red,draw=red] (0,0) circle (1.5pt) node[left] {\footnotesize$v_1$};}, 
            mark=at position 0.83 with {\arrow[thick,draw=blue]{>}},
            mark=at position 0.999 with {\filldraw[fill=blue,draw=blue] (0,0) circle (1.5pt) node[right] {\footnotesize$v_2$};}, 
        },
        postaction={decorate}
    ]
    (.6,-.3) to [bend left=30,looseness=0.8] (0,.6) to [bend left=30,looseness=0.8] (-.6, -.3) to [bend left=30,looseness=0.8] (.6,-.3) ;
    
    \draw[ 
        decoration={
            markings,
            mark=at position 0.17 with {\arrow[thick,draw]{>}},
            mark=at position 0.50 with {\arrow[thick,draw=red]{>}}, 
            mark=at position 0.83 with {\arrow[thick,draw=blue]{>}},
        },
        postaction={decorate}
    ]
    (.6,-.3) to [bend right=30,looseness=0.8] (0,.6) to [bend right=30,looseness=0.8] (-.6, -.3) to [bend right=30,looseness=0.8] (.6,-.3) ;
\end{tikzpicture}
}
\noindent the $C_3$-action on the image is counter-clockwise rotation by $4\pi/3$. Whereas, the $C_3$-action on $3S^1_\bullet \wedgey_{C_3} 3S^1_\bullet$ is counter-clockwise rotation by $2\pi/3$
\end{example}

This problem persists for larger primes, and these maps also fail to be coassociative. There may be a way to produce a counit $R \to \THH_{C_p}(R)$ or a coproduct $\THH_{C_p}(R) \to \THH_{C_p}(R) \smashy_R \THH_{C_p}(R)$ in a non-simplicial way. However, in this paper we are focused on extending these structures to the twisted B\"okstedt spectral sequence and are therefore focused on simplicial maps.

\section{Algebraic structure on the twisted B\"okstedt spectral sequence} \label{Sec: Equivariant Bokstedt SS Structure}

In this section, we will discuss the algebraic structure of the twisted B\"okstedt spectral sequence that is induced from the algebraic structure on twisted THH. Recall the twisted B\"okstedt spectral sequence defined in \cite[Theorem 4.2.7]{AGHKK} (see \Cref{EquivBokSS}). These authors show that the twisted B\"okstedt spectral sequence has an algebraic structure. The first two sentences of the following proposition are the same conditions as in the definition of the twisted B\"okstedt spectral sequence.

\begin{proposition}[{\cite[Proposition 4.2.8]{AGHKK}}]  Let $C_p$ be a finite subgroup of $S^1$ such that $C_p = \langle \gamma\rangle$. Let $A$ be an associative ring $C_p$-spectrum, and $E$ a commutative ring $C_p$-spectrum such that $\gamma$ acts trivially on $E$ such that $\underline{E}_\star(A)$ is flat over $\underline{E}_\star$. If $A$ is a commutative ring $C_n$-spectrum, then the twisted B\"okstedt spectral sequence is a spectral sequence of $RO(C_n)$-graded algebras over $\underline{E}_\star$.
\end{proposition}

In \Cref{THHAlgebra} we show that for $R$ a commutative ring $C_p$-spectrum, $\THH_{C_p}(R)$ is a commutative $R$-algebra for all primes $p$ in the category of $C_p$-spectra. The structure maps which demonstrate this structure, namely (\ref{SimpProduct}), are all induced from simplicial maps. Since they are induced from simplicial maps, they respect the skeletal filtration on $\THH_{C_p}(R)$. We will use this fact to induce structure maps on the twisted B\"okstedt spectral sequence.

Let us recall an equivariant analogue to a differential bigraded algebra. 

\begin{definition}
    Let $p$ be prime, $\underline{R}$ a commutative $C_p$-Green functor, and $\underline{M}$ a commutative $\uR$-algebra that is flat over $\underline{R}$. A \emph{differential $(\ZZ,RO(C_p))$-graded $\underline{M}$-algebra} is a collection of $(\ZZ,RO(C_p))$-graded $\underline{M}$-bimodules and a differential, $(E_{*,\star}, d),$ with the following maps of $\uM$-bimodules:
    \begin{align*}
        d\colon& E_{s,\alpha} \to E_{s-r,\alpha +r-1} \\
        \mu\colon& E_{s,\alpha} \, \square_{\underline{M}} \, E_{t,\beta} \to E_{s+t, \alpha + \beta} \\
        \eta\colon& \underline{M} \to E_{*,\star}
        \end{align*}
    \noindent for some $s,r \in \ZZ$ and $\alpha \in RO(C_p)$. The first map defines the differential, the second map is the multiplication map, and the third map is the unit map. These maps must make all the usual associativity and unitality diagrams commute. The differential $d$ must be compatible with the product map in the sense that it satisfies the Leibniz rule:
    \[
        d \circ \mu = \mu \circ (d \, \square_{\underline{M}} \, \id + (-1)^{s + \dim(\alpha^{C_p})} \id \, \square_{\underline{M}} \, d).
    \] 
\end{definition}

Now, let us recall an equivariant analogue to a spectral sequence of algebras. Note that the flatness assumption in the following definition is so that the equivariant K\"unneth spectral sequence, as defined in \cite[Theorem 1.3]{LM}, collapses to give the K\"unneth isomorphism.

\begin{definition}
    Let $p$ be prime, $\underline{R}$ a commutative $C_p$-Green functor, and $\underline{M}$ a commutative $\uR$-algebra which is flat over $\underline{R}$. A \emph{spectral sequence of $\underline{M}$-algebras} is a collection of differential $(\ZZ,RO(C_p))$-graded $\uM$-algebras $\{E^r_{*,\star}, d^r \},$ with multiplication maps $\phi_r$ such that $\phi_{r+1}$ is the composite
    \begin{center}
        $\phi_{r+1}\colon E^{r+1}_{*,\star} \, \square_{\underline{M}} \, E^{r+1}_{*,\star} \overset{\cong}{\longrightarrow} H_*(E^r_{*,\star}) \, \square_{\underline{M}} \, H_*(E^r_{*,\star}) \overset{p}{\longrightarrow}$ \\
        $H_*(E^r_{*,\star} \, \square_{\underline{M}} \, E^r_{*,\star}) \overset{H_*(\phi_r)}{\longrightarrow} H_*(E^r_{*,\star}) \overset{\cong}{\longrightarrow} E^{r+1}_{*,\star},$
    \end{center}
    \noindent where the homomorphism $p$ is induced from the homology cross product map.
\end{definition}

Using this definition and the simplicial maps from earlier, we can prove the following theorem.

\begin{theorem}\label{BokAlgebra}
    Let $p$ be prime, and let $R$ and $E$ be commutative ring $C_p$-spectra, such that the chosen generator of $C_p$ acts trivially on $E$, and $\underline{E}_\star(R)$ is flat over $\underline{E}_\star$.  The twisted B\"okstedt spectral sequence $E^r_{*,\star}$ is a spectral sequence of commutative $\underline{E}_\star(R)$-algebras.
\end{theorem}

\begin{proof} Unless otherwise specified, every box product is over $\underline{E}_\star$. Let ${}'E^r_{*,\star}$ be the spectral sequence associated to the skeleton filtration on $R \otimes_{C_p} (pS^1_\bullet \wedgey_{C_p} pS^1_\bullet) \cong \THH_{C_p}(R) \smashy_R \THH_{C_p}(R).$ By \Cref{THHAlgebra} we know that $\THH_{C_p}(R)$ is a commutative $R$-algebra. Further, all the relevant structure maps (\ref{SimpProduct}) are induced from $C_p$-equivariant simplicial maps, so they respect the skeleton filtration on $\THH_{C_p}(R)$. As a result, the ``fold" maps $\phi\colon {}'E^r_{*,\star} \to E^r_{*,\star}$ respect the differentials of the twisted B\"okstedt spectral sequence.

Consider $\THH_{C_p}(R)_\bullet$, the simplicial $C_p$-spectrum $R \otimes_{C_p} pS^1_\bullet \cong B^{cy,C_p}_\bullet(R)$. By our assumption that $\underline{E}_\star(R)$ is flat over $\underline{E}_\star$, and since $\THH_{C_p}(R)_s = R^{\smashy s + 1}$, then we have the following isomorphism:
    \begin{center}
$\underline{E}_\star(R^{\smashy (s+1)} \smashy_R R^{\smashy (s+1)}) \cong \underline{E}_\star(R)^{\square s+1} \, \square_{\underline{E}_\star(R)} \, \, \underline{E}_\star(R)^{\square s+1}.$
    \end{center}
    \noindent By definition, the left hand side is ${}'E^1_{s,\star}$ and the right hand side is $E^1_{s,\star} \, \square_{\underline{E}_\star(R)} \, \, E^1_{s,\star}$.

    The following map is induced from the homology cross product map    
    \[
    H_*(E^r_{*,\star}) \, \square_{\underline{E}_\star(R)} \, H_*(E^r_{*,\star}) \to H_*(E^r_{*,\star} \, \square_{\underline{E}_\star(R)} \, E^r_{*,\star}).
    \]
    \noindent Therefore we have a map
    \[
    E^2_{*,\star} \, \square_{\underline{E}_\star(R)} \, \, E^2_{*,\star} \cong H_*(E^1_{*,\star}) \, \square_{\underline{E}_\star(R)} \, \, H_*(E^1_{*,\star}) \to H_*(E^1_{*,\star} \, \square_{\underline{E}_\star(R)} \, \, E^1_{*,\star}) \cong H_*({}'E^1_{*,\star}) \cong {}'E^2_{*,\star}
    \]
    \noindent which, by induction, induces the following maps
    \[
    E^r_{*,\star} \, \square_{\underline{E}_\star(R)} \, \, E^r_{*,\star} \to {}'E^r_{*,\star}
    \]
    \noindent for all $r \geq 2$.
    
    Then we can define the composite map of spectral sequences
    \[
    \phi_r\colon E^r_{*,\star} \, \, \, \square_{\underline{E}_\star(R)} \, \, E^r_{*,\star} \longrightarrow {}'E^r_{*,\star} \overset{\phi}{\longrightarrow} E^r_{*,\star}
    \]
    \noindent for any $r \geq 1$. The maps $\phi_r$ respect differentials for all $r$ since $\phi$ and the product maps respect the differentials. Since all of the necessary maps respect the differential, the commutative diagrams in \Cref{THHAlgebra} induce the necessary commutative diagrams for the twisted B\"okstedt spectral sequence. Therefore $E^r_{*,\star}$ is a spectral sequence of commutative $\underline{E}_\star(R)$-algebras.
\end{proof}

\section{Further computations}\label{More Computations}

In \Cref{A Computation} we made a computation of $C_2$-twisted THH of the Real bordism spectrum with coefficients in the $C_2$-Mackey field $\uF$ such that $\uF(C_2/C_2) = \FF_2$ and $\uF(C_2/e) = 0$ (see \Cref{thm:MUR computation}). In this section, we will leverage known computations of classical THH and the simplicial structure from \Cref{bar complexes isoms} to make new computations of twisted THH. These computations include the $C_p$-geometric fixed points of twisted THH of specific $C_p$-spectra and $\underline{E}_\star(\THH_{C_p}(A))$ for $A$ a cofibrant associative ring $C_p$-spectrum and $E$ a geometric $C_p$-spectrum. We want to note that Hill alerted us to the arguments in this section.

Geometric fixed points and geometric realization commute with each other as they are both colimits. So we can use the following proposition to make new computations of $C_p$-twisted THH from already known computations of classical THH. The following result is discussed around 9.11.33 in \cite{HHRBook}.

\begin{proposition}[{\cite{HHRBook}}] \label{factor Phi}
    Let $A$ be an associative ring $G$-spectrum, and let $M$ and $N$ be right and left $A$-modules respectively. The map
    \[
\Phi^G(M \smashy_A N) \to \Phi^G(M) \smashy_{\Phi^G(A)} \Phi^G(N)
\]
\noindent is an isomorphism if $M$ and $N$ are cofibrant.
\end{proposition}

Recall that when discussing $\THH_{C_p}(-)$, we mean the left derived functor. This is the same as considering $\THH_{C_p}(A)$ to mean $\THH_{C_p}(-)$ of the cofibrant replacement of $A$.

\begin{proposition}\label{twisted to THH}
    Let $U$ be a complete $S^1$-universe, and let $\widetilde{U} \coloneqq i^*_{C_p} U$. Let $A$ be a cofibrant, associative ring $C_p$-spectrum indexed on the $C_p$-universe $\widetilde{U}$, for $p$ prime. Then the following map is an isomorphism of spectra:

\begin{center}
    $\Phi^{C_p}(\THH_{C_p}(A)) \to \THH(\Phi^{C_p}A)$.
\end{center}
\end{proposition}

\begin{proof} For simplicity, let $\widetilde{A} \coloneqq \Phi^{C_p}(A)$. Recall from \Cref{bar complexes isoms} that $B_\bullet^{cy,C_p}(A) \cong A \smashy_{(A \smashy A\op)} B_\bullet(A,A,{}^\gamma A)$, and geometric realization commutes with geometric fixed points, so we have the following isomorphism of ring spectra
\[
\Phi^{C_p}(\THH_{C_p}(A)) \cong |\Phi^{C_p}(A \smashy_{(A \smashy A\op)} B_\bullet(A,A,{}^\gamma A))|.
\]
\noindent Note that $B_\bullet(\widetilde{A},\widetilde{A},\Phi^{C_p}({}^\gamma A)) \cong B_\bullet(\widetilde{A},\widetilde{A},\widetilde{A})$ as the only difference between $A$ and ${}^\gamma A$ is the left $A$-module multiplication map, but after taking fixed points their left $\widetilde{A}$-module multiplication maps are the same. Since we assumed that $A$ is cofibrant then so is $B_\bullet(A,A,{}^\gamma A)$, therefore \Cref{factor Phi} gives us the following isomorphism
\[
|\Phi^{C_p}(A \smashy_{(A \smashy A\op)} B_\bullet(A,A,{}^\gamma A))| \to |\widetilde{A} \smashy_{\widetilde{A} \smashy \widetilde{A}\op} B_\bullet(\widetilde{A},\widetilde{A},\widetilde{A})|
\]
\noindent By \cite[IX.2.4]{EKMM} we know that $|\widetilde{A} \smashy_{\widetilde{A} \smashy \widetilde{A}\op} B_\bullet(\widetilde{A},\widetilde{A},\widetilde{A})| \cong \THH(\widetilde{A})$. \end{proof}

This corollary can be used in conjunction with the following two theorems to make new computations of twisted THH.

\begin{theorem}[{\cite[Theorem 1.3]{BCSTHH(Thom)}}] \label{THH(HZp)}
     There is a stable equivalence
    \[
    \THH(H\FF_p) \simeq H\FF_p \smashy \Omega(S^3)_+
    \]
    \noindent for each prime $p$.
\end{theorem}

\begin{theorem}
    [{\cite[Theorem 3]{BCSTHH(Thom)}}] There are stable equivalences of spectra
    \[
    \THH(MG) \simeq MG \smashy BBG_+
    \]
    \noindent for $G$ one of the stabilized Lie groups: $O, SO, Spin, U,$ or $Sp$.
\end{theorem}

Recall the result in \Cref{phiHR} allows us to say that for $\uF$ a $C_p$-Mackey field such that $\uF(C_p/C_p) = k$ and $\uF(C_p/e) = 0$, that $\Phi^{C_p}(H\uF) \cong Hk$ as a non-equivariant spectrum.

\begin{corollary}
    Let $p$ and $q$ be distinct primes. If $\uF$ is the $C_p$-Mackey field such that $\uF(C_p/C_p) = \FF_q$ and $\uF(C_p/e) = 0$, then
    \[
    \Phi^{C_p}(\THH_{C_p}(H\uF)) \simeq H\FF_q \smashy \Omega(S^3)_+.
    \]
\end{corollary}

We can also use the fact that $\Phi^{C_2}(MU_\RR) = MO$ to show the following corollary.

\begin{corollary}
    There is an equivalence of spectra
    \[
    \Phi^{C_2}(\THH_{C_2}(MU_\RR)) \simeq MO \smashy BBO_+.
    \]
\end{corollary}

\begin{remark}
   Kochman computes $H_*(BBSO;\FF_2)$ in \cite[Theorem 92]{Kochman}, in fact $H_*(BBSO;\FF_2)$ is polynomial. There is a fiber sequence $B\FF_2 \to BBSO \to BBO$ and therefore one could use the Serre spectral sequence to compute $H_*(BBO;\FF_2)$.
\end{remark}

The following proposition gives a formula for $\underline{E}_\star(\THH_{C_p}(A))$, where $E$ is a geometric $C_p$-spectrum, and $A$ is a cofibrant, associative ring $C_p$-spectrum. The above computations of the geometric fixed points of twisted THH can be plugged into this formula to make many more computations of the homology of twisted THH.

\begin{proposition}
    Let $U$ be a complete $S^1$-universe, and let $\widetilde{U} \coloneqq i^*_{C_p} U$. Let $A$ be a cofibrant, associative ring $C_p$-spectrum indexed on the $C_p$-universe $\widetilde{U}$, for $p$ prime. If $E$ is a geometric $C_p$-spectrum, then $\underline{E}_\star(\THH_{C_p}(A))(C_p/C_p) \cong \Phi^{C_p}(E)_{\dim(\star^{C_p})}(\THH(\Phi^{C_p}A))$ and $\underline{E}_\star(\THH_{C_p}(A))(C_p/e) \cong 0$.
\end{proposition}

\begin{proof}
    Since $E$ is a geometric $C_p$-spectrum, it is $C_p$-weakly equivalent to $\widetilde{E}\mathcal{P} \smashy E$. Further, since any spectrum smashed with a geometric spectrum is again geometric, then $E \smashy \THH_{C_p}(A)$ is geometric. Therefore $\underline{E}_\star(\THH_{C_p}(A)) \cong \underline{\pi}_\star(E \smashy \THH_{C_p}(A))$ is concentrated over $C_p$, namely, $\underline{E}_\star(\THH_{C_p}(A))(C_p/e) \cong 0$.
    
     Using the definition of geometric spectra, \Cref{SwitchEquivariance}, the definition of geometric fixed points, and \Cref{twisted to THH} we have the following isomorphisms:
    \begin{align*}
        \underline{E}_\star(\THH_{C_p}(A)) & \cong [S^\star,E \smashy \widetilde{E}\mathcal{P} \smashy \THH_{C_p}(A)]^{C_p} \\
        & \cong [\Phi^{C_p}(S^\star), (E \smashy \widetilde{E}\mathcal{P} \smashy \THH_{C_p}(A))^{C_p}]^e \\
        & \cong [S^{\dim(\star^{C_p})}, \Phi^{C_p}(E \smashy \THH_{C_p}(A))]^e \\
        & \cong [S^{\dim(\star^{C_p})}, \Phi^{C_p}(E) \smashy \THH(\Phi^{C_p} A)]^e.
    \end{align*}
\noindent Evaluated at $C_p/C_p$ we have the following isomorphisms:
\begin{align*}
    \underline{E}_\star(\THH_{C_p}(A))(C_p/C_p) & \cong [S^{\dim(\star^{C_p})} \smashy \Sigma^\infty (C_p/C_p)_+, \Phi^{C_p}(E) \smashy \THH(\Phi^{C_p} A)]^e \\
   & \cong [S^{\dim(\star^{C_p})}, \Phi^{C_p}(E) \smashy \THH(\Phi^{C_p} A)]^e \\
   & \cong \Phi^{C_p}(E)_{\dim(\star^{C_p})}(\THH(\Phi^{C_p} A)).
\end{align*} 
\noindent This proves the result. \end{proof}

\bibliographystyle{amsalpha}
\bibliography{bibby.bib}

\end{document}